\documentclass[12pt,a4paper,reqno]{amsart}

\usepackage{graphicx}

\usepackage{a4wide}

\usepackage{amsfonts,amssymb,amscd,amsopn,amsmath,amsthm}
\usepackage{color}
\usepackage{latexsym,euscript}
\usepackage[ansinew]{inputenc}

\numberwithin{equation}{section}

\theoremstyle{plain}
\newtheorem{maintheorem}{Theorem}
\newtheorem{maincorollary}[maintheorem]{Corollary}

\newtheorem{theorem}{Theorem}[section]
\newtheorem{proposition}[theorem]{Proposition}
\newtheorem{lemma}[theorem]{Lemma}
\newtheorem{claim}[theorem]{Claim}
\newtheorem{corollary}[theorem]{Corollary}
\theoremstyle{definition} \theoremstyle{remark}
\newtheorem{remark}[theorem]{Remark}
\newtheorem{definition}[theorem]{Definition}
\newtheorem{conjecture}{Conjecture}

\usepackage{hyperref}  

\newcommand{\ov}{\overline}
\newcommand{\un}{\underline}
\newcommand{\wt}{\widetilde}
\newcommand{\leb}{\operatorname{Leb}}
\newcommand{\Lip}{\operatorname{Lip}}

\newcommand{\dist}{\operatorname{dist}}

\newcommand{\inter}{\operatorname{int}}

\newcommand{\per}{\operatorname{Per}}

\newcommand{\supp}{\operatorname{supp}}

\renewcommand{\epsilon}{\varepsilon}

\newcommand{\vfi}{\varphi}

\def \wh {\widehat}

\def \EE {{\mathbb E}}

\def \RR {{\mathbb R}}

\def \ZZ {{\mathbb Z}}

\newcommand{\cT}{\EuScript{T}}
\newcommand{\cF}{\EuScript{F}}
\newcommand{\cE}{\EuScript{E}}
\newcommand{\F}{\EuScript{F}}
\newcommand{\cD}{\EuScript{D}}
\newcommand{\cP}{\EuScript{P}}
\newcommand{\cO}{\EuScript{O}}

\newcommand{\cC}{\EuScript{C}}
\newcommand{\cS}{\EuScript{S}}
\newcommand{\M}{\EuScript{M}}

\newcommand{\R}{\EuScript{R}}

\newcommand{\qand}{\quad\text{and}\quad}

\def \caC {{\mathcal C}}
\def \cD {{\mathcal D}}
\def \cE {{\mathcal E}}
\def \cF {{\mathcal F}}

\def \cO {{\mathcal O}}
\def \cP {{\mathcal P}}
\def \cQ {{\mathcal Q}}
\def \cR {{\mathcal R}}
\def \cS {{\mathcal S}}
\def \cT {{\mathcal T}}
\def \cU {{\mathcal U}}
\def \cV {{\mathcal V}}
\def \cW {{\mathcal W}}

\def \fX {{\mathfrak X}}

\title[Large deviations for singular-hyperbolic attracting sets]
{Upper large deviations bound for singular-hyperbolic
  attracting sets}

\date{\today}

\author{Vitor Araujo}

 \address[V.A.]{Instituto de Matem\'atica  e Estat\'{\i}stica,
  Universidade Federal da Bahia Av. Ademar de Barros s/n,
  40170-110 Salvador, Brazil.}

\email{vitor.d.araujo@ufba.br,
 www.sd.mat.ufba.br/$\sim$vitor.d.araujo}

\author{Andressa Souza}

\address[A.S.]{ Centro Multidisciplinar de Bom
  Jesus da Lapa , Universidade Federal do Oeste da Bahia,
  Av. Manuel Novais, s/n, Centro, 47600-000 Bom Jesus da
  Lapa, BA, Brazil}

\email{andressa.lima@ufob.edu.br,
  andressalima.mat@gmail.com}

\author{Edvan Trindade}

\address[E.T.]{Instituto de Matem\'atica e Estat\'{\i}stica,
  Universidade Federal da Bahia Av. Ademar de Barros s/n,
  40170-110 Salvador, Brazil.}

\email{trindade.matematica@gmail.com}

\thanks{V.A. was partially supported by CNPq-Brazil; and
  A.S. and E.T. were partially supported by CAPES-Brazil.}

\keywords{Singular-hyperbolic attracting set, large
  deviations, exponentially slow approximation, piecewise
  expanding transformation with singularities.  }

\subjclass[2010]{Primary: 37D30; Secondary:
37C40, 37C10, 37D45, 37D35, 37D25.}

\begin{document}

\begin{abstract}
  We obtain a exponential large deviation upper bound for
  continuous observables on suspension semiflows over a
  non-uniformly expanding base transformation with non-flat
  singularities and/or discontinuities, where the roof function
  defining the suspension behaves like the logarithm of the
  distance to the singular/discontinuous set of the base
  map.  To obtain this upper bound, we show that the base
  transformation exhibits exponential slow recurrence to the
  singular set. 

  The results are applied to semiflows modeling
  singular-hyperbolic attracting sets of $C^2$ vector
  fields. As corollary of the methods we obtain results on
  the existence of physical measures and their statistical
  properties for classes of piecewise $C^{1+}$ expanding
  maps of the interval with singularities and
  discontinuities. We are also able to obtain exponentially
  fast escape rates from subsets without full measure.
\end{abstract}

\maketitle
\tableofcontents

\section{Introduction}
\label{sec:introduction}

Arguably one of the most important concepts in Dynamical
Systems theory is the notion of physical (or $SRB$)
measure. We say that an invariant probability measure $\mu$
for a flow $X^t$ is \emph{physical} if the set
\[
B(\mu)=\left\{z\in M:
\lim_{t\to\infty}\frac{1}{t}\int_{0}^{t}\psi(X^{s}(z))\,ds=
\int\psi\, d\mu, \forall \psi\in C^0(M,\RR)\right\}
\]
has non-zero volume, with respect to any volume form on the
ambient compact manifold $M$. The set $B(\mu)$ is by
definition the \emph{basin} of $\mu$. It is assumed that
time averages of these orbits be observable if the flow
models a physical phenomenon.

On the existence of physical/SRB measures for uniformly
hyperbolic diffeomorphisms and flows we mention the works of
Sinai, Ruelle and Bowen
\cite{Bo75,BR75,Ru76,Ru78,Si72}. More recently, Alves,
Bonatti and Viana \cite{ABV00} obtained the existence of
physical measures for partial hyperbolic diffeomorphism and
non-uniformly expanding transformations. Many developments
along these lines for non uniformly hyperbolic systems have
been obtained; see
e.g. \cite{PS82,CT88,BeY92,BeY93,ArPa04}. Closer to our
setting, Araujo, Pacifico, Pujals and Viana \cite{APPV}
obtained physical measures for singular-hyperbolic
attractors.

It is natural to study statistical properties of physical
measures, such as the speed of convergence of the time
averages to the space average, among many other properties
which have been intensely studied recently; see
e.g. \cite{kifer90,Yo90,Yo98,BeY99,AA03,alves-luzzatto-pinheiro2005,CasVar13,MelNicol05,Mel07,FMT,HoMel,LuzzMelb13,AMV15,ArMel16}. The
main motivation behind all these results is that for chaotic
systems the family $\left\{ \psi\circ X^{t}\right\} _{t>0}$
asymptotically behaves like an i.i.d. family of random
variables.

One of the ways to quantify
this is the volume of the subset of points whose time
averages $\frac1T S_T\psi$ are away from the space average
$\mu(\psi)$ by a given amount. More precisely, fixing
$\epsilon>0$ as the error size, we consider the set
\[
B_{T}=\left\{ x\in
  M:\left|\frac{1}{T}\int_{0}^{T}\psi(X^{s}(x))ds-\int\psi\,d\mu
\right|>\epsilon\right\}
\]
and search for conditions under which the volume of this set
decays exponentially fast with $T$. That is, there are
constants $C,\xi>0$ so that
\[
Leb(B_{T})\leq Ce^{-\xi T},\mbox{ for all }T>0.
\]
The decay rate is related to the Thermodynamical Formalism,
first developed for hyperbolic diffeomorphisms by Bowen,
Ruelle and Sinai; see e.g. \cite{Bo75,BR75,Ru89,ruelle2004}.
However, in our setting Lebesgue measure is not necessarily
an invariant measure and so some tools from the
Thermodynamical Formalism are unavailable.

This work extends and corrects the proof of the first
author's result \cite{araujo2006a} of upper large deviation
estimate for the geometric Lorenz attractor to the
singular-hyperbolic attracting setting, encompassing a much
more general family of singular three-dimensional flows, not
necessarily transitive, with several singularities and with
higher dimensional stable direction.

This demanded, first, the construction of a global
Poincar\'e map as in \cite{APPV} through adapted
cross-sections to the flow obtained without assuming
transitivity; and, second, to deal with the possible
existence of finitely many distinct ergodic physical
measures whose convex linear combinations form the set $\EE$
of equilibrium states with respect to the $\log$ of the
central unstable Jacobian. This led us to adapt the strategy
of reduction of set of deviations for the flow to a set of
deviations for a one-dimensional map while still following the
general path presented in \cite{araujo2006a}.

This extension is done through a special choice of adapted
cross-sections in the definition of the global Poincar\'e
map which is shown to be always possible for
singular-hyperbolic attracting sets of $C^2$ flows.

Finally, and technically more delicate, we extend and
correct the proof of exponentially slow recurrence from
\cite[Section 6]{araujo2006a} for a piecewise expanding
one-dimensional interval map with only singular
discontinuities (with unbounded derivative) to allow
\emph{both singularities} (with unbounded derivative)
\emph{and discontinuities (with bounded derivative)} as
boundary points of the monotonicity intervals of the
one-dimensional map. In particular, this result allows us to
obtain many statistical properties of a new class of
piecewise expanding interval maps. This can be seen as an
extension of \cite{OHL2006} to H\"older-$C^1$ piecewise
expanding maps with singularities and discontinuities, and
also assuming strong interaction between them; see
Section~\ref{sec:coroll-extens} for more comments.


\subsection{The setting: singular-hyperbolicity} 
\label{sec:sing-hyperb}

We need some preliminary definitions.

From now on $M$ is a compact boundaryless $d$-dimensional
manifold; $\fX^2(M)$ is the set of $C^2$ vector
fields on $M$, endowed with the $C^2$ topology; we fix some
smooth Riemannian structure on $M$ and an induced normalized
volume form $\leb$ that we call Lebesgue measure; and we
write also $\dist$ for the induced distance on $M$.

Given $X\in{\fX}^2(M)$, we write $X^t$, $t \in \RR$
the flow induced by $X$ and, for $x \in M$ and
$[a,b]\subset \RR$ we set
$X^{[a,b]}(x)= \{X^t(x), a\leq t \leq b\}$.

For a compact invariant set $\Lambda$ for $X\in\fX^2(M)$, we
say that it is \emph{isolated} if we can find an open
neighborhood $U\supset \Lambda$ so that
$ \Lambda =\bigcap_{t\in\RR}X^t(U)$. If $U$ above also
satisfies $X^t(U)\subset U$ for $t>0$ then we say that
$\Lambda$ is an \emph{attracting set} and that $U$ is a
\emph{trapping region} for $\Lambda$. The \emph{topological
  basin} of the attracting set $\Lambda$ is
$W^s(\Lambda)= \{ x\in M : \lim_{t\to+\infty}\dist\big(
X^t(x) , \Lambda\big) =0 \}.$

The invariant set $\Lambda$ is \emph{transitive} if it
coincides with the $\omega$-limit set of a regular
$X$-orbit: $\Lambda=\omega_X(p)$ where $p\in\Lambda$ and
$X(p)\neq\vec0$.  If $\sigma\in M$ and $X(\sigma)=0$, then
$\sigma$ is called an {\em equilibrium} or
\emph{singularity}.

A point $p\in M$ is \emph{periodic} if $p$ is regular and
there exists $\tau>0$ so that $X^\tau(p)=p$; its orbit
$\cO_X(p)=X^{\RR}(p)=X^{[0,\tau]}(p)$ is a \emph{periodic orbit}.  An
invariant set of $X$ is \emph{non-trivial} if it is neither
a periodic orbit nor a singularity.

An \emph{attractor} is a transitive attracting set. An
attractor is \emph{proper} if it is not the whole manifold.

\begin{definition}
\label{d.dominado}
Let $\Lambda$ be a compact invariant  set of $X \in 
\fX^2(M)$ , $c>0$, and $0 < \lambda < 1$.
We say that $\Lambda$ has a $(c,\lambda)$-dominated splitting if
the tangent bundle over $\Lambda$ can be written as a continuous 
$DX^tt$-invariant sum of sub-bundles
$
T_\Lambda M=E^1\oplus E^2,
$ (that is, $DX^tE^i_x=E^i_{X^tx}, \forall t\in\RR, i=1,2$)
such that for every $t > 0$ and every $x \in \Lambda$, we have
\begin{equation}\label{eq.domination}
\|DX^t \mid E^1_x\| \cdot 
\|DX^{-t} \mid E^2_{X^t(x)}\| < c \, \lambda^t.
\end{equation}
\end{definition}

We say that a $X$-invariant subset $\Lambda$ of $M$ is
\emph{partially hyperbolic} if it has a
$(c,\lambda)$-dominated splitting, for some $c>0$ and
$\lambda\in(0,1)$, such that the sub-bundle $E^1=E^s$ is
uniformly contracting: for every $t > 0$ and every $x \in
\Lambda$ we have
$
\|DX^t \mid E^s_x\| < c \, \lambda^t.
$

We assume that $E^s$ has codimension $2$ so that $E^{cu}$ is
two-dimensional: $\dim E^{cu}_\Lambda=2$ and $\dim
E^s_\Lambda=d_s=d-2$.

Let now $J_t^{cu}(x)$ be the center Jacobian of $DX^t$ for
$x\in\Lambda$, that is, the absolute value of the
determinant of the linear map
$ DX^t \mid E^{cu}_x:E^{cu}_x\to E^{cu}_{X^t(x)}.  $ We say
that the sub-bundle $E^{cu}_\Lambda$ of the partially
hyperbolic invariant set $\Lambda$ is
\emph{$(c,\lambda)$-volume expanding} if
$J_t^{cu}(x)\geq c\, e^{\lambda t}$ for every $x\in \Lambda$
and $t\geq 0$, for some given $c,\lambda>0.$

\begin{definition}
\label{d.singularset}
Let $\Lambda$ be a compact invariant set of $X \in \fX^2(M)$
with singularities.  We say that $\Lambda$ is a
\emph{singular-hyperbolic set} for $X$ if all the
singularities of $\Lambda$ are hyperbolic and $\Lambda$ is
partially hyperbolic with volume expanding central
direction.
\end{definition}

Singular-hyperbolicity is an extension of the notion of
hyperbolic set, which we now recall.

\begin{definition}
\label{def:hyperbolic}
Let $\Lambda$ be a compact invariant set of
$X \in \fX^2(M)$.  We say that $\Lambda$ is a
\emph{hyperbolic set} for $X$ if it admits a continuous
$DX^t$-invariant splitting
$T_\Lambda M=E^s\oplus [X]\oplus E^u$ where $[X]=\RR\cdot X$
is the flow direction; $E^s$ is $(c,\lambda)$ contracting
and $E^u$ is $(c,\lambda)$-contracting for the inverse flow,
for some $(c,\lambda)\in\RR^+\times(0,1)$.
\end{definition}

In particular, every equilibrium point in a hyperbolic set
must be isolated in the set. The following result shows that
singular-hyperbolicity is a natural extension of notion of
hyperbolicity for singular flows.

\begin{theorem}[Hyperbolic Lemma]
  \label{thm:hyplemma}
  A compact invariant singular-hyperbolic set without
  singularities is a hyperbolic set.
\end{theorem}

\begin{proof}
  See \cite[Lemma 3]{MPP99} or \cite[Proposition
  6.2]{AraPac2010}.
\end{proof}

The most representative example of a singular-hyperbolic
attractor is the Lorenz attractor; see
e.g. \cite{Tu99,viana2000i}. Singular-hyperbolic attracting
sets form a class of attracting set sharing similar
topological/geometrical features with the Lorenz attractor.
For more on singular-hyperbolic attracting sets see
e.g. \cite{AraPac2010}.


\subsubsection{Lorenz-like singularities}
\label{sec:lorenz-like-singul}

A \emph{Lorenz-like singularity} is an equilibrium $\sigma$
of $X$ contained in a singular-hyperbolic set having index
(dimension of the stable direction) equal to $d-1$. Since we
are assuming that $\dim E^c=2$, this ensures the existence
of the $DX(\sigma)$-invariant splitting
$T_\sigma M=E^{s}_\sigma\oplus F^s_\sigma\oplus F^u_\sigma$
so that
\begin{itemize}
\item $E^c_\sigma=F^s\oplus F^u$ and $\dim F^s_\sigma=\dim
  F^u_\sigma=1$;
\item $F^u_\sigma$ uniformly expands and $F^s_\sigma$
  uniformly contracts: there exists $0<\wt{\lambda}<\lambda$
  such that
  $ \|DX^t\mid F^u_\sigma\|\ge C^{-1}\lambda^{-t}$ and
  $ \|DX^tt\mid F^s_\sigma\|\le C\wt{\lambda}^t$ for all
  $t\ge0$.
\end{itemize}

\begin{remark}
  \label{rmk:no-recur-non-Lorenz}
  \begin{enumerate}
  \item Partial hyperbolicity of $\Lambda$ implies that the
    direction $X(p)$ of the flow is contained in the
    center-unstable subbundle $E^2_p=E^{cu}_p$ at every
    point $p$ of $\Lambda$; see \cite[Lemma
    5.1]{ArArbSal}.
  \item The index of a singularity $\sigma$ in a
    singular-hyperbolic set $\Lambda$ equals either
    $\dim E^s_\Lambda$ or $1+\dim E^s_\Lambda$.  That is,
    $\sigma$ is either a hyperbolic saddle with codimension
    $2$, or a Lorenz-like singularity.
  \item \emph{If a singularity $\sigma$ in a
      singular-hyperbolic set $\Lambda$ is not Lorenz-like,
      then there is no orbit of $\Lambda$ that accumulates
      $\sigma$ in the positive time direction.}

    Indeed, in this case ($\dim W^s_\sigma=\dim M-2$ and
    $\dim W^u_\sigma=2$), if $\sigma\in\omega(z)$ for
    $z\in\Lambda\setminus\{\sigma\}$, then there exists
    $p\in
    (W^s_\sigma\setminus\{\sigma\})\cap\omega(z)\subset\Lambda$
    by the local behavior of trajectories near hyperbolic
    saddles.  By the properties of the stable manifold
    $W^s_\sigma$, we have
    $X^p(p)\in W^s_\sigma\cap\Lambda, \forall t\ge0$ and
    $T_{X^t(p)}W^s_\sigma\xrightarrow[t\to+\infty]{}
    E^s_\sigma$. Moreover, since
    $X(X^t(p))=\partial_uX^u(p)\mid_{u=t}\in
    T_{X^t(p)}W^s_\sigma$ and
    $T_\Lambda M=E^s_\Lambda\oplus E^{cu}_\Lambda$ is a
    continuous splitting, then
    $T_{X^t(p)}W^s_\sigma = E^s_{X^t(p)}, t\ge0$.

    But $X(p)\in E^{cu}_p$ by the previous item (1). This
    contradiction shows that $\sigma$ is not accumulated by
    any positive regular trajectory within $\Lambda$.
  \end{enumerate}
\end{remark}


\subsection{Statement of the results}
\label{sec:statement-results}

In \cite{araujo2006a} Araujo obtained exponential upper large
deviations decay for continuous observables on suspension
semiflows over a non-uniformly expanding base transformation
with non-flat criticalities/singularities, where the roof
function defining the suspension grows as the $\log$ of the
distance to the singular/critical set.

Here we extend this result to a more general class of base
transformations which, after constructing a global
Poincar\'e map describing the dynamics of
singular-hyperbolic attracting sets and reducing this
dynamics to that of a certain semiflow, enables us to obtain
the following.

\subsubsection{Upper bound for large deviations}
\label{sec:upper-bound-large}

\begin{maintheorem}
  \label{mthm:LD-sing-hyp-attracting}
  Let $X$ be a $C^2$ vector field on a compact manifold
  exhibiting a connected singular-hyperbolic attracting set
  on the trapping region $U$ having at least one Lorenz-like
  singularity. Let $\psi:U\to\RR$ be a bounded continuous
  function. Then
  \begin{enumerate}
  \item the set of equilibrium states with respect to the
    central Jacobian
    $\EE=\{\mu\in \M_X(U): h_\mu(X^1)=\int\log\big|\det (DX^1\mid
    E^{cu})\big|\,d\mu\}$ is nonempty and contains finitely
    many ergodic probability measures, where $\M_X$ is the
    family of all $X$-invariant probability measures
    supported in $U$; and
  \item for every $\epsilon>0$
    \[
      \limsup_{T\rightarrow\infty}\frac{1}{T}\log Leb\left\{
        z\in U:
        \inf_{\mu\in\EE}\left|\frac{1}{T}\int_{0}^{T}\psi(X^{t}(z))dt
          -\mu(\psi)\right|>\epsilon\right\} <0.
    \]
  \end{enumerate}
\end{maintheorem}

\begin{remark}
  \label{rmk:basicsets}
  If the singular-hyperbolic attracting set $\Lambda$ in $U$
  has several connected components (recall that an
  attracting set might not have a dense trajectory), then
  each is an attracting set and
  Theorem~\ref{mthm:LD-sing-hyp-attracting} applies to each
  singular component. Those components which have no
  singularities, or only non-Lorenz-like equilibria, are
  necessarily (by Remark~\ref{rmk:no-recur-non-Lorenz})
  hyperbolic basic sets to which we can apply known large
  deviations results~\cite{Yo90,wadd96}.
\end{remark}

From this result it is easy to deduce escape rates from
subsets of the attracting set.  Fix $K\subset U$ a compact
subset. Given $\epsilon>0$ we can find an open subset
$W\supset K$ contained in $U$ and a smooth bump function
$\varphi:U\to\RR$ so that $\leb(W\setminus K)<\epsilon$;
$0\leq\varphi\leq1$; $\varphi|_{K}\equiv1$ and
$\varphi|_{M\setminus W}\equiv0$.  Then
\[
\left\{ z\in K:X^{t}(z)\in K,0<t<T\right\} \subset\left\{ z\in\mathcal{U}:\frac{1}{T}\int_{0}^{T}\varphi(X^{t}(z))dt\geq1\right\} ,
\]
for all $T>0$, and using
Theorem~\ref{mthm:LD-sing-hyp-attracting} we deduce the
following.

\begin{maincorollary}
  \label{mcor:escaperate}
  In the same setting of
  Theorem~\ref{mthm:LD-sing-hyp-attracting}, let $K$ be a
  compact subset of $M$ such that
  $\sup_{\mu\in\mathbb{E}}\mu(K)<1$. Then
\[
\limsup_{T\to+\infty}
\frac1T \log \leb\Big( \left\{ x\in K :
  X^t(x)\in K, 0<t<T \right\} \Big) < 0.
\]
\end{maincorollary}

\subsubsection{Exponentially slow recurrence to singular
  set}
\label{sec:exponent-slow-recurr}

In Section \ref{sec:general-setting} we show that each
singular-hyperbolic attracting set $\Lambda$ for
$X\in\fX^2(M)$ admits a finite family $\Xi$ of Poincar\'e
sections to $X$ and a global Poincar\'e map
$R:\Xi_0\to\Xi, R(x)=X^{\tau(x)}(x)$ for a Poincar\'e return
time function $\tau:\Xi_0\to\RR^+$, where
$\Xi_0=\Xi\setminus\Gamma$ and $\Gamma$ is a finite family
of smooth hypersurfaces within $\Xi$.

Moreover, by a proper choice of coordinates in $\Xi$ the map
$R$ can be written
$F:(I\setminus\cD)\times B^{d_s} \to I\times B^{d_s},
F(x,y)=(f(x),g(x,y))$ where
\begin{itemize}
\item $f:I\setminus\cD\to I$ is uniformly expanding and
  piecewise $C^{1+\alpha}$, for some $\alpha>0$, on the
  connected components of $I\setminus\cD$, where $\cD$ is a
  finite subset of $I=[0,1]$ and $B^{d_s}$ denotes the
  $d_s$-dimensional open unit disk endowed with the
  Euclidean distance induced by the Euclidean norm
  $\|\cdot\|_2$; and 
\item $g:(I\setminus\cD)\times B^{d_s}\to B^{d_s}$ is a
  contraction in the second coordinate; see
  Theorem~\ref{thm:propert-singhyp-attractor}.
\end{itemize}

The main technical result in this work is that a map with
the properties of the above $f$
has exponentially slow recurrence to
the set $\cD$, as follows.

Let $f$ be a piecewise $C^{1+\alpha}$ map of the interval
$I$ for a given fixed $\alpha\in(0,1)$, that is, there
exists a finite subset $\caC$ such that \( f \) is (locally)
\( C^{1+\alpha} \) and monotone on each connected component
of \( I\setminus\caC\) and admits a continuous extension to
the boundary so that
\( f(c^\pm) = \lim_{t\to c^{\pm}} f(t)\) exists for each
$c\in\caC$. We denote by \(\cD\) the set of all ``one-sided
critical points'' $c^+$ and $c^-$ and define corresponding
one-sided neighborhoods
\begin{align}\label{eq:1sided}
\Delta(c^{+},\delta) = (c^+,c^+ +\delta) \quad\text{and}\quad
\Delta(c^{-},\delta) = (c^- -\delta,c^-),
\end{align}
for each small enough $\delta>0$. For simplicity, from now
on we use $c$ to represent the generic element of \( \cD\).
We assume that each \( c\in\cD\) has a well-defined
(one-sided) critical order $0<\alpha(c)\le1$ in the sense
that
\begin{align}\label{eq:der0}
  |f(t)-f(c)|
  \approx |t-c|^{\alpha(c)} \qand 
  |f'(t)|\approx|t-c|^{\alpha(c)-1}, \quad
  t\in\Delta(c,\delta)
\end{align}
for some $\delta>0$, where we write \( f\approx g \) if the
ratio \( f/g \) is bounded above and below uniformly in the
stated domain.

We set $\cS=\{c\in\cD: 0<\alpha(c)<1\}$ in what follows.
Moreover, note that there exists a global constant
$H=H(\alpha,f)>0$ such that for $c\in\cD\setminus\cS$,
i.e. $\alpha(c)=1$, we also have
\begin{align}
  \label{eq:derD-S}
|f'(t)-f'(s)|\le H|t-s|^\alpha, \quad t,s\in\Delta(c,\delta).
\end{align}
Let us write
$\Delta_{\delta}(x)=|\log d_{\delta}(x,\cD)|$ for the
\emph{smooth} $\delta$\emph{-truncated distance} of $x$ to
$\cD$ on $I$, that is, for any given $\delta > 0$ we set
\begin{align*}
  d_{\delta}(x,\cD)=\left\{
  \begin{array}{lll}
d(x,\cD) & \textrm{if $0<d(x,\cD)\leq\delta$;}
\\
\left(\frac{1-\delta}{\delta}\right)d(x,\cD)+2\delta-1 
& \textrm{if $\delta<d(x,\cD)<2\delta$;}
\\
1 & \textrm{if $d(x,\cD)\geq2\delta$.}
\end{array} \right.,
\end{align*}
where $d(x,y)=|x-y|, x,y\in I$ is the usual Euclidean
distance on $I$.

\begin{maintheorem}
  \label{mthm:expslowapprox}
  Let $\alpha>0$ and $f$ be a piecewise $C^{1+\alpha}$
  one-dimensional map, with monotonous branches on the
  connected components of $I\setminus\cD$ so that
  $\inf\{ |f'x|: x\in I\setminus\cD\}>1$. In addition,
  assume that $\cS\neq\emptyset$ and that


  \begin{description}
  \item[\emph{discontinuities visit singularities}]
    $\exists T_0\in\ZZ^+\, \forall c\in\cD\setminus\cS\,
    \exists T=T(c)\leq T_0$ so that $f^T(c)\in\cS$ and
    $\forall 0<j<T \exists c_j\in\cD\setminus\cS:
    d(f^j(c),\cS)=d(f^j(c),c_j)$, and we can find
    $\epsilon, \delta>0$ such that
    $f\mid_{\Delta(c,\delta)}$ is a diffeomorphism into
    $\Delta(f^T(c),\epsilon)$.
  \end{description}
  Then $f$ has \emph{exponentially slow recurrence to the
    singular/discontinuous subset $\cD$}, that is, for each
  $\epsilon>0$ we can find $\delta>0$ so that there exists
  $\xi>0$ satisfying
  \begin{align}\label{eq.expslowrecurrence}
    \limsup_{n\to\infty}\frac{1}{n}\log\lambda
    \left\{ t\in I: \frac1n \sum_{j=0}^{n-1}
      \Delta_{\delta}(f^j(t))>\epsilon\right\} <-\xi
  \end{align}
  where $\lambda$ is Lebesgue measure on $I$.
\end{maintheorem}

The proof of Theorem C provides another result on escape
rates in the setting of piecewise expanding maps with
small holes, in our case restricted to neighborhoods of the
singular/discontinuity subset.

\begin{maincorollary}
  \label{mcor:expmapsholes}
  Let $f$ be a piecewise $C^{1+\alpha}$
  one-dimensional map in the setting of
  Theorem~\ref{mthm:expslowapprox}. Then there exist
  $\delta_0>0$ and $\epsilon_0>0$ such that, for every
  $0<\delta<\delta_0$ we have
  \begin{align*}
    \limsup_{n\to\infty}\frac1n\log\lambda\left\{
    x\in I: d(f^i(x),\cD)>\delta, \forall 0<j<n
    \right\}<-\epsilon_0.
  \end{align*}
\end{maincorollary}


\subsection{Comments, corollaries and possible
  extensions}
\label{sec:coroll-extens}

The construction of adapted cross-sections for general
singular-hyperbolic attracting sets provides an extension of
the results of \cite{APPV} in line with the work of
\cite{Sataev2010,sataev2009}.

From the representation of the global Poincar\'e map as a
skew-product given by
Theorem~\ref{thm:propert-singhyp-attractor}, we can follow
\cite[Sections 6-8]{APPV} to obtain
\begin{theorem}
  \label{thm:appv-attracting}
  Let $\Lambda$ be a singular-hyperbolic attracting set for
  a $C^2$ vector field $X$ with the open subset $U$ as
  trapping region. Then
  \begin{enumerate}
  \item there are finitely many ergodic physical/SRB
    measures $\mu_1,\dots,\mu_k$ supported in $\Lambda$ such
    that the union of their ergodic basins covers $U$
    Lebesgue almost everywhere:
    \begin{align}\label{eq:appv-attracting}
      \leb\left(U\setminus\big(\cup_{i=1}^k
      B(\mu_i)\big)\right)=0.
    \end{align}
  \item Moreover, for each $X$-invariant ergodic probability
    measure $\mu$ supported in $\Lambda$ the following are
    equivalent
    \begin{enumerate}
    \item
      $h_\mu(X^1)=\int\log|\det DX^1\mid_{E^{cu}}|\,d\mu>0$;
    \item $\mu$ is a $SRB$ measure, that is, admits an
      absolutely continuous disintegration along unstable
      manifolds;
    \item $\mu$ is a physical measure, i.e., its basin
      $B(\mu)$ has positive Lebesgue measure.
    \end{enumerate}
  \item In addition, the family $\EE$ of all $X$-invariant
    probability measures which satisfy item (2a) above is the
    convex hull
    $ \EE=\{\sum_{i=1}^k t_i \mu_i : \sum_i t_i=1; 0\le
    t_i\le1, i=1,\dots,k\}.  $
  \end{enumerate}
\end{theorem}

The proof of Theorem~\ref{thm:appv-attracting} ,
characterizing physical/SRB measures and the set $\EE$ of
equilibrium states for the logarithm of the central
Jacobian, in the same way as for hyperbolic attracting sets,
is presented in
Subsection~\ref{sec:equival-between-srbp}. This result
proves item (1) of
Theorem~\ref{mthm:LD-sing-hyp-attracting}.

We note that there are many examples of singular-hyperbolic
attracting sets, non-transitive and containing
non-Lorenz-like singularities; see
Figure~\ref{fig:singhypattracting} for an example obtained
by conveniently modifying the geometric Lorenz construction,
and many others in \cite{Morales07}.

\begin{figure}[h]
\centering
\includegraphics[width=9cm]{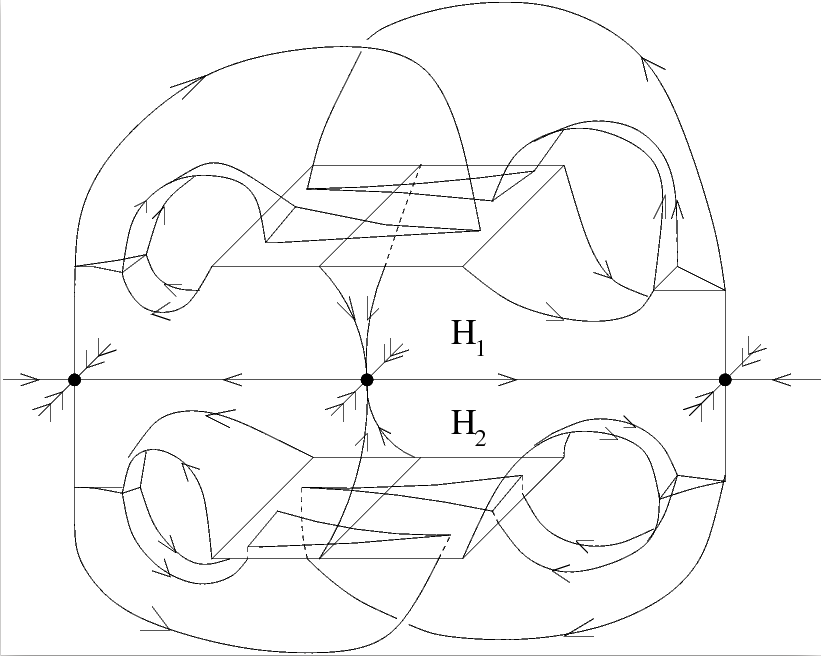}
\caption{\label{fig:singhypattracting}Example of a
  singular-hyperbolic attracting set, non-transitive and
  containing non-Lorenz like singularities.}
\end{figure}

In addition, recent results obtained in
\cite{ArGalPac,ArGalPac16} depend on the skew-product
representation of a global Poincar\'e map given by
Theorem~\ref{thm:propert-singhyp-attractor} (corresponding
to \cite[Theorem 5]{ArGalPac} ensuring the application of
\cite[Theorem A and Proposition 1]{ArGalPac} to
singular-hyperbolic attractors), which now holds without
assuming transitivity or that all equilibria are
non-resonant Lorenz-like singularities for $3$-dimensional
vector fields only.

Hence, exponential decay of
correlations for the physical measures of the Global
Poincar\'e map together with exact dimensionality and the
logarithm law for hitting times for the physical/SRB
measures of the flow on $\Lambda$ \cite[Corollaries 1 and
2]{ArGalPac} are true in the same setting of
Theorem~\ref{mthm:LD-sing-hyp-attracting}.


\subsubsection{Consequences for one-dimensional maps}
\label{sec:existence-acip-one}

In the statement of Theorem~\ref{mthm:expslowapprox} we
assume that
\begin{itemize}
\item each discontinuity point with finite lateral
  derivative (in $\cD\setminus\cS$) admits a one-sided
  neighborhood which is sent to a one-sided neighborhood of a
  singular point (in $\cS$, a discontinuity point with
  unbounded lateral derivative) in a finite and uniformly
  bounded number of iterates;
\item the set $\cS$ is non-degenerate in the usual sense
  from one-dimensional dynamics \cite{MS93} as in assumption
  \eqref{eq:der0}, that is, $|f'|$ grows as a power of the
  distance to $\cS$ (see also e.g. the conditions on the
  critical/singular set in \cite{OHL2006} for a similar
  statement in the $C^2$ setting).
\end{itemize}
Near singular points the rate of expansion is proportional
to a power of the distance to the singularity, which
allows distortion control.  The coexistence of
singularities and discontinuity points in the same map makes
it more difficult to control distortion near the boundaries
of the monotonicity intervals. The assumption that each
point in $\cD\setminus\cS$ is eventually sent in $\cS$
enables us to adapt the combinatorial method of proof from
\cite[Section 6]{araujo2006a} to this setting, which uses
partition refinement techniques first developed in the works
of Benedicks and Carleson \cite{BC85,BC91} later expanded in
\cite{MV93,PRV98,LV00,ArPa04}.

Similar techniques were used by Freitas \cite{freitas}
applied to the quadratic family to obtain exponentially slow
approximation to the critical point on Benedicks-Carleson
parameters; and by Diaz-Ordaz, Holland and Luzzatto
\cite{OHL2006} to study one-dimensional $C^2$ maps with
critical points or singularities.

In contrast to these works, where only one or finitely many
criticalities and/or singularities were allowed and with no
interaction between them, here we deal with a
\emph{H\"older-$C^1$ map having 
  non-degenerate singularities and criticalities and assume
  a strong interaction between them}.

In most works seeking the construction of a Gibbs-Markov
tower for non-uniformly expanding maps, asymptotic
conditions on the recurrence to an exceptional subset (of
criticalities, discontinuities or singularities) are assumed
providing existence of hyperbolic times, which are endowed
with automatic recurrence control to the
singular/discontinuity set. This control however does not
ensure specific rates of decay and other finer statistical
properties, which are often stated conditionally on the
class of decay rates (polynomial, super-polynomial,
stretched exponential, exponential etc);
e.g.~\cite{alves-luzzatto-pinheiro2004,AFLV11,LuzzMelb13,gouezel}.

Here we want to deduce this asymptotic recurrence control
from weaker assumptions on $f$, in a similar vein
as~\cite{BC85,BC91,MV93,PRV98,LV00} but in a much simpler
setting, which then provides hyperbolic times and many
strong statistical properties, as explained below.

However in~\cite{araujo2006a} a problem with the bounded
distortion argument was unfortunately overlooked: the
derivative $f'$ of the map $f$ (even in the case of the
Lorenz map from the geometric Lorenz attractor) is not
H\"older continuous on a whole one-sided neighborhood of the
singularities, since $f'$ is unbounded there. In fact, for
$x,y\in\Delta(c,\delta_c)$ near a singularity $c\in\cS$ we
have $|f'x-f'y|/|f'x|\approx |x-y|/|x-c|$
from~\eqref{eq:der0}, so to be able to bound $|f'x-f'y|$ we
need $|x-y|\approx|x-c|^\theta$ for some convenient
$0<\theta<1$ and $x,y$ in the same atom of a convenient
partition. \emph{But this is not provided by any exponential
  partition around the points of $\cD$, even with polynomial
  refinement} (since in this case
$|x-y|/|x-c|\approx1/|\log|x-c||$) as used in all the
works~\cite{BC85,BC91,MV93,PRV98,LV00} (and many others)
related to the Benedicks-Carleson refinement technique.

\emph{We overcome this issue by changing the way the initial
  partition is chosen}: the length of its elements
(intervals) is comparable to a suitable power of the
distance to $\cD$ in order to ensure bounded distortion.
However, we do not change the refinement strategy with
respect to~\cite{araujo2006a}, but the finer details have
been thoroughly presented in
Section~\ref{sec:exponent-slow-recurr-1}.  To the best of
the authors' knowledge, this is the first time
non-exponential initial partitions are used in this setting
and still provide upper exponential decay for the Lebesgue
measure of the deviations subset.

 Applications of this result are given by
 singular-hyperbolic attracting sets as in
 Theorem~\ref{mthm:LD-sing-hyp-attracting}, where we reduce
 the analysis to a one-dimensional map with a finite
 singular/discontinuity subset $\cD$; see
 Section~\ref{sec:reduct-large-deviat}. Coupling with
 well-known results on non-uniformly expanding maps we
 obtain results on existence of absolutely continuous
 invariant probability measures and its statistical
 properties.

We say that $f$ is \emph{non-uniformly expanding} if there
exists $c>0$ such that
\begin{align*}
  \liminf_{n\to+\infty}\frac1n
  \sum_{j=0}^{n-1}\log|f'(f^jx)|\ge c
  \quad{for}\quad\lambda-\text{a.e.} x.
\end{align*}
This condition implies in particular that the lower Lyapunov
exponent of the map $f$ is strictly positive Lebesgue almost
everywhere.

Condition \eqref{eq.expslowrecurrence} implies that
$S_n\Delta_\delta/n\to0$ in measure, i.e., the map $f$ has
\emph{slow recurrence to $\cD$}: for every $\epsilon>0$
there exists $\delta>0$ such that
\begin{align*}
  \limsup_{n\to\infty} \frac1n S_n\Delta_\delta(x)\le
  \epsilon
  \quad{for}\quad\lambda-\text{a.e.} x,
\end{align*}
where from now on
$S_ng(z)=S_n^fg(z)=\sum_{i=0}^{n-1}g(f^iz)$ denotes the
ergodic sum of the function $g:Q\to\RR$ with respect to $f$
in $n\ge1$ iterates.

We note that from properties~\eqref{eq:der0}
and~\eqref{eq:derD-S} together show that the singular set
$\cD$ is non-flat/non-degenerate similarly to the
assumptions on \cite{ABV00}

These notions suitably generalized to an arbitrary
dimensional setting were presented in~\cite{ABV00} and in
\cite{ABV00,Ze03} the following result on existence of
finitely many absolutely continuous invariant probability
measures was obtained.
\begin{theorem}[\cite{ABV00,Ze03}]
  \label{thm:abv}
  Let $f:M\circlearrowleft$ be a $C^2$ local diffeomorphism outside a
  non-degenerate singular set $\cD$. Assume that $f$ is
  non-uniformly expanding with slow recurrence to $\cD$.
  Then there are finitely many ergodic absolutely continuous
  (in particular \emph{physical} or
  \emph{Sinai-Ruelle-Bowen}) $f$-invariant probability
  measures $\mu_1,\dots,\mu_k$ whose basins cover the
  manifold Lebesgue almost everywhere, that is
  $ B(\mu_1)\cup\dots\cup B(\mu_k) = M,\quad \leb-\bmod0$.
  Moreover the support of each measure contains an open disk
  in $M$.
\end{theorem}
Clearly, $f$ in the setting of
Theorem~\ref{mthm:expslowapprox} satisfies both the
non-uniformly expanding and slow recurrence
conditions. Moreover, considering the tail sets
$\cE(x)=\min\{ N\ge1: |(f^n)'(x)|>\sigma^{n/3}, \forall n\ge
N\}$ and
$\cR(x)=\min\{ N\ge1 : S_n\Delta_\delta<\zeta n, \forall
n\ge N\}$, the exponentially slow recurrence
\eqref{eq.expslowrecurrence} can be translated as: there are
constants $\delta,\zeta,C_1,\xi>0$ so that
\begin{align*}
\lambda\Big(
  \{x\in I: \cR(x)>n \} \Big)
\le C_1\cdot e^{-\xi_1\cdot n}\quad\text{for all}\quad n\ge1;
\end{align*}
and the uniform expanding assumption on $f$ means that there
exist $\sigma>1$ and $N\in\ZZ^+$ so that $\{x\in
I:\cE(x)>n\}$ equals $I$ except finitely many points, for
all $n>N$. Hence we have
\begin{align}
  \label{eq:exptail}
  \lambda\Big(
  \{x\in I: \cR(x)>n \qand \cE(x)>n \} \Big)
\le C_1\cdot e^{-\xi_1\cdot n}\quad\text{for all}\quad n\ge1.
\end{align}
This allows us to deduce the following ergodic/statistical
properties of $f$.

\begin{maincorollary}
  \label{thm:alp}
  Let $f$ be as in the statement of
  Theorem~\ref{mthm:expslowapprox}. Then
  \begin{enumerate}
  \item (\cite{ABV00} and \cite[Theorem
    3]{alves-luzzatto-pinheiro2004}) there are finitely many
    absolutely continuous invariant probability measures
    $\mu_1,\dots,\mu_k$ such that
    $ B(\mu_1)\cup\dots\cup B(\mu_k) =
    M,\quad\lambda-\bmod0$
    and some finite power of $f$ is mixing with respect to
    $\mu_i, i=1,\dots,k$;
  \item \cite{gouezel} there exists an interval $Y_i$ with a
    return time function $R_i:Y_i\to\ZZ^+$ defining a Markov
    Tower over $f$ so that
    $\limsup\frac1n\log\mu_i\{R_i>n\}<0$ for each
    $i=1,\dots,k$;
  \item (\cite{alves-luzzatto-pinheiro2004} and
    \cite[Theorem 1.1]{gouezel}) there exist constants
    $C,c>0$ such that the correlation function
    $ \mathrm{Corr}_{n}(\vfi, \psi) = \left|\int (\varphi
      \circ g^{n})\cdot \psi \, d\mu_i - \int \varphi \,
      d\mu_i \int \psi \, d\mu_i\right|, $ for H\"older
    continuous observables $\varphi,\psi:I\to\RR$, satisfy
    $ \mathrm{Corr}_{n}(\varphi, \psi)\le C \cdot e^{-c\cdot
      n} $ for all $n\ge1$ and each $i=1,\dots,k$;
\item \cite[Theorem 4]{alves-luzzatto-pinheiro2004} $\mu_i$
  satisfies the Central Limit Theorem: given a H\"older
  continuous function \( \phi:I\to\RR \) which is not a
  coboundary (\( \phi\neq \psi\circ f - \psi \) for any continuous
  \( \psi:I\to\RR \)) there exists \( \theta>0 \) such that
  for every interval \( J\subset \RR \) and each
  $i=1,\dots,k$
\[
\lim_{n\to\infty}
\mu_i \Big(\Big\{x\in I: 
\frac{1}{\sqrt n}\sum_{j=0}^{n-1}\Big(\phi(f^{j}(x))-\int\phi d\mu_i
\Big)\in J \Big\} \Big)=
\frac{1}{\theta \sqrt{2\pi} }\int_{J} e^{-t^{2}/ 2\theta^{2}}dt.
\]
\item \cite{MelNicol05,gouezel2010} For each $i=1,\dots,k$
  let $\phi:I\to\RR$ be a H\"older observable so that
  $\mu_i(\phi)=0$. Then $\phi$ satisfies the Almost Sure
  Invariance Principle: there exist $\epsilon>0$, a sequence
  $S_N$ of random variables and a Brownian motion $W$ with
  variance $\sigma^2>0$ such that
  $\{\sum_{j=0}^{N-1}\phi\circ f^j\}=_d\{S_N\}$ and
  \begin{align*}
    S_N=W(N)+O(N^{\frac12-\epsilon})\quad\text{as}\quad
    N\to\infty\quad\mu_i-\text{almost everywhere}.
  \end{align*}
  \end{enumerate}
\end{maincorollary}

\begin{remark}\label{rmk:LIL, FCLT}
  The Almost Sure Invariance Principle implies the Central
  Limit Theorem and also the functional CLT (weak invariance
  principle), and the law of the iterated logarithm together
  with its functional version, and many other results; see e.g.
  \cite{PhilippStout75} for a comprehensive list.
\end{remark}

\subsubsection{Possible extensions and conjectures}
\label{sec:possible-general}

A natural issue is whether is it possible to remove the
assumption that $\cS$ is nonempty or to relax the assumption
that there are only finitely many discontinuity points all
of which are sent to singular points in a uniformly bounded
number of iterates.

\subsubsection*{Example with countable infinite $\cD$}
\label{sec:example-with-countab}

An example of a transformation with infinitely many
monotonicity domains and $\cS$ a single point, satisfying
the conditions of Theorem~\ref{mthm:expslowapprox} except
that $\cD$ is countably infinite, is given by a
topologically exact Lorenz transformation in the interval
$[f(0^+),f(0^-)]$ strictly contained in $J=[-1/2,1/2]$ whose
graph we complete as a function $J\to J$ with affine pieces
between points having the same values of $f$ at some element
of the preorbits of the unique singularity at $0$; see
Figure~\ref{fig:Lorinfty}.

\begin{figure}[h]
\centering
\includegraphics[scale=0.5]{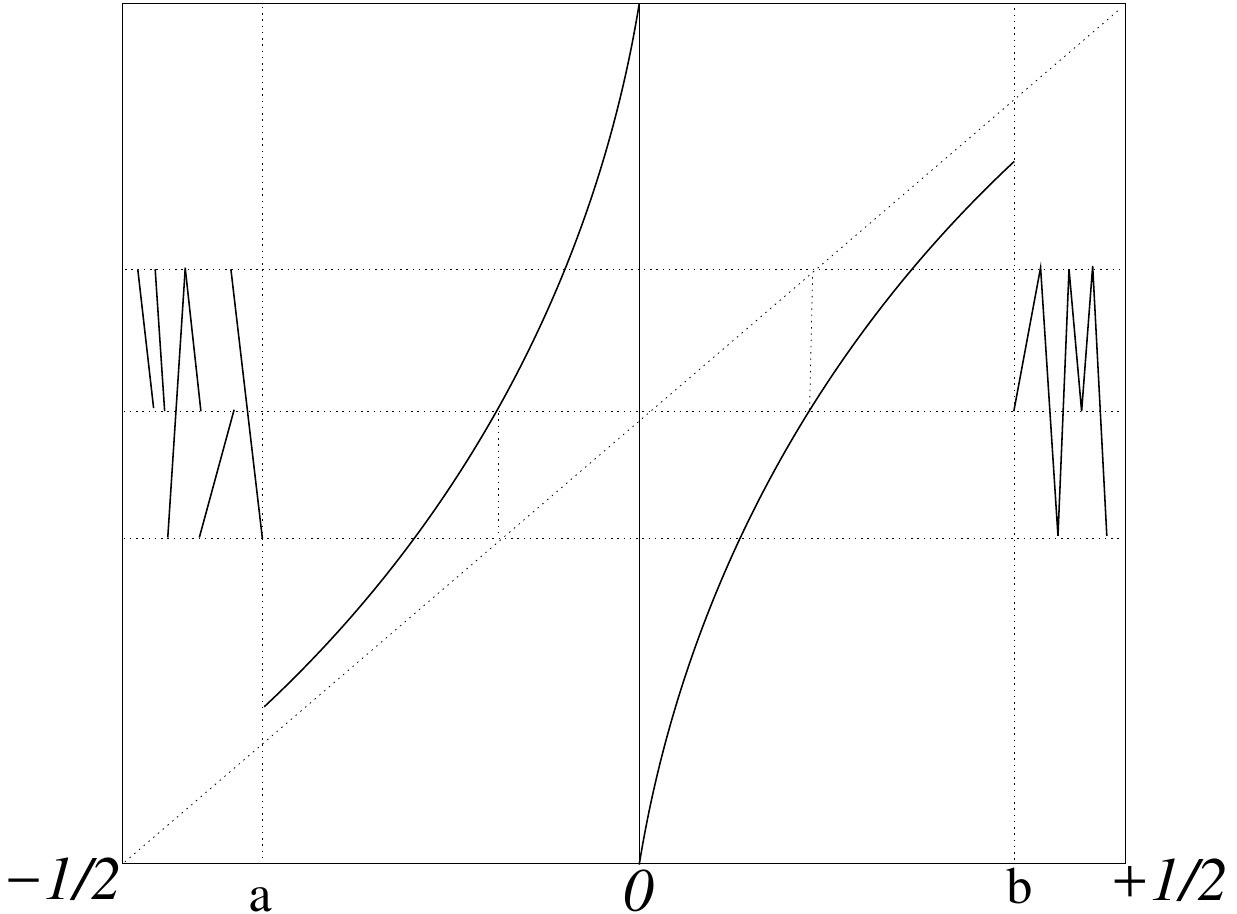}
\caption{\label{fig:Lorinfty}Example of a piecewise expading
  maps with infinitely many branches in the setting of
  Theorem~\ref{mthm:expslowapprox} but with an infinite
  $\cD$.}
\end{figure}

We can perform this extension in a way that
\begin{itemize}
\item the slope of the affine branches be larger than $2$;
\item the monotonicity domains form a denumerable partition
  of $J$;
\item  the singularity at $0$ is a Lorenz-like singularity
  which together with the discontinuity points for a
  non-degenerate singular set;
\item every discontinuity point of the map is sent to $0$ in
  finitely many iterates and the orbit of the
  discontinuities up to arriving at $0$ forms a finite
  subset.
\end{itemize}

\begin{conjecture}
  \label{conj:infinitesing}
  Exponential slow recurrence to the singular/discontinuous
  set still holds in the setting of
  Theorem~\ref{mthm:expslowapprox} with $\cS=\emptyset$ and
  countably many discontinuities.
\end{conjecture}

Extensions of Theorem~\ref{mthm:LD-sing-hyp-attracting} to
the class of sectional-hyperbolic attracting sets for flows
in higher dimensions, with dimension of the central direction
higher than two, introduced by Metzger-Morales in
\cite{MeMor06}, seem to involve subtle questions on the
smoothness of the stable foliation of these sets which, on
the one hand, might prevent the existence of a smooth
quotient map of the Poincar\'e return map over the stable
foliation in a natural way and, on the other hand, the
dynamics of higher dimensional piecewise expanding maps is
not so well understood as its one-dimensional couterpart,
where the boundaries of the domains of smoothness have low
complexity.

\begin{conjecture}
  \label{conj:sectionalhyp}
  Large deviations with respect to Lebesgue measure versus
  physical measures, for continuous observables on a
  neighborhood of general sectional-hyperbolic attracting
  sets for $C^2$ flows have exponential upper bound.
\end{conjecture}

Another issue is regularity: what can we say about large
deviations for singular-hyperbolic attracting sets of $C^1$
flows?

\begin{conjecture}\label{conj:c1smooth}
  The statements of
  Theorem~\ref{mthm:LD-sing-hyp-attracting} and
  Conjecture~\ref{conj:sectionalhyp} are still valid for
  $C^1$ flows.
\end{conjecture}

Using the existence of Markov towers with exponential tails
for the one-dimensional map as in Corollary~\ref{thm:alp} it
is natural to search for statistical properties for the
flows in the setting of
Theorem~\ref{mthm:LD-sing-hyp-attracting} along the lines of
\cite{AMV15,ArMel16,ArMel17}. This will be done in a
systematic way in \cite{ArMel18}.


\subsection{Organization of the text}
\label{sec:comments-organiz-tex}

The proof of Theorem~\ref{mthm:LD-sing-hyp-attracting}
demanded the extension of the construction of adapted
cross-sections used in \cite{APPV} for singular-hyperbolic
\emph{attractors (i.e. transitive attracting sets)} since the
existence of a dense forward orbit inside the attracting set
was crucial to find Poincar\'e sections whose boundaries
are contained in stable
manifolds of some singularity of the attracting
set.
Moreover, since we are not assuming the existence of a dense
regular orbit, we need to consider the possible existence of
singularities in the attracting set which are not
Lorenz-like.

This construction, without the transitivity assumption, is
presented in Section \ref{sec:general-setting} were a global
Poincar\'e map is built and
Theorem~\ref{thm:propert-singhyp-attractor} on the
representation of this map as a skew-product over a
one-dimensional transformation is proved.

The proof of item (1) of
Theorem~\ref{mthm:LD-sing-hyp-attracting} follows from
Theorem~\ref{thm:propert-singhyp-attractor}, whose proof is
presented in Subsectionc~\ref{sec:global-poincare-retu},
together with Theorem~\ref{thm:appv-attracting} presented in
Subsection~\ref{sec:coroll-extens}.  The deduction of item
(2) of Theorem~\ref{mthm:LD-sing-hyp-attracting} from the
reduction to a one-dimensional transformation in the setting
of Theorem~\ref{mthm:expslowapprox}, following the route in
\cite{araujo-pacifico2006,araujo2006a}, is presented in
Section \ref{sec:reduct-large-deviat} assuming the statement
of Theorem~\ref{mthm:expslowapprox}. Then
Theorem~\ref{mthm:expslowapprox} together with
Corollary~\ref{mcor:expmapsholes} are proved in Section
\ref{sec:exponent-slow-recurr-1}.

\subsection*{Acknowledgments}
\label{sec:acknowledgments}

We thank the referee for the careful reading of the
manuscript and the many detailed questions which greatly
helped to improve the statements of the results and the
readability of the text.

This is the PhD thesis of A. Souza and part of the PhD
thesis work of E. Trindade at the Instituto de Matem\'atica
e Estat\'{\i}stica-Universidade Federal da Bahia (UFBA,
Salvador) under a CAPES (Brazil) scholarship. Both thank the
Mathematics and Statistics Institute at UFBA for the use of
its facilities and the financial support from CAPES during
their M.Sc. and Ph.D. studies.


\section{Existence of adapted cross-sections and
  construction of global Poincar\'e map}
\label{sec:general-setting}

We let $X\in{\fX}^2(M)$ admit an singular-hyperbolic
attracting set $\Lambda=\bigcap_{t>0}\overline{X^t(U_0)}$
for some open neighborhood $U_0$, with $DX^t$-invariant
splitting $T_\Lambda M=E^{s}\oplus E^{cu}$, $d_s=\dim E^s=d-2$
and $d_{cu}=\dim E^{cu}=2$.

\subsection{Properties of singular-hyperbolic attracting
  sets}
\label{sec:propert-partially-hy}

We extend the stable direction on $\Lambda$ to a
$DX^t$-invariant stable bundle over $U_0$ and then integrate
these directions into a topological foliation of $U_0$ which
admits H{\"o}lder-$C^1$ holonomies, combining the following
results.

\begin{proposition} \label{pr:Es} Let $\Lambda$ be a
  partially hyperbolic attracting set.  The stable bundle
  $E^s$ over $\Lambda$ extends to a continuous uniformly
  contracting $DX^t$-invariant bundle $E^s$ over an open
  neighborhood of $\Lambda$.
\end{proposition}

\begin{proof} See~\cite[Proposition~3.2]{ArMel17} and
  note~\cite[Remark~3.3]{ArMel17} ensuring that the
  existence of this extension does not depend on
  transitivity.
\end{proof}

We assume without loss of generality that $E^s$ extends as
in Proposition~\ref{pr:Es} to $U_0$.

Recall that $B^k$ is the $k$-dimensional open unit disk
endowed with the Euclidean distance induced by the Euclidean
norm $\|\cdot\|_2$, and let $\mathrm{Emb}^2(B^k,M)$ denote
the set of $C^2$ embeddings $\phi:B^k\to M$ endowed with
the $C^2$ distance.

\begin{proposition}\label{pr:Ws}
  Let $\Lambda$ be a partially hyperbolic attracting set.
  There exists a positively invariant neighborhood $U_0$ of
  $\Lambda$, and constants $\lambda\in(0,1), c>0$, such that
  the following are true:
  \begin{enumerate}
  \item For every point $x \in U_0$ there is a $C^2$
    embedded $d_s$-dimensional disk $W^s_x\subset M$, with
    $x\in W^s_x$, such that
\begin{enumerate}
\item $T_xW^s_x=E^s_x$.
\item $X^t(W^s_x)\subset W^s_{X^t x}$ for all $t\ge0$.
\item $d(X^t x,X^t y)\le c\lambda^t d(x,y)$ for all
  $y\in W^s_x$, $t\ge0$.
\end{enumerate}

\item The disks $W^s_x$ depend continuously on $x$ in the
  $C^0$ topology: there is a continuous map
  $\gamma:U_0\to {\rm Emb}^0(B^{d_s},M)$ such that
  $\gamma(x)(0)=x$ and $\gamma(x)(B^{d_s})=W^s_x$.
  Moreover, there exists $L>0$ such that
  $\Lip\gamma(x)\le L$ for all $x\in U_0$.
\item The family of disks $\{W^s_x:x\in U_0\}$ defines a
  topological foliation of $U_0$.
  \end{enumerate}
\end{proposition}

\begin{proof} See~\cite[Theorem~4.2 and Lemma~4.8]{ArMel17},
  where $\Lip\gamma(x)$ is the Lipschitz constant of
  $\gamma(x)$, given by
  $\sup\{\frac{\dist(\gamma(x)(u),\gamma(x)(v))}{\|u-v\|_2}:u\neq
  v, w,v\in B^{d_s}\}$. Note also that the results
  in~\cite[Section~4]{ArMel17} do not use either
  transitivity or central volume expansion in its proofs,
  that is, these results hold for partial hyperbolic
  attracting sets.
\end{proof}

The splitting $T_\Lambda M=E^s\oplus E^{cu}$ extends
continuously to a splitting $T_{U_0} M=E^s\oplus E^{cu}$
where $E^s$ is the invariant uniformly contracting bundle in
Proposition~\ref{pr:Es} and, in general, $E^{cu}$ is not
invariant.  Given $a>0$, we consider the
\emph{center-unstable cone field}
\[
  \cC^{cu}_x(a)
  =
  \{v= v^s+v^{cu}\in E^s_x\oplus E^{cu}_x:
  \|v^s\|\le a\|v^{cu}\|\}, \quad x\in U_0.
\]

\begin{proposition} \label{pr:cu} Let $\Lambda$ be a
  partially hyperbolic attracting set.  There exists $T_0>0$
  such that for any $a>0$, after possibly shrinking $U_0$,
\[
DX^t\cdot \cC^{cu}_x(a)\subset \cC^{cu}_{X^t x}(a) \quad\text{for all $t\ge T_0$, $x\in U_0$}.
\]
\end{proposition}

\begin{proof} See~\cite[Proposition~3.1]{ArMel17} and again
  note that the results in~\cite[Section~3]{ArMel17} do not
  use either transitivity or central volume expansion in its
  proofs.
\end{proof}

\begin{proposition} \label{pr:VE}
Let $\Lambda$ be a singular hyperbolic attracting set.
After possibly increasing $T_0$ and shrinking $U_0$, there exist
constants $K,\theta>0$ such that
$|\det(DX^t| E^{cu}_x)|\geq K \, e^{\theta t}$ for all
$x\in U_0$, $t\geq 0$.
\end{proposition}

\begin{proof}
  See \cite[Proposition 2.10]{ArMel18}
\end{proof}

\subsubsection{The stable lamination is a topological
  foliation}
\label{sec:stabltopfolio}

The Stable Manifold Theorem \cite{Sh87} ensures the
existence of an $X^t$-invariant stable lamination
$\cW^s_\Lambda$ consisting of smoothly embedded disks
$W^s_x$ through each point $x\in\Lambda$.  Although not true
for general partially hyperbolic attractors, for
singular-hyperbolic attractors in our setting
$\cW^s_\Lambda$ indeed defines a topological foliation in an
open neighborhood of $\Lambda$.

\begin{theorem} \label{thm:stabltopfol} Let $\Lambda$ be a
  singular hyperbolic attracting set.  Then the stable
  lamination $\cW^s_\Lambda$ is a topological foliation of
  an open neighborhood of~$\Lambda$.
\end{theorem}

\begin{proof}
  See \cite[Theorem 5.1]{ArMel18} where it is shown that
  $\cW^s_\Lambda$ coincides with the topological foliation
  $\{W_x^s:x\in U_0\}$ of item (3) of
  Proposition~\ref{pr:Ws}.
\end{proof}

From now on, we refer to $\cW^s=\{W^s_x:x\in\Lambda\}$ as
the \emph{stable foliation}.

\subsubsection{Absolute continuity of the stable foliation}
\label{sec:abscontWs}

A key fact for us is regularity of stable holonomies.  Let
$Y_0,\,Y_1\subset U_0$ be two smooth disjoint
$d_{cu}$-dimensional disks that are transverse to the stable
foliation $\cW^s$.  Suppose that for all $x\in Y_0$, the
stable leaf $W^s_x$ intersects each of $Y_0$ and $Y_1$ in
precisely one point.  The {\em stable holonomy}
$H:Y_0\to Y_1$ is given by defining $H(x)$ to be the
intersection point of $W^s_x$ with $Y_1$.

\begin{theorem}\label{thm:H}
  The stable holonomy $H:Y_0\to Y_1$ is absolutely
  continuous. That is, $m_1\ll H_*m_0$ where $m_i$ is
  Lebesgue measure on $Y_i$, $i=0,1$.  Moreover, the
  Jacobian $JH:Y_0\to\R$ given by
  \begin{align*}
    JH(x)=\frac{dm_1}{dH_*m_0}(Hx)=
    \lim_{r\to0}\frac{m_1(H(B(x,r)))}{m_0(B(x,r))},\quad x\in Y_0,
  \end{align*}
  is bounded above and below and is $C^\epsilon$ for some
  $\epsilon>0$.
\end{theorem}

\begin{proof}
  This essentially follows from \cite[Theorems 8.6.1 and
  8.6.13]{BarPes2007} or from~\cite[Theorem 7.1]{Pesin2004}
  applied to $f=X^T$ for a large enough fixed $T>0$ which
  satisfies conditions (4.1), (4.2) and (4.3) of
  \cite[Section 4]{Pesin2004} in the open positively
  invariant set $U$; see \cite[Theorem 6.3]{ArMel18}.
\end{proof}

Hence, we can assume without loss of generality, that there
exists a foliation $\cW^s$ of $U_0$, which continuously
extends the stable lamination of $\Lambda$ together with a
positively invariant field of cones
$(\cC^{cu}_x)_{x\in U_0}$ on $T_{U_0}M$.  Moreover, the
Jacobian of holonomies along contracting leaves on
cross-sections of singular-hyperbolic attracting sets in our
setting is a H\"older function.

\subsection{Global Poincar\'e return map}
\label{sec:global-poincare-retu}

In \cite{APPV} the construction of a global Poincar\'e map
for any singular-hyperbolic attractor is carried out based
on the existence of ``adapted cross-sections'' and
H{\"o}lder-$C^1$ stable holonomies on these
cross-sections. With the results just presented this
construction can be performed for any singular-hyperbolic
attracting set.

This construction was presented in~\cite[Sections 3 and
4]{ArMel18}: we obtain
\begin{itemize}
\item a finite collection
  $\Xi=\Sigma_1+\dots+\Sigma_m$\footnote{We write $A+B$ the
    union of the disjoint subsets $A$ and $B$.} of (pairwise
  disjoint) cross-sections to $X$ so that
  \begin{itemize}
  \item each $\Sigma_i$ is diffeomorphically identified with
    $(-1,1)\times\cD^{d_s}$;
  \item the stable boundary
    $\partial^s\Sigma_i\cong \{\pm1\}\times \cD^{d_s}$
    consists of two curves contained in stable leaves; and
  \item each $\Sigma_i$ is foliated by
    $W^s_x(\Sigma_i)= \bigcup_{|t|<\epsilon_0}X^t(W^s_x)\cap
    \Sigma_i$ for a small fixed $\epsilon>0$. We denote this
    foliation by $W^s(\Sigma_i), i=1,\dots,m$;
  \end{itemize}
\item a Poincar{\'e} map $R:\Xi\setminus\Gamma\to\Xi$ which
  is $C^2$ smooth in $\Sigma_i\setminus\Gamma, i=1,\dots,m$;
  preserves the foliation $W^s(\Xi)$ and a big enough time
  $T>0$, where $\Gamma=\Gamma_0\cup\Gamma_1$ is a finite
  family of stable disks $W^s_{x_i}(\Xi)$ so that
  \begin{itemize}
  \item
    $\Gamma_0=\{x\in\Xi:X^{T+1}(x)\in\bigcup_{\sigma\in
      S}\gamma^s_\sigma\}$ for
    $S=S(X,\Lambda)=\{\sigma\in\Lambda: X(\sigma)=\vec0\}$
    and $\gamma_\sigma^s$ is the local stable manifold of
    $\sigma$ in a small fixed neighborhood of $\sigma\in S$;
    and
  \item
    $\Gamma_1=\{x\in\Xi:R(x)\in\partial^s\Xi=\cup_i\partial^s\Sigma_i\}$
  \end{itemize}
\item and an open neighborhood $V_0$ of $S$ of so that every
  orbit of a regular point $z\in U_0\setminus V_0$
  eventually hits $\Xi$ or else $z\in\Gamma$.
\end{itemize}
Having this, the same arguments from \cite{APPV} (see
\cite[Proposition 4.1 and Theorem 4.3]{ArMel18}) show that
$DR$ contracts $T_\Xi W^s(\Xi)$ and
expands vectors on the unstable cones
$\{C^u_x(\Xi)=C^{cu}_x(a)\cap T_x\Xi\}_{x\in\Xi\setminus\Gamma}$. The
stable holonomies for $R$ enable us to reduce its dynamics
to a one-dimensional map, as follows.

Let $\gamma_i\subset\Sigma_i$ be curves that \emph{cross}
$\Sigma_i$, that is, $\dot\gamma_i\in C^u_{\gamma}$ and for
all $x\in \Sigma_i$ the stable leaf $W^s_x(\Sigma_i)$
intersects $\gamma_i$ in precisely one point, $i=1,\dots,m$.
The \emph{(sectional) stable holonomy}
$h:\Xi\to\gamma=\sum_i\gamma_i$ is defined by setting
$h(x)$ to be the intersection point of $W^s_x(\Sigma_i)$
with $\gamma_i$, $x\in\Sigma_i, i=1,\dots,m$.

  \begin{lemma}\label{le:hC1+}
    The stable holonomy $h$ is $C^{1+\epsilon}$ for some
    $\epsilon>0$.
  \end{lemma}

  \begin{proof}
    See \cite[Lemma 7.1]{ArMel18}
  \end{proof}

  Following the same arguments in \cite{APPV} (see also
  \cite[Section 7]{ArMel18}) we obtain a one-dimensional
  piecewise $C^{1+\alpha}$ quotient map over the stable
  leaves $f:\gamma\setminus\Gamma\to \gamma$ for some
  $0<\alpha<1$ so that $h(Rx) = f(x)$ and
  $|f'x|>2$. Choosing smooth parametrizations of $\gamma_i$
  in a concatenated fashion we can write
  $I=I_1\cup\dots\cup I_m$ and $f:I\setminus\cD\to I$ where
  $\cD$ is a finite set of points identified with
  $h(\Gamma)$ and each $I_i$ is identified with
  $\gamma_i, i=1,\dots,m$. In addition, as shown in
  \cite[Proof of Lemma 8.4]{ArMel18},
  $|f'\mid_{I\setminus\cD}|$ behaves near singular points
  $\cS$, identified with $h(\Gamma_0)$ as a subset of $\cD$,
  as a power of the distance to $\cS$, as in assumption of
  the statement of \eqref{eq:der0}.

  We can use the choice of coordinates and concatenation
  defining $I$ from $\gamma=\sum_i\gamma_i$ to choose
  coordinates in $\Xi=\sum_i\Sigma_i$ so that: $\Xi$ can be
  identified with $I\times B^{d_s}$; and each $\Sigma_i$ with
  $I_i\times B^{d_s}, i=1,\dots,m$.
  
  This construction can be summarized as in \cite[Theorem
  5]{ArGalPac} as follows, with adaptations to our more
  general setting: items (1-5) can be found in \cite{APPV}
  but item (6), which is crucial for us, will be obtained in
  Remark~\ref{rmk:discontsing} following
  Corollary~\ref{cor:singular-adapted-section} in
  Subsection~\ref{sec:stable-manifolds-sin}. This is the
  only argument where the assumption of connectedness is
  used; see \cite{ArMel17,ArMel18}.

  In what follows, we say that a function
  $\vfi:\Xi_0=\Xi\setminus\Gamma\to\RR$ has
  \emph{logarithmic growth near $\Gamma$} if there is
  $K=K(\vfi)>0$ so that
  $|\vfi|\chi_{B(\Gamma,\delta)}\le K \Delta_\delta\circ h$
  for every small enough $\delta>0$.

\begin{theorem}{\cite[Theorem 5, Section 4, p 1021]{ArGalPac}}
  \label{thm:propert-singhyp-attractor}
  For a $C^2$ vector field $X$ on a compact manifold having
  a connected singular hyperbolic attracting set $\Lambda$,
  there exists $\alpha>0$ and a finite family $\Xi$ of
  adapted cross-sections and a global Poincar\'e map
  $R:\Xi_0\to\Xi$, $R(x)=X^{\tau(x)}(x)$ such that
  \begin{enumerate}
  \item the domain $\Xi_0=\Xi\setminus\Gamma$ is the entire
    cross-sections with a family $\Gamma$ of finitely many
    smooth arcs removed and
    \begin{enumerate}
    \item $\tau:\Xi_0\to[\tau_0,+\infty)$ is a smooth
      function with logarithmic growth near $\Gamma$ and
      bounded away from zero by some uniform constant
      $\tau_0>0$;
    \item there exists a constant $\kappa>0$ so that
      $|\tau(y)-\tau(w)|<\kappa\dist(y,w)$ for all points
      $w\in W^s(y,\Xi)$ in the stable leaf through a point
      $y$ inside a cross-section of $\Xi$;
    \end{enumerate}
  \item We can choose coordinates on $\Xi$ so that the map
    $R$ can be written as $F:\tilde Q\to Q$,
    $F(x,y)=(f(x),g(x,y))$, where $Q= I\times B^{d_s}$,
    $ I=[0,1]$ and $\tilde Q=Q\setminus\tilde\Gamma$ with
    $\tilde\Gamma=\cD\times B^{d_s}$ and
    $\cD=\{c_1,\dots,c_n\}\subset I$ a finite set of points.
  \item The map $f: I\setminus\cD\to I$ is a piecewise
    $C^{1+\alpha}$ map of the interval with finitely many
    branches defined on the connected components of
    $I\setminus\cD$ and has a finite set of ergodic
    a.c.i.m. $\mu^i_f, i=1,\dots,k$ whose ergodic basins
    cover $I$ Lebesgue modulo zero. Also
    \begin{enumerate}
    \item $\inf\{|f'x|:x\in I\setminus\cD\}>2$;
    \item the elements $c$ of $\cD$ have a one-sided critical
      order $0<\alpha(c^\pm)\le1$ as in \eqref{eq:der0};
    \item $1/|Df|$ has universal bounded
      $p$-variation\footnote{See \cite{Ke85} for the
        definition of $p$-variation.}; and
    \item $d\mu^i_f/dm$ has bounded $p$-variation for some
      $p>0$.
    \end{enumerate}
  \item
    The map $g:\tilde Q\to B^{d_s}$ preserves and uniformly contracts
    the vertical foliation
    $\cF=\{\{x\}\times B^{d_s}\}_{x\in I}$ of $Q$: there
    exists $0<\lambda<1$ such that $
    \dist(g(x,y_1),g(x,y_2)) \leq \lambda \cdot |y_1-y_2|$
    for each $y_1, y_2 \in  B^{d_s}$. 
  \item The map $F$ admits a finite family of physical
    ergodic probability measures $\mu^{i}_{F}$ which are
    induced by $\mu^i_f$ in a standard way\footnote{See
      \cite[Section 6.1]{APPV} where it is shown how to get
      $h_*\mu^i_F=\mu^i_f$.}. The Poincar\'e time $\tau$ is
    integrable both with respect to each $\mu^{i}_{f}$ and
    with respect to the two-dimensional Lebesgue area
    measure of $Q$.
  \item The subset\footnote{The subset $\cS$ can be
      identified with $h(\Gamma_0)$ while $\cD\setminus\cS$
      can be identified with $h(\Gamma_1)$}
    $\cS=\{c\in\cD: 0<\alpha(c)<1\}$ is nonempty and
    satisfies
    $\exists T_0\in\ZZ^+\, \forall c\in\cD\setminus\cS\,
    \exists T=T(c)\leq T_0$ so that $f^T(c)\in\cS$,
    $\forall 0<j<T\exists c_j\in\cD\setminus\cS:
    f^j(c)\in\Delta(c_j,\delta_{c_j})$, and we can find
    $\epsilon,\delta>0$ such that
    $f\mid_{\Delta(c,\delta)}$ is a diffeomorphism into
    $\Delta(f^T(c),\epsilon)$. Moreover
    \begin{enumerate}
    \item
      $c\in\cS\ \iff \lim_{t\to
        c}|t-b|^{1-\alpha(c)}\cdot|f'(t)|$ exists and is finite;
    \item $c\in\cD\setminus\cS \iff$ the limit
      $\lim_{t\to c}|f'(t)|$ exists and is finite.
    \end{enumerate}
  \end{enumerate}
\end{theorem}

\begin{remark}
  \label{rmk:lebngbhdD}
  Due to the dimension and codimension of $\cD$ as a
  submanifold of the quotient $I=\Xi/W^s(\Xi)$ together with
  logarithmic growth of $\tau$ near $\Gamma$, there exists
  $C_d,d>0$ such that for all small $\rho>0$
$
\leb\{x\in M: \dist(x,\cD)<\rho\}\le C_d \rho^d,
$
  that is, the Lebesgue measure of neighborhoods of $\cD$ is
  comparable to a power of the distance to $\cD$. In our
  sectional-hyperbolic case $d=1$.
\end{remark}


\subsection{Equivalence between SRB/physical measure and
  equilibrium state}
\label{sec:equival-between-srbp}

We now prove Theorem~\ref{thm:appv-attracting} showing that
in singular-hyperbolic attracting sets for a $C^2$ smooth
flow we can characterize physical/SRB measures in the same
way as in hyperbolic attracting sets.

\begin{proof}[Proof of
  Theorem~\ref{thm:appv-attracting}]

  Let $\mu$ be a $X$-invariant probability measure supported
  in the singular-hyperbolic attracting set $\Lambda$.

  We start by recalling that from
  Theorem~\ref{thm:propert-singhyp-attractor} we can follow
  \cite[Sections 6-8]{APPV} to
  obtain~\eqref{eq:appv-attracting}. More precisely: the
  arguments in \cite[Sections 6-8]{APPV} show that from the
  existence of a finite family $\Xi$ of adapted
  cross-sections and a global Poincar\'e map satisfying
  items (1-5) of
  Theorem~\ref{thm:propert-singhyp-attractor}, we induce
  finitely many physical/SRB ergodic probability measures
  $\mu_1,\dots,\mu_k$ for the flow (one for each measure
  $\mu_i^f$ given in
  Theorem~\ref{thm:propert-singhyp-attractor}) whose ergodic
  basins cover the trapping region of the attracting set
  $\Lambda$ and if, in addition, $\Lambda$ is transitive,
  then these ergodic physical measures cannot be distinct;
  for this final reasoning see~\cite[Subsection
  7.1]{APPV}. This proves item (1) of the statement of
  Theorem~\ref{thm:appv-attracting}.

  To prove item (2), we note that, since (i)
  $\int\log|\det DX^1\mid_{E^{cu}}|\,d\mu_i>0$ by
  singular-hyperbolicity, (ii) each $\mu_i$ is an ergodic
  physical/SRB measure, (iii) the Lyapunov exponent along
  the direction of the flow is zero and (iv) this direction
  is contained in the central direction then, by the
  characterization of measures satisfying the Entropy
  Formula \cite{LY85}, we get
  \begin{align}\label{eq:entroform}
  h_{\mu_i}(X^1)=\int\lambda^+(x)\,d\mu_i(x)= \int\log|\det
  DX^1\mid_{E^{cu}}|\,d\mu_i
  \end{align}
  where $\lambda^+(x)$ is the positive Lyapunov exponent
  along the orbit of $x\in\Lambda$ in the direction of
  $E^{cu}_x$. That is, each $\mu_i$ satisfies property (a):
  it is an equilibrium state with respect to the central
  Jacobian, $i=1,\dots,k$.

  We now prove the implication $(c)\implies(a)$.

  If the basin $B(\mu)$ of $\mu$ has positive
  Lebesgue measure, then by invariance $B(\mu)\cap U$ must
  have positive Lebesgue measure. So we get a Lebesgue
  modulo zero decomposition
  $B(\mu)\cap U=U\cap\big(\sum_{i=1}^k B(\mu)\cap
  B(\mu_i)\big)$. By definition of
  physical measure, this means that for each continuous
  observable $\vfi:U\to\RR$
  \begin{align*}
    \int\vfi\,d\mu
    &=
    \frac1{\leb(U)}\int_U \int\vfi\,
    d\left(\lim_{n\to+\infty}\frac1n\sum_{j=0}^{n-1}\delta_{f^jx}\right)
      \,d\leb(x)
      \\
      &=
      \sum_{i=1}^k \frac{\leb(B(\mu)\cap B(\mu_i)\cap
        U)}{\leb(U)}\int \vfi\,d\mu_i,
  \end{align*}
  where the limit above is in the weak$^*$ topology of the
  probability measures of the manifold.  Hence, we obtain
  $\mu=\sum_{i=1}^k \frac{\leb(B(\mu)\cap B(\mu_i)\cap
    U)}{\leb(U)}\mu_i$ and $\mu$ is a convex linear
  combination of the ergodic physical/SRB measures provided
  by item (1). In particular, $\mu=\mu_i$ for some
  $i\in\{1,\dots,k\}$ by ergodicity and $\mu$ is a equilibrium
  state with respect to the central Jacobian.

  Next we prove that $(a)\implies(b)\implies(c)$. The
  assumption (a) implies, since $\mu$ is ergodic, the
  Lyapunov exponent along the flow direction is zero and
  this direction is contained in the two-dimensional central
  direction, that is, \eqref{eq:entroform} is true for $\mu$ in
  the place of $\mu_i$.

  Hence, by the work of Ledrappier \cite[Theorem
  2,7]{Ledrappier1984}, we conclude that $\mu$ is a $SRB$
  measure: $\mu$ has absolutely continuous disintegration
  along unstable manifolds $W^u(x)$ for
  $\mu$-a.e. $x\in\Lambda$ if the Entropy Formula
  \eqref{eq:entroform} holds.\footnote{The proof
    of~\cite[Theorem 2.7]{Ledrappier1984} based on a
    combination of \cite{Ru78} with
    \cite{ledrappier_strelcyn_1982} does not assume that the
    measure has only non-zero Lyapunov exponents: it also
    applies to non-uniformly partially hyperbolic measures.}
  Hence, by invariance of the unstable manifolds and
  smoothness of the flow, we see that $\mu$ has absolutely
  continuous disintegration along the central-unstable
  manifolds $W^{cu}(x)$ for $\mu$-a.e. $x$, where
  $W^{cu}(x)=\cup_{t\in\RR}X_tW^u(x)$. This shows that
  $(a)\implies(b)$.

  Let $\Lambda_1$ be a full $\mu$-measure subset of
  $\Lambda$ where the previous absolutely continuous
  disintegration property holds.  Since $\Lambda$ is a
  attracting set, then $\Lambda$ contains all unstable
  manifolds\footnote{For if $y\in W^u(x)\cap U$ and
    $x\in\Lambda$, then
    $d(X^{-t}(y),X^{-t}(y))\xrightarrow[t\to+\infty]{}0$
    thus $X^{-t}(y)\subset U$ for all $t\ge0$, that is,
    $y\in \cap_{t\ge0}X^t(U)=\Lambda$.} and so all the
  central-unstable manifolds $W^{cu}(x)$ for
  $x\in\Lambda_1$. We recall that $W^{cu}(x)$ is tangent to
  $E^c(x)$ at $x\in\Lambda_1$.

In addition, because $\Lambda$ has a partially hyperbolic
splitting, \emph{every point} $y$ of $\Lambda$ admits a
stable manifold $W^s(y)$ which is tangent to $E^s(y)$ at
$y$. Thus $W^s(y)$ are transverse to $W^{cu}(x)$ for all
$y\in W^{cu}(x)$ and $x\in\Lambda_1$. 

Moreover, the future time averages of $z\in W^s(y)$ and $y$
are the same for all continuous observables. Also, the
future time averages of $y\in A_x\subset W^u(x)$ and $x$ are
the same for all continuous observables and some subset
$A_x$ of $W^{cu}(x)$ with positive area (by the absolutely
continuous disintegration) for each $x\in\Lambda_1$. Hence
the subset
\begin{align*}
  B=\{W^s(y): y\in A_x, x\in \Lambda_1\}
\end{align*}
is contained in the ergodic basin of the measure $\mu$.

For $C^2$ smooth partially hyperbolic flows it is well-known
that the strong-stable foliation $\{W^s(x)\}_{x\in\Lambda}$
is an absolutely continuous foliation of $U$; recall
Theorem~\ref{thm:H} and see~\cite{Pesin2004}. In particular,
the set $B$ has positive Lebesgue measure. Hence
$\leb(B(\mu))\ge\leb(B)>0$ and $\mu$ is a physical measure.

This shows that $(b)\implies(c)$ and completes the proof of
item (2).

Finally, the characterization of $\EE$ is a consequence of
the equivalence obtained in item (2): using Ergodic
Decomposition \cite{Man87} applied to both sides of the
Entropy Formula
\begin{align*}
  h_\mu(X^1)&=
              \int\log|\det DX^1\mid_{E^{cu}}|\,d\mu\iff\\
  \int h_{\mu_x}(X^1)\,d\mu(x)&=\int
  \int\log|DX^1\mid_{E^{cu}}|\,d\mu_x \,d\mu(x)
\end{align*}
together with Ruelle's Inequality
$h_{\mu_x}(X^1)\le\int\log|DX^1\mid_{E^{cu}}|\,d\mu_x$
ensures that $\mu_x\in\EE$ for $\mu$-a.e $x$ and so
$\mu_x=\mu_{i(x)}$ for some $i(x)\in\{1,\dots,k\}$, since
$\mu_x$ is ergodic for $\mu$-a.e. $x$. Thus
$\mu$ is a linear convex combination of $\mu_1,\dots,\mu_k$,
completing the proof of item (3).
\end{proof}


\subsection{Density of stable manifolds of singularities}
\label{sec:denstable-manifolds-sin}

The following is essential to obtain the property in item
(6) of Theorem~\ref{thm:propert-singhyp-attractor}.

From Subsection~\ref{sec:global-poincare-retu} the
one-dimensional map $f:I\setminus\cD\to I$ can be written
$f:\sum_j I_j \to I:=\overline{\sum_j I_j}$, where the $I_j$
are the finitely many connected components of
$I\setminus\cD$ (that is, open subintervals).

\subsubsection{Topological properties of the dynamics of $f$}
\label{sec:topological-dynamics}

The following provides the existence of a special class of
periodic orbits for $f$.

\begin{proposition}
  \label{pr:intervalocresce}
  Let $f:\sum_j I_j\to I$ be a piecewise $C^1$ expanding map
  with finitely many branches $I_1,\dots,I_l$ such that each
  $I_j$ is a non-empty open interval, $|Df\mid
  I_j|\ge\sigma>2$ and $I\setminus(\sum_j I_j)$ is finite.

  Then, for each small $\delta>0$ there exists $n=n(\delta)$
  such that, for every non-empty open interval
  $J\subset\sum_j I_j$ with $|J|\ge\delta$, we can find
  $0\le k\le n$, a sub-interval $\hat J$ of $J$ and
  $1\le j \le l$ satisfying
  \begin{align*}
    f^k\mid \hat J : \hat J\to I_j \quad\text{is a diffeomorphism}.
  \end{align*}
  In addition, $f$ admits finitely many periodic orbits
  $\cO(p_1),\dots,\cO(p_k)$ contained in $\sum_j I_j$ with
  the property that every non-empty open interval
  $J\subset\sum_j I_j$ admits an open sub-interval $\hat J$,
  a periodic point $p_j$ and an iterate $n$ such that
  $f^{n}\mid\hat J$ is a diffeomorphism onto a neighborhood
  of $p_j$.
\end{proposition}

\begin{proof}
  See \cite[Lemma 6.30]{AraPac2010}.
\end{proof}

\begin{remark}\label{rmk:denstableF}
  \begin{enumerate}
  \item For the bidimensional map $F$ this shows that there
    are finitely many periodic orbits
    $\cO(P_1), \dots,\cO(P_k)$ for $F$ so that
    $\pi(\cO(P_i))=\cO(p_i), i=1,\dots,k$, where
    $\pi:Q\to I$ is the projection on the first
    coordinate. Moreover, the union of the stable manifolds
    of these periodic orbits is dense in $Q$. See
    \cite[Section 6.2]{AraPac2010} for details.
  \item This also implies that the stable manifolds of the
    periodic orbits $P_i$ obtained above are dense in a
    neighborhood $U_0$ of $\Lambda$.

    Indeed, we can write the flow $X_t$ on a neighborhood of
    $\Lambda$ as a suspension flow over $F$; see
    \cite{APPV}. Then the orbit of each $P_i$ is periodic
    and hyperbolic and $W^s_G(P_i)$ is the suspension of
    $W^s_F(P_i)$. Therefore, the density of
    $\cup_iW^s_F(P_i)$ in $Q$ implies the density of
    $\cup_iW^s_G(P_i)$ in a neighborhood $U_0$ of $\Lambda$.
  \end{enumerate}
\end{remark}

\subsubsection{Ergodic properties of $f$}
\label{sec:ergodic-properties}

The map $f$ is piecewise expanding with H\"older
derivative which enables us to use strong results on
one-dimensional dynamics.

\subsubsection*{Existence and finiteness of acim's}
\label{sec:existence-finiten-ac}

It is well known \cite{HK82} that $C^1$ piecewise expanding
maps $f$ of the interval such that $1/|Df|$ has bounded
variation have absolutely continuous invariant probability
measures whose basins cover Lebesgue almost all points of
$I$.

Using an extension of the notion of bounded variation this
result was extended in \cite{Ke85} to $C^1$ piecewise expanding maps
$f$ such that $g=1/|f'|$ is $\alpha$-H\"older for some
$\alpha\in(0,1)$.  In addition from \cite[Theorem 3.3]{Ke85}
there are finitely many ergodic absolutely continuous
invariant probability measures $\upsilon_1,\dots,\upsilon_l$
of $f$ and every absolutely continuous invariant probability
measure $\upsilon$ decomposes into a convex linear
combination $\upsilon=\sum_{i=1}^l a_i\upsilon_i$. From
\cite[Theorem 3.2]{Ke85} considering any subinterval
$J\subset I$ and the normalized Lebesgue measure
$\lambda_J=(\lambda\mid J)/\lambda(J)$ on $J$, every
weak$^*$ accumulation point of
$n^{-1}\sum_{j=0}^{n-1} f_*^j(\lambda_J)$ is an absolutely
continuous invariant probability measure $\upsilon$ for $f$
(since the indicator function of $J$ is of generalized
$1/\alpha$-bounded variation). Hence the basin of the
$\upsilon_1,\dots,\upsilon_l$ cover $I$ Lebesgue modulo
zero:
$ \lambda\big(I\setminus(B(\upsilon_1)+\dots+
B(\upsilon_l)\big) =0. $

Note that from \cite[Lemma 1.4]{Ke85} we also know that
\emph{the density $\varphi$ of any absolutely continuous
  $f$-invariant probability measure} \emph{is bounded from
  above}.

\subsubsection*{Absolutely continuous measures and
   periodic orbits}
\label{sec:absolut-contin-measu}

Now we relate some topological and ergodic properties.

\begin{lemma}\label{le:persuppacim}
  For each periodic orbit $\cO(p_i)$ of $f$ given by
  Proposition~\ref{pr:intervalocresce}, there exists some
  ergodic absolutely continuous $f$-invariant probability
  measure $\upsilon_j$ such that $p_i\in\supp\upsilon_j$,
  and vice-versa.
\end{lemma}

\begin{proof}
  Define $E=\{1\le i\le k: \exists 1\le j\le l$ s.t.
  $\cO(p_i)\in\supp\upsilon_j\}$. Note that since
  $\inter(\supp\upsilon_1)$ is nonempty, then for an
  interval $J\subset\inter(\supp\upsilon_1)$ we can by
  Proposition~\ref{pr:intervalocresce} find another interval
  $\hat J\subset J$ and $n>1$ so that
  $f^n\mid_{\hat J}:\hat J\to V(p_i)$ is a diffeomorphism to
  a neighborhood $V(p_i)$ of $p_i$, for some $1\le i\le
  k$. The invariance of $\supp\upsilon_1$ shows that
  $p_i\in\inter(\supp\upsilon_1)$ and $E\neq\emptyset$.

  We set $B=\{1,\dots,k\}\setminus E$ and show that
  $B=\emptyset$. For that, we write $i\to j$ if the preorbit
  of $\cO(p_j)$ accumulates $\cO(p_i)$.

  \begin{claim}\label{cl:toE}
    If $i\to j$ and $j\in E$, then $i\in E$.
  \end{claim}

  Hence orbits in $B$ cannot link to orbits in $E$. Since
  the union of the preorbits of $\cO(p_i)$
  are dense in $I$, then $B$ can only be
  accumulated by preorbits of elements of $B$. Thus, the
  union $W$ of the preorbits of the elements of $B$ is
  $f$-invariant and dense in a neighborhood of the orbits of
  the elements of $B$. Therefore, $\ov{W}$ is a compact
  $f$-invariant set with nonempty interior of $I$ and so
  $\ov{W}$ contains the support of some
  $\upsilon_j$. Consequently, the preorbit of some element
  of $B$ intersects $\inter(\supp\upsilon_j)$ and so $B\cap
  E\neq\emptyset$. This contradiction proves that $B$ must
  be empty, except for the proof of the claim.

  \begin{proof}[Proof of Claim~\ref{cl:toE}]
    There exists $x_n\xrightarrow[n\to\infty]{}p_i$ so that
    $x_n\in\cup_{m\ge0}f^{-m}p_j$ and then we can find $V_n$
    neighborhood of $x_n$ and $m_n>1$ such that
    $f^{m_n}\mid_{V_n}: V_n\to V(p_i)$ is a diffeomorphism
    onto a neighborhood $V(p_j)$ of $p_j$.

    But $V(p_j)\cap\supp\upsilon_h\neq\emptyset$ for some
    $1\le h\le l$ and so, by invariance of
    $\supp\upsilon_h$, there are points of $\supp\upsilon_h$
    in $V_n$, for all $n\ge1$. This shows that $p_i$ is a
    limit point of $\supp\upsilon_h$, and so $i\in E$. This
    proves the claim and finishes the proof of the lemma.
  \end{proof}

\end{proof}


\subsubsection{Stable manifolds of singularities}
\label{sec:stable-manifolds-sin}

We are now ready to obtain the property in item (6) of
Theorem~\ref{thm:propert-singhyp-attractor}.

\begin{theorem}
  \label{thm:densestablesing}
  The union of the stable manifolds of the singularities in
  a connected singular-hyperbolic attracting set is dense in
  the topological basin of attraction, that is
  \begin{align*}
    U_0\subset\ov{\bigcup_{\sigma\in\Lambda\cap S(G)}W^s(\sigma)}
  \end{align*}
\end{theorem}

We present a proof of Theorem~\ref{thm:densestablesing} in
Subsection~\ref{sec:density-stable-manif}.

One important consequence is the possibility of choosing
adapted cross-sections with a special feature crucial to
obtain item (6) of the statement of
Theorem~\ref{thm:propert-singhyp-attractor}.

\begin{corollary}
  \label{cor:singular-adapted-section}
  Every regular point of a connected singular-hyperbolic
  attracting set admits cross-sections with arbitrarily
  small diameter whose stable boundary is formed by stable
  manifolds of singularities of the set, for every small
  enough $\delta>0$.
\end{corollary}

\begin{proof}
  Since there is a dense subset of stable leaves in $U_0$
  that are part of $W^s(S)=\bigcup_{\sigma\in S}W^s(\sigma)$, we
  can choose a cross-section $\Sigma$ to $X$ at any point
  $x\in U_0$ with diameter as small as we like having a
  stable boundary contained in $W^s(S)$.
\end{proof}

\begin{remark}
  \label{rmk:discontsing}
  As a consequence of
  Corollary~\ref{cor:singular-adapted-section}, the
  one-dimensional map $f$ satisfies item (6) of
  Theorem~\ref{thm:propert-singhyp-attractor}. 

  Indeed, points in the stable manifold of a singularity
  $\sigma\in\Lambda$ are sent in finite positive time by the
  flow to the local stable manifold of the singularity in a
  cross-section close to the singularity. We can also ensure
  that orbits of such stable boundaries do not contain stable
  boundaries of other cross-sections. For the
  one-dimensional map the corresponding behavior is
  precisely given by item (6) of
  Theorem~\ref{thm:propert-singhyp-attractor}.

  Items (6a-6b) are consequences of the properties of $f$;
  see~\cite[Proof of Lemma 8.4]{ArMel18} for more
  details. Moreover, we can assume that all singular points
  of $f$ are related to Lorenz-like singularities, after
  Remark~\ref{rmk:no-recur-non-Lorenz}.
\end{remark}

\subsubsection{Density of stable manifolds of singularities}
\label{sec:density-stable-manif}

To prove Theorem~\ref{thm:densestablesing} we use
non-uniform hyperbolic theory through the following result.



\begin{theorem}
  \label{thm:stablepertransversalsing}
  Let $\mu$ be an ergodic $f$-invariant hyperbolic
  probability measure supported in a connected
  singular-hyperbolic attracting set $\Lambda$.

  Let us assume that $\mu$ is a $u$-Gibbs state, that
  is, for $\mu$-a.e. $x$ the unstable manifold $W^u_x$ is
  well-defined and Lebesgue-a.e. $y\in W^u_x$ is
  $\mu$-generic:
  $\frac1T\int_0^T\delta_{X_sy}\,ds\xrightarrow[T\to\infty]{w^*}\mu$.

  Then there exists $\sigma\in S$ such that
  $W^u(p)\pitchfork W^s(\sigma)\neq\emptyset$ for every
  $p\in\supp\mu$.
\end{theorem}

We prove Theorem~\ref{thm:stablepertransversalsing} in
Subsection~\ref{sec:transv-inters-betwee} and, based on this
result, we can now present the following.

\begin{proof}[Proof of Theorem~\ref{thm:densestablesing}]
  This theorem really is a corollary of
  Theorem~\ref{thm:stablepertransversalsing} since we
  already know that the stable manifolds of the periodic
  orbits $\cO(p_i)$ are dense in a neighborhood $U_0$ of
  $\Lambda$; see Remark~\ref{rmk:denstableF}(2). The
  transverse intersection provided by
  Theorem~\ref{thm:stablepertransversalsing} ensures,
  through the Inclination Lemma, that each of these stable
  manifolds is accumulated the stable manifold of some
  singularity, and the statement of
  Theorem~\ref{thm:densestablesing} follows.
\end{proof}

\subsubsection{Transversal intersection between unstable
  manifolds of periodic orbits and stable manifolds of
  singularities}
\label{sec:transv-inters-betwee}

The proof of Theorem~\ref{thm:stablepertransversalsing} is
based on a few results.

In what follows, we say that a disk
$\gamma\subset M$ is a (local) {\em strong-unstable
  manifold}, or a {\em strong-unstable manifold}, if
$\dist(X_{-t}(x),X_{-t}(y))$ tends to zero exponentially
fast as $t\to+\infty$, for every $x,y\in\gamma$. It is
well-know \cite{KH95,Palmer} that every point $x$ of a
hyperbolic periodic orbit $\cO_X(p)$ for a vector field $X$
admits a local strong-unstable manifold which is an embedded
disk tangent at $x\in\cO_X(p)$ to the unstable direction $E^u_x$.

Considering the action of the flow we get the (global)
\emph{strong-unstable manifold}
$$W^{uu}(x)=\bigcup_{t>0}X_{t}\Big(W^{uu}_{loc}\big(X_{-t}(x)\big)\Big)$$
for every point $x$ of a uniformly hyperbolic set: in
particular, for a hyperbolic periodic orbit $\cO(p)$ of the
flow of $X$.

In the present setting, since the singular-hyperbolic
attracting set $\Lambda$ has codimension $2$ and the central
direction $E^c_\Lambda$ contains the flow direction, then
every periodic orbit $\cO_X(p)$ in $\Lambda$ is hyperbolic and its
unstable direction is one-dimensional. Hence the
strong-unstable manifold through any point $x\in\cO_X(p)$ is
an immersed curve.

\begin{lemma}
  \label{le:unstablesing}
  In the setting of the statement of
  Theorem~\ref{thm:stablepertransversalsing}, fix
  $p_0\in\per(X)\cap\supp\mu$ and let $J=[a,b]$ be an arc on
  a connected component of $W^{uu}(p_0)\setminus\{p_0\}$
  with $a\neq b$.  Then $H=\ov{\cup_{t>0} X^t(J)}$
  contains some singularity of $\Lambda$.
\end{lemma}

\begin{proof}
  It is well-known from the non-uniform hyperbolic theory
  (Pesin's Theory) that the support of a non-atomic
  hyperbolic ergodic probability measure $\mu$ is contained in a
  homoclinic class of a hyperbolic periodic orbit $\cO(p)$; see
  e.g. \cite[Appendix]{KH95} or \cite{BaPe13}.

  Hence, for $\mu$-a.e. $x$ we have $W^u_x\subset\supp\mu$
  (since $\mu$ is a $u$-Gibbs measure) and
  $W^u_x\pitchfork W^s(\cO(p))\neq\emptyset$. Thus by the
  Inclination Lemma (see \cite{PM82}) we have
  $W^u(p)\subset\ov{W^u(x)}\subset\supp\mu$. 

  Since every periodic point $p_0\in\supp\mu$ is
  homoclinically related to $p$ (that is, $W^s(p)\pitchfork
  W^u(p_0)\neq\emptyset\neq W^s(p_0)\pitchfork W^u(p)$), then we
  also have $W^{uu}(p_0)\subset W^u(p_0)\subset\supp\mu$.

  Note that $H\subset\ov{W^u_0(p_0)}\subset\supp\mu$ and $H$
  is a compact invariant set by construction, where
  $W^u_0(p_0)$ is the connected component of
  $W^u(p_0)\setminus\cO(p_0)$ containing $J$. In addition,
  $H$ is clearly connected, since $H$ is also the closure of
  the orbit of the connected set $J$ under a continuous
  flow.

  If $H$ has no singularities, then $H$ is a compact
  connected hyperbolic set and so contains the
  strong-unstable manifolds through any of its points, since
  every point in $H$ is accumulated by forward iterates of
  the arc $J$. This means that $H$ is an attracting set and
  so $H=\Lambda$ by connectedness, and $H$ contains all
  singularities of $\Lambda$. This contradiction proves that
  $H$ must contain a singularity.
\end{proof}

Fix $p_0$ and $\sigma\in S\cap H$ as in the statement of
Lemma~\ref{le:unstablesing}.  We have shown that there
exists $\sigma\in\supp\mu\cap S$ so that
$\sigma\in\ov{W^u(p_0)}$. We assume that $J$ is a
fundamental domain for $W^u(p_0)$, that is, $b=X^{T}(a)$
with $T>0$ the first return time of the orbit of $a$ to
$W^{uu}(p_0)$, i.e., $X^t(a)\notin W^{uu}(p_0)$ for all
$0<t<T$. We now argue just as in \cite[Section 6.3.2, pp
199-202]{AraPac2010} and show that there exists some
singularity whose stable manifold transversely intersects
$J$.

This is enough to conclude the proof of
Theorem~\ref{thm:stablepertransversalsing}. Indeed, since
all periodic orbits in $\supp\mu$ are homoclinically
related, it is enough to obtain $W^u(p_0)\pitchfork
W^s(\sigma)\neq\emptyset$ for a  periodic point
$p_0\in\supp\mu$.

To complete the argument, since in \cite[Section
6.3.2]{AraPac2010} it was assumed that $\Lambda$ was either
a singular-hyperbolic attractor or attracting set with dense
periodic orbits for a $3$-vector field, we state \cite[Lemma
6.49]{AraPac2010} in our setting.

\begin{lemma}
  \label{le:fallsoffS}
  Let $\tilde\Sigma$ be a cross-section of $X$ containing a
  compact $cu$-curve $\zeta$, which is the image of a
  regular parametrization $\zeta:[0,1]\to\tilde\Sigma$, and
  assume that $\zeta$ is contained in $\supp\mu$. Let
  $\Sigma$ be another cross-section of $X$. Suppose that
  \emph{$\zeta$ falls off $\Sigma$}, that is
  \begin{enumerate}
  \item the positive orbit of $\zeta(t)$ visits
    $\inter(\Sigma)$ for all $t\in[0,1)$;
  \item and the $\omega$-limit of $\zeta(1)$ is disjoint
    from $\Sigma$.
  \end{enumerate}
  Then $\zeta(1)$ belongs either to the stable manifold of some
  periodic orbit $p$ in $\supp\mu$, or to the stable manifold
  of some singularity.
\end{lemma}

\begin{proof}
  Just follow the same arguments in the proof of \cite[Lemma
  6.49]{AraPac2010} since the proof assumes that stable
  manifolds of the flow intersected with cross-sections
  disconnect the cross-sections (that is, the transverse
  intersection is a hypersurface inside the cross-section);
  and either the existence of a dense regular orbit, or the
  denseness of periodic orbits, each of which is true in the
  invariant subset $\supp\mu$ in our setting.
\end{proof}


\section{Dimensional reduction of large deviations subset}
\label{sec:reduct-large-deviat}

Here we explain how to use the representation of the global
Poincar\'e map obtained in Subsection
\ref{sec:global-poincare-retu} to reduce the problem of
estimating an upper bound for the large deviations subset of
the flow to a similar problem for a expanding quotient map
on the base dynamics of a suspension semiflow, in the
setting of Theorem~\ref{thm:propert-singhyp-attractor}
\emph{assuming exponentially slow recurrence to the subset
  $\cD$} as in Theorem~\ref{mthm:expslowapprox}.

We start by representing the flow as a suspension semiflow
over the global Poincar\'e map constructed in
Section~\ref{sec:general-setting} to reduce the large
deviations subset of a continuous bounded observable to a
similar large deviations subset of an induced observable for
the dynamics of $F$ and its quotient $f$ over stable
leaves. Then we use the uniform expansion of $f$ and assume
exponentially slow recurrence to a singular subset to deduce
exponential decay of large deviations for continuous
observables on a neighborhood of the attracting set.

\subsection{Reduction to the global Poincar\'e map and
  quotient along stable leaves}
\label{sec:reduct-global-poinca}

Let $\phi^t:\hat Q^\tau\to \hat Q^\tau$ denote the
suspension semiflow with roof function $\tau$ and base
dynamics $F$, where $F$ and $\tau$ satisfying the properties
stated in Theorem~\ref{thm:propert-singhyp-attractor}.

More precisely, we assume that
\begin{description}
\item[(P1) $\tau$ grows as $|\log \dist(\cdot,\cD)|$] the
  roof function $\tau$ has logarithmic growth near $\cD$; is
  uniformly bounded away from zero $\tau\ge\tau_0>0$;
\end{description}
and set
$\hat Q^\tau=\{ (x,y)\in Q\times[0,+\infty): 0\le y <
\tau(x), x_n=F^n(x)\notin\tilde\Gamma, \forall n\ge1\}$.
Then for each pair $(x_0,s_0)\in\hat Q^\tau$ and $t>0$ there
exists a unique $n\ge1$ such that
\begin{align*}
S_n \tau(x_0) \le s_0+ t < S_{n+1} \tau(x_0).
\end{align*}
 We are now ready to define
\begin{align*}
   \phi^t(x_0,s_0) = \big(x_n,s_0+t-S_n \tau(x_0)\big),\quad
  (x_0,s_0)\in\hat Q^r, t\ge0.
\end{align*}
For each $F$-invariant physical measure
$\mu^i_F, i=1,\dots,k$ from
Theorem~\ref{thm:propert-singhyp-attractor}, we denote by
$\mu^i=\mu^i_F\ltimes\lambda$ the natural $\phi^t$-invariant
extension of $\mu_F^i$ to $\hat Q^\tau$ and by
$\lambda^\tau$ the natural
extension of $\leb$ induced on $Q$ to $\hat Q^\tau$,
i.e. $\lambda^\tau=\leb\ltimes\lambda$, where $\lambda$ is
one-dimensional Lebesgue measure on $\RR$: for any subset
$A\subset\hat Q^\tau$ and $\chi_A$ its characteristic function
\begin{align*}
  \mu^i(A)&=\frac1{\mu_F^i(\tau)}\int d\mu_F^i(x)
\int_0^{\tau(x)}\!\!\!\! ds \,\chi_A(x,s),
            \quad\text{and}
  \\
  \lambda^\tau(A)&=\frac1{\leb(\tau)}\int d\leb(x)
\int_0^{\tau(x)}\!\!\!\! ds \,\chi_A(x,s).
\end{align*}
From the construction of $F$ from the proof of
Theorem~\ref{thm:propert-singhyp-attractor} we see that the
map $\Psi:\hat Q^\tau\to M, (x,s)\mapsto X^s(x)$ is a
finite-to-$1$ locally $C^2$ smooth semiconjugation
$\Psi\circ\phi^t= X^t\circ\Psi$ for all $t>0$ so that we can
naturally identify $\Psi_*(\mu^i)=\mu_i$, where $\mu_i$ are
the physical measures supported on the singular-hyperbolic
attracting set given by
Theorem~\ref{thm:appv-attracting}. In particular we get
$\Psi_*(\lambda^\tau)\le \ell\cdot\leb$ where $\ell$ is the
maximum number of preimages of $\Psi$.

\subsubsection{The quotient map}
\label{sec:quotient-map}

Let $Q$ be a compact metric space, $\Gamma\subset Q$ and
$F:(Q\setminus\Gamma)\to\Xi$ be a measurable map.  We assume
that there exists a partition $\F$ of $Q$ into measurable
subsets, having $\Gamma$ as the union of a collection of
atoms of $\F$, which is
\begin{description}
\item[(P2) {\em invariant\/}] the image of any $\xi\in\F$
  not in $\Gamma$ is contained in some element $\eta$
  of $\F$;
\item[(P3) {\em contracting\/}] the diameter of $F^n(\xi)$
  goes to zero when $n\to\infty$, uniformly over all the
  $\xi\in\F$ for which $F^n(\xi)$ is defined.
\end{description}
We denote $p:Q\to \F$ the canonical projection, i.e. $p$
assigns to each point $x\in Q$ the atom $\xi\in\F$ that
contains it.  By definition, $A\subset \F$ is measurable if
and only if $p^{-1}(A)$ is a measurable subset of $Q$ and
likewise $A$ is open if, and only if, $p_\Sigma^{-1}(A)$ is
open in $Q$.  The invariance condition means that there is a
uniquely defined map
$$
f:(\F\setminus\{\Gamma\}) \to \F
\quad\text{such that}\quad
f\circ p = p \circ F.
$$
Clearly, $f$ is measurable with respect to the measurable
structure we introduced in $\F$.  We assume from now on that
the leaves are sufficiently regular so that
\begin{description}
\item[(P4) regular quotient] the quotient $M=Q/\F$ is a
  compact finite dimensional manifold with the topology
  induced by the natural projection $p$ and
  $\lambda=p_*\leb$ is a finite Borel measure.
\end{description}

It is well-known (see e.g. \cite[Section 6]{APPV}) that each
$F$-invariant probability measure $\mu_F$ is in one-to-one
correspondence with the $f$-invariant probability measure
$\mu_f$ by $p_*\mu_F=\mu_f$ and this map preserves
ergodicity. We also need
\begin{description}
\item[(P5) uniform expansion and non-degenerate singular
  set] the quotient map $f$ is uniformly expanding: there
  are $\sigma>2$ and $q\in\ZZ^+, q\ge2$ so that $f$ is
  expanding with rate $\|Df^{-1}\|<\sigma^{-1}$ and
  number of pre-images of a point (degree) bounded by $q$;
  also $p(\Gamma)$ is a non-degenerate singular set for $f$.

\item[(P6) integrability]
  \begin{enumerate}
  \item $\tau$ satisfies condition (1b) of
    Theorem~\ref{thm:propert-singhyp-attractor} and so there
    exists $\kappa_0>0$ so that
    $\hat\tau(p(x))=\sup\{\tau(y): y\in Q, p(y)=p(x)\}$
    satisfies $|\tau(x)-\hat\tau(p(x))|\le\kappa_0, x\in Q$ and 
    $\hat\tau$ is both $\lambda$-integrable and
    $\mu_f$-integrable;
  \item $\tau$ is $\leb$-integrable and $\mu_F$-integrable
    for any $F$ invariant probability measure $\mu_F$ such
    that $p_*\mu_F$ is absolutely continuous with respect to
    $\lambda$.
  \end{enumerate}
\item[(P7) measure of singular neighborhoods] there exists
  $d,C_d>0$ so that
  $\leb\{x\in M: \dist(x,\cD)<\rho\} \le C_d \rho^d$, for
  all small $\rho>0$.
\end{description}

In our singular-hyperbolic setting, we have $d=1$ in (P7).

Moreover, we identify the equilibrium states $\EE$ for
$\log J_1^{cu}$ with $\Psi_*\EE$. In addition, the ergodic
physical/SRB measures that are the extremes points of $\EE$
are naturally induced uniquely by ergodic physical measures
for $F$ which, in turn, are also related to a unique
absolutely continuous ergodic invariant probability measure
for $f$. We denote in what follows $\EE_F$ and $\EE_f$ to be
the convex hull of these ergodic measures with respect to
$F$ and $f$, respectively; and note that $p_*\EE_F=\EE_f$.

\subsubsection{Exponentially slow recurrence for the
  suspension flow}
\label{sec:exponent-slow-recurr-2}

In the rest of this section we prove the following.

\begin{theorem}
  \label{thm:LD-suspension_F}
  Let $\phi^t:\Xi_0^\tau\to\Xi_0^\tau$ be the suspension
  semiflow with roof function $\tau$ and base dynamics $F$,
  where $F$ and $\tau$ satisfy conditions (P1)-(P7) stated
  above.  Let the quotient map $f$ have exponentially slow
  recurrence to the finite subset $\cD$; set $\EE$ to be the
  family of all measures that are sent into equilibrium
  states of $X$ for $\log J_1^{cu}$ on $\Lambda$; and let
  $\psi:\Xi^\tau\to\RR$ be a bounded uniformly continuous
  observable. Then, for any given $\epsilon>0$
  \begin{align*}
    \limsup_{T\to+\infty}\frac1T\log\lambda^\tau
    \left\{z\in\Xi^\tau:
    \inf_{\mu\in\EE}\left|\frac1T\int_0^T\psi(\phi^t(z))
    \,dt - \mu(\psi)\right|>\epsilon\right\} < 0;
  \end{align*}
\end{theorem}
This result proves Theorem~\ref{mthm:LD-sing-hyp-attracting}
as soon as we prove exponentially slow recurrence to $\cD$
for the quotient base map $f$: this is
Theorem~\ref{mthm:expslowapprox} to be proved in
Section~\ref{sec:exponent-slow-recurr-1}

\begin{remark}
  \label{rmk:finiteD}
  We assumed that $\cD$ is finite in several places along
  the following argument.
  This is a natural assumptiom for the quotient map
  induced from singular-hyperbolic attracting sets.
\end{remark}

The proof of this result is based on the observation that, 
for a continuous function $\psi:\Xi^\tau\to\RR, T>0,
z=(x,s)\in\Xi^\tau$ we have
\begin{align*}
  \int_{0}^{T}\psi(\phi^{t}(z))\,dt
  &=
    \int_{s}^{\tau(x)}\psi(\phi^{t}(x,0))\,dt+
    \sum_{j=1}^{n-1}\int_{0}^{\tau(F^{j}(x))}\psi(\phi^{t}(F^{j}(x),0))\,dt
  \\
  & +\int_{0}^{T+s-S_{n}\tau(x)}\psi(\phi^{t}(F^{n}(x),0))\,dt
\end{align*}
where $n=n(x,s,T)$ is the lap number so that
$0\le T+s-S^F_n\tau(x)<\tau(F^n(x))$.  So setting
$\varphi(x)=\int_{0}^{\tau(x)}\psi(\phi^{t}(x,0))\,dt$ we
obtain
\begin{align*}
  \frac{1}{T}\int_{0}^{T}\psi(\phi^{t}(z))\,dt
  &=
    \frac{1}{T}S^F_n\vfi(x)+I(x,s,T)
\end{align*}
where
\begin{align*}
  I=I(x,s,T)
  =\frac{1}{T}\left(\int_{0}^{T+s-S_{n}\tau(x)}\psi(\phi^{t}(F^{n}(x),0))\,dt
  -\int_{0}^{s}\psi(\phi^{t}(x,0))\,dt\right)
\end{align*}
can be bounded as follows, with $\|\psi\|=\sup|\psi|$
\begin{align*}
  I\leq\left(2\frac{s}{T}+
  \frac{S_{n+1}^{F}\tau(x)-S_{n}^{F}\tau(x)}{T}\right)\cdot\|\psi\|.
\end{align*}
Hence, given $\omega>0$ for $0<s<\tau(x)$ and $n=n(x,s,T)$
the subset
\begin{align}\label{eq:LDsemiflow}
  \left\{ (x,s)\in\Xi^\tau :
  \inf_{\mu\in\EE}
  \left|
  \frac{1}{T}S_{n}^{F}\varphi(x)+I
  -\frac{\mu(\varphi)}{\mu(\tau)}
  \right|>\omega\right\} 
\end{align}
is contained in the union
\begin{align}\label{eq:red-din-base-1}
  \left\{ (x,s)\in\Xi^\tau:
  \inf_{\mu\in\EE}
  \left|\frac{1}{T}S_{n}^{F}\varphi(x)
  -\frac{\mu(\varphi)}{\mu(\tau)}\right|
  >
  \frac{\omega}{2}\right\}
  \bigcup
  \left\{ (x,s)\in\Xi^\tau:I>\frac{\omega}{2}\right\}.
\end{align}
Assuming that $\psi\neq0$ (otherwise we consider only
the right hand side of \eqref{eq:red-din-base-1}) we estimate
the $\lambda^\tau$-measure of each subset
in~\eqref{eq:red-din-base-1} showing that they are deviations
sets for the dynamics of $F$.

We note that assumption (P6) (consequence of
Theorem~\ref{thm:propert-singhyp-attractor}(1b)), ensures
that
\begin{align}
  I(x,s,T)
  &\leq
  \frac1T(2s+\tau(F^{n}x))\|\psi\|
  \le\nonumber
  \frac1T(2s+\tau(f^{n}(p(x))))\|\psi\| +\frac{\kappa_0}{T}\|\psi\|
  \\
  &\le\label{eq:Ibound}
    \|\psi\|\frac{S_{n+1}^f\tau -S_n^f\tau}T\circ p (x)
    + \frac{2s+\kappa_0\|\psi\|}T;
\end{align}
which shows that $I(x,s,T)$ is bounded by an expression
depending essentially on the dynamics of $f$.

Now the left hand side subset of \ref{eq:red-din-base-1} is
contained in
\begin{align}
  \label{eq:omega4-1}
  \left\{ (x,s)\in\Xi^{\tau}:
  \inf_{\mu\in\EE}\left|
  \frac{n}{T}\left(\frac{S_{n}^{F}\varphi}{n}
  -
  \mu(\varphi)\right)\right|>\frac{\omega}{4}\right\}
  \cup
  \left\{ (x,s)\in\Xi^{\tau}:\inf_{\mu\in\EE}
  \left|\frac{n}{T}-\frac{1}{\mu(\tau)}\right|
  >\frac{\omega}{4\left|\mu(\varphi)\right|}\right\} 
\end{align}
since for each $\mu\in\EE$ we have
\begin{align*}
  \left|\frac{1}{T}S_{n}^{F}\varphi-\frac{\mu(\varphi)}{\mu(\tau)}\right|
  &\leq
    \left|
    \frac{n}{T}\frac{S_{n}^{F}\varphi}{n}-\frac{n}{T}\mu(\varphi)
    \right|
    +
    \left|\frac{n}{T}\mu(\varphi)-\frac{\mu(\varphi)}{\mu(\tau)}\right|
  \\
  &\leq
    \frac{n}{T}\left|\frac1n S_{n}^{F}\varphi-\mu(\varphi)\right|+
    \mu(\varphi)\left|\frac{n}{T}-\frac{1}{\mu(\tau)}\right|,
\end{align*}
and the lap number $n=n(x,s,T)$ satisfies
\begin{align*}
  \frac{S_{n}^{f}\tau(p(x))}{n}-c
  \leq
  \frac{S_{n}^{F}\tau(x)}{n}\leq\frac{T+s}{n}
  <
  \frac{S_{n+1}^{F}\tau(x)}{n}
  \leq
  \frac{S_{n+1}^{f}\tau(p(x))}{n}+\frac{n+1}{n}c.
\end{align*}
Therefore, bounds involving $n(x,s,T)/T$ can be replaced by
others involving ergodic sums
$\frac{S_{n}^{f}\tau(p(x))}{n}$ and hence we reduce its
study to the dynamics of the one-dimensional map $f$. We
deal with the sums $S_{n}^{F}\vfi$ in the next
Subsection~\ref{sec:reduct-one-dimens} and with the sums
$S_n^f\tau$ in Subsection~\ref{sec:roof-function-as}.

\subsection{Reduction to the quotiented base dynamics}
\label{sec:reduct-one-dimens}

Here we use the contracting foliation that covers the
cross-sections $\Xi$ to show that large deviations of an
induced observable for the dynamics of $F$ can be reduced to
a similar property for the dynamics of the quotient map
$f$. Then we show how this large deviation bound for $f$
follows assuming exponentially slow recurrence to $\cD$.

\begin{proposition}\label{pr:reduction-1d}
  Let $\epsilon>0$ and a continuous and bounded
  $\psi:U\to\RR$ be given on the trapping neighborhood $U$
  of $\Lambda$ and set $\vfi:\Xi_{0}\to\RR$ as
  $\vfi(z)=\int_{0}^{\tau(z)}\psi(X^{t}(z))\,dt$, where
  $\tau(z)$ is the Poincar\'e time of $z\in\Xi_{0}$. Let
  $\mu$ be a measure on $\Xi$ such that
  $\int|\vfi|\,d\mu<\infty$. If we assume that there are
  $\sigma>2$ and $q\in\ZZ^+$ so that
  \begin{itemize}
  \item the quotient map $f:M\setminus\cD\to M$ is a $C^1$
    local diffeomorphism away from the finite subset $\cD$
    of the finite-dimensional compact manifold $M$,
  \item $f$ is expanding with rate $\|Df^{-1}\|<\sigma^{-1}$
    and the number of pre-images of a point is at most $q$,
  \end{itemize}
then there exist
  $N,k\in\ZZ^+$, $\delta>0$, a constant $\gamma>0$ depending
  only on $\psi$ and the flow, and a continuous function
  $\xi: M\setminus\cup_{j=0}^{k-1}f^{-j}\cD\to\RR$ with
  logarithmic growth near $\cD_k=\cup_{j=0}^{k-1}f^{-j}\cD$
  such that, for all $n>N$
\[
  \left\{ \left|\frac1n S_{n}^{F^k}\vfi-\mu(\vfi)\right|
    >3\epsilon\right\}
  \subset
  p^{-1}\left(\left\{ \frac1n
      S_{n}^{f^k}\Delta_{\delta}>\frac{\epsilon}{\gamma}\right\}
    \cup
    \left\{ \left|\frac1n
        S_{n}^{f^k}\xi-\mu(\xi)\right|>\epsilon\right\}
  \right).
\]
\end{proposition}

This shows that it is enough to obtain an exponential decay
for large deviations for observables with logarithmic growth
near $\cD$ if we are able to obtain such exponential decay
for a power of $f$ together with exponentially slow
recurrence to $\cD$.

Indeed, for $n=k\ell+m$ with
$0\le m<k$ and all big enough $\ell\in\ZZ^+$
\begin{align*}
  \left\{ \left|\frac{S_{n}^{F}\vfi}n -\mu(\vfi)\right|
  >(4k+1)\epsilon\right\}
  \subset
  \bigcup_{i=0}^k
  \left\{
  \left|\frac{S_{\ell}^{F^k}\vfi}{\ell} \circ
  F^{m+i}-\mu(\vfi)
  \right| >3\epsilon\right\}
  \cup
  \left\{\left|\frac{S_m^F\vfi}{k\ell+m}\right|>\epsilon\right\}.
\end{align*}
Then the Lebesgue measure of the right hand side subset can
be bounded using that: $0\le m<k$, $\vfi$ has logarithmic
growth near the \emph{finite subset} $\cD$ and
\begin{align*}
  \left\{\left|\frac{S_m^F\vfi}{k\ell+m}\right|>\epsilon\right\}
  \subset
  p^{-1}\left(\bigcup_{i=0}^{k-1}
  f^{-i}\left\{x\in M: d(x,\cD)<\frac{e^{-k\ell\epsilon}}K\right\}\right)
\end{align*}
so that, since $p_*\leb=\lambda$ and by Remark
\ref{rmk:lebngbhdD}
\begin{align*}
  \leb\left\{\left|\frac{S_m^F\vfi}{k\ell+m}\right|>\epsilon\right\}
  \le
  \sum_{i=0}^{k-1}\left(\frac{q}{\sigma^d}\right)^i
  \frac{e^{-kd\ell\epsilon}}K
  \le
  \frac{\sigma^d}{q-1}
  \left(\frac{q}{\sigma^d}\right)^k\frac{e^{-kd\ell\epsilon}}K,
\end{align*}
where $\sigma>1$ is the least expansion rate of $f$, $q$ is
the maximum number of pre-images of the map $f$ and $d$ is
the dimension of the quotient manifold $M$. For the
remaining union of subsets we obtain
\begin{align*}
  \leb\left(
  \bigcup_{i=0}^k
  \left\{
  \left|\frac{S_{\ell}^{F^k}\vfi}{\ell} \circ
  F^{m+i}\!\!-\mu(\vfi)
  \right| >3\epsilon\right\}
  \right)
  \le
  k 
  \left(\frac{q}{\sigma^d}\right)^{2k}
  \!\!\lambda\left(
    \left\{
  \left|\frac{S_{\ell}^{F^k}\vfi}{\ell} \circ
  F^{m+i}\!\!-\mu(\vfi)
  \right| >3\epsilon\right\}
  \right).
\end{align*}
Thus, from the statement of
Proposition~\ref{pr:reduction-1d}, we are left to study
upper large deviations for continuous observables with
logarithmic growth near $\cD$ and exponentially slow
recurrence to $\cD$ for a power of $f$.

\begin{proof}[Proof of Proposition~\ref{pr:reduction-1d}]
  First note that since $\vfi$ is continuous on $\Xi_0$ and
  $\psi$ is bounded on $U$ we get
  $\vfi(x)\le \tau(x)\cdot\sup|\psi|\le
  K\Delta_\delta(p(x))\cdot\sup|\psi|$ for
  $x\in B(\cD,\delta)$, some small enough $\delta>0$ and
  $K=K(\vfi)>0$, since the return time function has
  logarithmic growth near the singular set $\Gamma$.

  From the assumptions (P1)-(P7) we can write $F$ as a
  skew-product as in
  Theorem~\ref{thm:propert-singhyp-attractor} and so
  $\dist(F^k(x,y),F^k(x,y'))< \lambda^k$ for all
  $1\le k\le n$ and points in the same stable leaf of $F$,
  where $n$ is the first time the points visit the singular
  lines $\Gamma$. These times $n$ are given by $f^nx\in\cD$
  and since $X_0=\cup_{n\ge1}f^{-n}\cD$ is enumerable the
  set of points which can be iterated indefinitely by $F$
  has full Lebesgue measure in $Q$.

  Moreover, there exists a constant $\kappa_1>0$ so that
  $\dist(X^t(x,y),X^t(x,y'))\le \kappa_1|y-y'|$ for all
  $t>0$, since stable leaves of $F$ correspond to curves
  contained in central stable leaves of the flow $X^t$, by
  construction of $W^s_x(\Sigma_i)$ in
  $\Sigma_i\in\Xi$. Indeed, central stable leaves are given
  by $W^{cs}(x)=\cup_{t\in\RR}W^s(X^tx)$ and so there exists
  $\delta$ close to $0$ such that $(x,y')$ can be identified
  with a point $z\in W^s(X^\delta x)$ and
  $\delta\approx|y-y'|$. Hence the distance between $X^tz$
  and $X^tx$ is comparable with the distance between $x$ and
  $X^\delta x$.

  Altogether this ensures the bound
  \begin{align*}
    \left|
    \frac1n\sum_{j=0}^{n-1} \left(\vfi(F^{j}(x,y))
    -
    \vfi(F^{j}(x,y'))\right)
    \right|
    \leq
    \frac1n\sum_{j=0}^{n-1}\bar\vfi_j(x)
  \end{align*}
  where
  $\bar\vfi_j(x) = \sup_{y,y'\in
    W^{s}(x,\Xi)}|\vfi(F^j(x,y))-\vfi(F^j(x,y'))|$.

  For each $\epsilon>0$ there exist $\delta,\eta>0$ such
  that $-K\eta\log\delta<\epsilon/3$ and
  $\eta\|\psi\|<\epsilon/3$ and, using uniform continuity,
  we can also find $\zeta>0$ satisfying
  $\dist(y,y')<\zeta\implies
  |\psi(x,y)-\psi(x,y')|<\eta$. Hence, by the uniform
  contraction of stable leaves by $F$ and because we can
  assume without loss of generality that
  $\zeta/\kappa_1<\eta/\kappa$, there exists
  $j_0=j_0(\eta)\in\ZZ^+$ so that
  $|F^{j}(x,y)-F^{j}(x,y')|
  \leq\frac{\zeta}{\kappa_1}<\frac{\eta}\kappa, \forall j\ge
  j_0$. Thus, by the previous choices together with item
  (1b) from Theorem~\ref{thm:propert-singhyp-attractor}, we
  get
  \begin{align}
    \bar\vfi_j(x)
    &\le
      K\max\{\Delta_{\delta}(x_j),\log\delta^{-1}\}
      \sup_{0<t<K\Delta_{\delta}(x_j)}
      |\psi(X^t(F^j(x,y)))-\psi(X^t(F^j(x,y'))))|\nonumber
    \\
    &\quad +
      |\tau(F^j(x,y))-\tau(F^j(x,y'))|\cdot\|\psi\|\nonumber
    \\
    &\le
      K\max\{\Delta_{\delta}(x_j),\log\delta^{-1}\}\eta
      +\kappa\frac{\zeta}{\kappa_1}\|\psi\|
      \le
      K\Delta_{\delta}(x_j)\eta+\frac23\epsilon, \label{eq:barvfi}
  \end{align}
  where we write $x_j=f^j(x), j\ge0$. Now take a continuous function
  $\xi:M\setminus\cD\to\RR$ such that for some
  $0<a<\epsilon/3$
  \begin{itemize}
  \item $\int \xi\circ p \, d\mu = \int\vfi\,d\mu$;
  \item $\min_{y\in W^s(x,\Xi)}\vfi(F^{j_0}(y))-a
    \le\xi(x)\le a+\max_{y\in W^s(x,\Xi)}\vfi(F^{j_0}(y))$.
  \end{itemize}
  This is possible since $\vfi$ is $\mu$-integrable and
  disintegrating $\mu$ on the measurable partition of $Q$
  given by the stable leaves we obtain the family
  $(\mu_x)_{x\in M}$ of conditional probabilities 
  and we set $\xi_0(x)=\int \vfi\,d\mu_x$. Then we
  approximate $\xi_0$ by a continuous function $\xi_1$
  satisfying $\int|\xi_0-\xi_1|\circ p\,d\mu<\epsilon/3$ and
  so for some $b\in(-\epsilon/3,\epsilon/3)$ the function
  $\xi=\xi_1+b$ satisfies the above items.

  Now note that $\xi$ also has logarithmic growth near
  $\cup_{i=0}^{j_0-1}f^{-i}\cD$. In addition, for
  $n\in\ZZ^+$ using \eqref{eq:barvfi} and
  $f\circ p =p\circ F$ and summing over orbits of $f^{j_0}$
  and $F^{j_0}$ we get
  \begin{align}
    |S_{n}^{F^{j_0}}(\xi\circ p)-S_{n}^{F^{j_0}}\vfi|(x,y)
    &\leq
      |\xi\circ p -\vfi|(x,y)+|S_{n-1}^{F^{j_0}}(\xi\circ
      p-\vfi)|(x,y) \nonumber
    \\
    &\leq \nonumber
      K\eta\Delta_{\delta}(x)+\frac23\epsilon+ a+
      \sum_{i=1}^{n-1}
      \left(K\eta\Delta_{\delta}(f^{i j_0}(x))+a+\frac23\epsilon\right)
    \\
    &\leq \label{eq:zeta0}
      n\epsilon+K\eta\cdot S_{n}^{f^{j_0}}\Delta_{\delta}(x)
  \end{align}
  We finally observe that
  \begin{align}\label{eq:zeta2}
    \big\{ \big|\frac{S_n^{F^{j_0}}\vfi}n -\mu(\vfi)\big|>3\epsilon\big\}
    \subseteq
    \big\{ \big|\frac{S_n^{F^{j_0}}(\xi\circ p)-S_n^{F^{j_0}}\vfi}n
    \big| > 2\epsilon \}
    \cup
    \big\{ \frac1n\big| S_n^{R^{j_0}}(\xi\circ p)-\mu(\vfi)\big|>\epsilon\big\}.
  \end{align}
and by \eqref{eq:zeta0} we obtain
\begin{align*}
  \big\{
  \big| \frac1n\big( S_n^{F^{j_0}}(\xi\circ P)-S_n^{F^{j_0}}\vfi\big)
    \big| > 2\epsilon \}
    \subseteq p^{-1}
    \big\{ \frac1n S_{n}^{f^{j_0}}\Delta_\delta > 
    \frac{\epsilon}{K\eta} \big\}
\end{align*}
which together with \eqref{eq:zeta2} completes the proof of
the proposition with $k=j_0$.
\end{proof}

\subsubsection{Large deviations for observables with
  logarithmic growth near singularities}
\label{sec:large-deviat-observ}

This is based in \cite[Section 3]{araujo2006a} adapted to
the setting where there might be several equilibria for the
the potential $\log J_1^{cu}=\log|\det DX^1\mid E^{cu}|$ on
$\Lambda$.

The main bound on large deviations for suspension semiflows
over a non-uniformly expanding base will be obtained from
the following large deviation statement for non-uniformly
expanding transformations \emph{assuming exponentially slow
  recurrence to the singular/discontinuous set}.

\begin{theorem}
  \label{thm:LDNUElog}
  Let $f:M\setminus\cD\to M$ be a regular\footnote{A map is
    regular if $f_*\leb\ll\leb$, that is, $\leb$-null sets
    are not images of positive $\leb$-measure subsets.}
  $C^{1+\alpha}$ local diffeomorphism, where $\cD$ is a
  non-flat critical set and $\alpha\in(0,1)$.  Assume that
  $f$ is a non-uniformly expanding map with
  \emph{exponentially slow recurrence to the
    singular/discontinuous set $\cD$} and let
  $\vfi:M\setminus\cD\to\RR$ be a continuous map which has
  logarithmic growth near $\cD$. Moreover, assume that the
  family of ergodic equilibrium states with respect to
  $\log J$ is finite, where $J=|\det Df|$, and each of them
  is an absolutely continuous $f$-invariant probability
  measure. Then for any given $\omega>0$
  \begin{align*}
    \limsup_{n\to+\infty}\frac1n\log\leb\Big\{ x\in M:
    \inf_{\mu\in\EE}\left|\frac1n
      S_n\vfi(x)-\mu(\vfi)\right|\ge\omega \Big\}<0,
\end{align*}
where $\EE$ is the family of all equilibrium states with
respect to $\log J$.
\end{theorem}

\begin{remark}
  \label{rmk:NUE}
  \begin{enumerate}
  \item Since we assume in Theorem~\ref{thm:LD-suspension_F}
    that $f$ has exponentially slow recurrence to the
    non-degenerate singular set $\cD$ and is also expanding,
    then $f$ is in particular non-uniformly expanding with
    slow recurrence to $\cD$.
  \item The statement of Theorem~\ref{thm:LDNUElog} and its
    proof does not assume that $f$ is a one-dimensional map:
    this reduction holds for local diffeomorphisms away from
    a singular subset of a compact manifold.
  \end{enumerate}
\end{remark}

This finishes the reduction of the estimate of the Lebesgue
measure of the large deviation subset~\eqref{eq:LDsemiflow}
to obtaining exponentially slow recurrence to $\cD$ as in
Theorem~\ref{mthm:expslowapprox}, through the inclusion
\eqref{eq:red-din-base-1}, Proposition~\ref{pr:reduction-1d}
and Theorem~\ref{thm:LDNUElog}.

\begin{proof}[Proof of Theorem~\ref{thm:LDNUElog}]
  Fix $\vfi:M\setminus\cS\to\RR$ as in the statement,
  $\epsilon_0>0$ and $c\in\RR$. By assumption we may choose
  $\epsilon_1,\delta_1>0$ small enough such that the
  exponential slow recurrence condition
  \eqref{eq.expslowrecurrence} is true for the pair
  $(\epsilon_1,\delta_1)$,
  $|\vfi\chi_{B(\cS,\delta_1)}|\le K(\vfi)\Delta_{\delta_1}$
  and $K(\vfi)\cdot\epsilon_1\le\epsilon_0$, where $K(\vfi)$
  is the constant given by the assumption of logarithmic
  growth of $\vfi$ near $\cS$.

  Let $\vfi_0:M\to\RR$ be the continuous extension of
  $\vfi\mid_{B(\mathcal{S},\delta_1)^c}$ given by the Tietze
  Extention Theorem, that is
  \begin{itemize}
  \item $\vfi_0$ is continuous;
    $\vfi_0\mid_{M\setminus B(\mathcal{S},\delta_1)} =
    \varphi\mid_{M\setminus B(\mathcal{S},\delta_1)}$, and
  \item
    $\sup_{x\in M}|\varphi_0(x)| = \sup_{x\in M\setminus
      B(\mathcal{S},\delta_1)}|\varphi(x)|$.
  \end{itemize}
We may choose
  $K\ge K(\vfi)$ big enough so that and
  $|(\vfi-\vfi_0)\chi_{B(\mathcal{S},\delta_1)}|  \leq K
  \Delta_{\delta_1}$.  Then for all $n\ge1$ we have
\begin{align*}
S_n\vfi_0- S_n\big|\vfi-\vfi_0\big| \le
S_n\vfi=S_n\vfi_0+S_n(\vfi-\vfi_0) \le S_n\vfi_0 +
S_n\big|\vfi-\vfi_0\big|.
\end{align*}
and deduce the following inclusions
\begin{align}
  \left\{\frac1n S_n\vfi > c\right\}
  &\subseteq
  \left\{\frac1n S_n\vfi_0 > c- \epsilon_0\right\}
  \cup
  \left\{\frac1n S_n\Delta_{\delta_1}\ge\epsilon_1\right\},
  \label{eq:approxlog1}
\end{align}
where in \eqref{eq:approxlog1} we use the assumption that
$\vfi$ is of logarithmic growth near $\cS$ and the choices
of $K,\epsilon_1,\delta_1$. Analogously we get with opposite
inequalities
\begin{align}
  \left\{\frac1n S_n\vfi < c \right\}
  &\subseteq
  \left\{\frac1n S_n\vfi_0 < c + \epsilon_0\right\}
  \cup
  \left\{\frac1n S_n\Delta_{\delta_1}\ge\epsilon_1\right\};
  \label{eq:approxlog2}
\end{align}
see \cite[Section 4, pp 352]{araujo2006a} for the derivation
of these inequalities.

From \eqref{eq:approxlog1} and \eqref{eq:approxlog2} we see
that \emph{to get the bound for large deviations in the statement
of Theorem~\ref{thm:LDNUElog} it suffices to obtain a large
deviation bound for the continuous function $\vfi_0$ with
respect to the same transformation $f$} and \emph{to have
exponentially slow recurrence to the singular set $\cS$}.

To obtain this large deviation bound, we use the following
result from~\cite{araujo-pacifico2006}.

\begin{theorem}{\cite[Theorem B]{araujo-pacifico2006}}
  \label{thm:LDNUE}
  Let $f:M\setminus\cD\to M$ be a local diffeomorphism outside a
  non-flat singular set $\cD$ which is non-uniformly
  expanding and has slow recurrence to $\cD$.  For
  $\omega_0>0$ and a
  continuous function $\vfi_0:M\to\RR$ 
  there exists $\epsilon,\delta>0$ arbitrarily close to $0$
  such that, writing
\begin{align*}
A_n=\{x\in M: \frac1nS_n\Delta_\delta(x)\le\epsilon\}
\quad\text{and}\quad
B_n=\left\{
x\in M : 
\inf_{\mu\in\EE}\left|
\frac1n S_n\vfi_0(x) - \mu(\vfi_0)
\right|
> \omega_0
\right\}
\end{align*}
we get $
\limsup_{n\to+\infty}\frac1n
\log \leb\big(A_n\cap B_n\big) <0.$
\end{theorem}
Recall that
$\EE=\EE_{\epsilon,\delta}=\{\nu\in\M_f : h_\nu(f)=\nu(\log
J)$ and $\nu(\Delta_\delta)<\epsilon\}$ is the set of all
equilibrium states of $f$ with respect to the potential
$\log J$ which have slow recurrence to $\cD$. From
\cite[Theorem 5.1]{araujo2006a} we have that $\EE$ is a
non-empty compact convex subset of the set of invariant
probability measures, in the weak$^*$ topology.

Note that exponentially slow recurrence implies
$\limsup_{n\to+\infty}\frac1n\leb(M\setminus A_n) <0$. Under
this assumption Theorem~\ref{thm:LDNUE} ensures that for
$(\epsilon,\delta)$ close enough to $(0,0)$ we get
$\limsup_{n\to+\infty}\frac1n \log \leb( B_n) <0.$

Now in Theorem~\ref{thm:LDNUE} we take $\omega,\epsilon_0>0$
small, choose $\vfi_0$ as before and
$\omega_0=\omega+\epsilon_0$. Hence
$\{\mu(\vfi_0):\mu\in\EE\}$ is a compact interval of the
real line.

In \eqref{eq:approxlog1} set
$c=\inf_{\mu\in\EE}\mu(\vfi_0)-\omega$ and in
\eqref{eq:approxlog2} set
$c=\sup_{\mu\in\EE}\mu(\vfi_0)+\omega$. Then we have the
inclusion
\begin{align}\label{eq:omegazero}
  \left\{\inf_{\mu\in\EE}\Big\{\Big|\frac1n
    S_n\vfi-\mu(\vfi)\Big|\Big\}>\omega\right\}
  \subseteq
  B_n
   \cup
  \left\{\frac1n S_n\Delta_{\delta_1}\ge\epsilon_1\right\}.
\end{align}
By Theorem~\ref{thm:LDNUE} we may find $\epsilon,\delta>0$
small enough so that the exponentially slow recurrence holds
also for the pair $(\epsilon,\delta)$ and hence
\begin{align}
  \label{eq:omegazero1}
\limsup_{n\to+\infty}\frac1n\log\leb
\left\{\inf_{\mu\in\EE}\Big\{\Big|\frac1n
S_n\vfi_0-\mu(\vfi_0)\Big|\Big\}>\omega_0\right\}<0.
\end{align}
Finally the choice of $\epsilon_1,\delta_1$ according to the
condition on exponential slow recurrence to $\cD$ ensures
that the Lebesgue measure of the right hand subset in
\eqref{eq:omegazero} is also exponentially small when
$n\to\infty$. This together with \eqref{eq:omegazero1}
concludes the proof of Theorem~\ref{thm:LDNUElog}.
\end{proof}


\subsection{The roof function and the induced observable as
  observables over the base dynamics}
\label{sec:roof-function-as}

We now proceed with the estimate of the Lebesgue measure of
the sets in \eqref{eq:red-din-base-1} using the results from
the previous Subsection~\ref{sec:reduct-one-dimens},
assuming exponentially slow recurrence to $\cD$ under the
dynamics of $f$ and also that $\cD$ is finite: we write
$\#\cD$ for the number of elements of $\cD$ in what follows.

To estimate the Lebesgue measure of the right hand side
subset in
\eqref{eq:red-din-base-1} we take a sufficiently large
$N\in\ZZ^+$ so that $N\|\psi\|>2$ and note that for
$\omega>0$ by using \eqref{eq:Ibound} and a large $T>0$
\begin{align}
  &\lambda\Big\{I>\frac\omega2\Big\}
  =
  \int d\leb(x)\int_0^{\tau(x)} \!\!\! ds\,
  \big(\chi_{(\omega/2,+\infty)}\circ I\big)(x,s,T)\nonumber
  \\
  &\le
  \leb\{ \tau>\omega T\}
  +
  \omega T \sum_{i=0}^{[T/\tau_0]+1}
  \lambda\left\{
    \frac{|S_{i+1}^f\tau-S_i ^f\tau|}T>\frac{\omega}4
    \right\},\label{eq:LD02}
\end{align}
where $\tau_0=\inf\tau>0$. Because $\tau$ has logarithmic
growth near the \emph{finite subset} $\cD$ together with (P7)
\begin{align*}
  \leb\{ \tau>\omega T\}
  =
  \lambda\{x\in I: d(x,\cD) \le e^{-\frac{\omega T}K}\}
  \le
  C_d e^{-\frac{d\omega T}K}\#\cD.
\end{align*}
On the other hand, since $T\ge S_i\tau(x)\ge\tau_0i$ we
obtain for each $i=0,\dots,[T/\tau_0]+1$
\begin{align}\label{eq:LDobstau}
  \lambda\left\{
    \frac{|S_{i+1}^f\tau-S_i ^f\tau|}T>\frac{\omega}4
  \right\}
  \le\sum_{j=0,1}
  \lambda\left\{
  \inf_{\mu\in\EE_f}\left|\frac1{i+j} S^f_{i+j}
  \tau-\mu(\tau)\right|>\frac{\omega\tau_0}4
  \right\}\le 2C_0e^{-\gamma i}
\end{align}
for some constants $C_0,\gamma>0$. This follows from Theorem
\ref{thm:LDNUE} assuming exponentially slow recurrence for
$f$. Hence \eqref{eq:LD02} is bounded from above by
\begin{align*}
  C_de^{-\frac{d\omega T}K}\#\cD
  +
  \omega T 2C_0\sum_{i=0}^{[T/r_0]+1} e^{-\gamma i}
  \le
  C_d C_0\omega T\big(e^{-\frac{d\omega T}K}+ e^{-\gamma T/\tau_0}\big)
\end{align*}
for all big enough $T>0$. Hence we have proved
\begin{align}\label{eq:limsup1}
  \limsup_{T\to\infty}
  \frac1T\log\lambda\Big\{I>\frac\omega2\Big\}<0.
\end{align}

\subsubsection{Using $\vfi$ as an observable for the $f$
  dynamics}
\label{sec:using-vfi-as}

Now we consider the measures of the subsets in
\eqref{eq:omega4-1}. For the right hand side subset in
\eqref{eq:omega4-1} we can bound its Lebesgue measure by
\begin{align}
  & \leb\left\{(x,s)\in\Xi^\tau:
    \inf_{\mu\in\EE}\left| \frac{n}T -\frac1{\mu(\tau)} \right|
    > \frac{\omega}{4|\mu(\vfi)|}\,\&\, \tau\le T\right\} +
  \leb\{(x,s)\in\Xi^\tau: \tau>  T \}\nonumber
  \\
  &\le T \sum_{i=0}^{[T/r_0]+1}\sum_{j=0,1} \leb\left\{\left|
      \frac{i}{S_{i+j}^F\tau} -\frac1{\mu(\tau)} \right| >
    \frac{\omega}{|\mu(\vfi)|}\right\} +
    \int_{\{\tau>T\}} \tau \,d\leb
    \label{eq:LD41}
\end{align}
Since $\tau$ has logarithmic growth near $\cD$ and $\cD$ is
finite, we get for $T$ large enough so that $i>[T]$ implies
$C_d(i+1)e^{-id/K}\cdot \#\cD<e^{-\gamma i}$ for some $\gamma>0$
\begin{align}
  \int_{\{\tau> T\}} \hspace{-0.5cm} \tau  \,\, d\leb
  &\le
  \sum_{i\ge[T]}\int_i^{i+1} \tau\,d\leb
  \le
  \sum_{i\ge[T]}(i+1)\leb\{\tau>i\}\nonumber
  \\
  &\le
  \sum_{i\ge[T]} (i+1) C_d e^{-id/K}\#\cD
  \le
  \sum_{i\ge[T]} e^{-\gamma i}
  \le
  \frac{e^{-\gamma T}}{1-e^{-\gamma}}\label{eq:LD42}.
\end{align}
For the double summation~\eqref{eq:LD41} we use again large
deviations for $f$ on the observable $\tau$ as in
\eqref{eq:LDobstau} to get the upper bound
$C_1 T e^{-\gamma T/\tau_0}$ for a constant $C_1$ depending
only on $f, C_0, C_d, \gamma$ and $\tau_0$.

This shows that the Lebesgue measure of the right hand side
subset of \eqref{eq:omega4-1} decays exponentially fast as
$T\nearrow\infty$.

Finally, for a small $\hat\omega>0$ the left hand side
subset of \eqref{eq:omega4-1} is contained in the union
\begin{align}
  \left\{
  \inf_{\mu\in\EE}\left|\frac{T}{n}-\mu(\tau)\right|>\hat{\omega}\right\}
  \bigcup
  \left\{
  \inf_{\mu\in\EE_F}\left|\frac{S_{n}\varphi}{n}-\mu(\varphi)\right|
  >\frac{\omega}{4}\left(\hat{\omega}+
  \inf_{\mu\in\EE_F}\mu(\tau)\right)\right\}.\label{eq:last}
\end{align}
The left hand side subset of~\eqref{eq:last} has Lebesgue
measure which decays exponentially fast as $T\nearrow\infty$
following the same arguments as in \eqref{eq:LD41}. For the
right hand side subset, we again use a large deviation bound
for $\vfi$ with respect to the dynamics of $f$ as in
\eqref{eq:LDobstau}. Putting all together we conclude the
proof of Theorem~\ref{thm:LD-suspension_F}.


\section{Exponentially slow recurrence}
\label{sec:exponent-slow-recurr-1}

As explained in Section~\ref{sec:reduct-large-deviat}, we
are now left to prove Theorem~\ref{mthm:expslowapprox} to
complete the proof of
Theorem~\ref{mthm:LD-sing-hyp-attracting}.


Let $f:I\setminus\cD\to I$ be a piecewise $C^{1+\alpha}$
one-dimensional map, for some $\alpha>0$, in the setting of
Theorem~\ref{mthm:expslowapprox}.
For every $c\in\cD$ and small $\delta>0$ we recall that
$\Delta(c,\delta)$ represents the one-sided neighborhood of
$c$, we set
$0<\beta_0=\inf_{c\in\cS}\alpha(c)\le\sup_{c\in\cS}\alpha(c)=\beta_1<1$
and
\begin{align*}
  \delta_c=\frac12\sup\{\delta>0:
  \Delta(c,\delta)\subset I\setminus\cD\}
\end{align*}
half of the largest possible radius of this neighborhood not
including other elements of $\cD$. We write the boundary
points $\partial\Delta(c,\delta_c)=\{c,c+\delta_c\}$ if
$c=c^+$ and $\partial\Delta(c,\delta_c)=\{c,c-\delta_c\}$ if
$c=c^-$, so that $c\pm\delta_c$ is the mid point between $c$
and the next element of $\cD$, according to the side of the
one-sided neighborhood; recall~\eqref{eq:1sided}. We also
fix a small $\epsilon_1>0$ so that
$1-\beta_1>\epsilon_1\beta_1$.

We define a partition $\cP_0$ of $I$ as follows.

\subsection{Initial partition}
\label{sec:initial-partit-refin}

The Lebesgue modulo zero partition of $I$ to be constructed
consists of \emph{intervals whose length is comparable to a
  power of the distance to $\cD$}.  For this, the following
simple result will be very useful.
\begin{lemma}
  \label{le:an}
  Let $a_n=n^\gamma$ with $\gamma=-1/(\epsilon_1\beta_1)$ and
  $n\ge1$. Then there exists $K_0>1$ so that for all $n\ge1$
  \begin{align*}
    \frac1{K_0}<\frac{a_n-a_{n+1}}{a_n^{1+\epsilon_1\beta_1}}< K_0
    \qand
    \frac{a_{n-1}-a_{n+2}}{a_n-a_{n+1}}
    <
    K_1
    :=\max\{1+2K_0,\sigma^{T_0}\}.
  \end{align*}
\end{lemma}
We define the partition in each $\Delta(c,\delta_c)$
according to whether $c\in\cS$ or $c\in\cD\setminus\cS$.

\subsubsection{Near a singular point}
\label{sec:near-singular-point}

If $c\in\cS$, then we partition $\Delta(c,\delta_c)$ into (see
Figure~\ref{fig:part1})
\begin{align*}
  M(c,p)=\Delta(c,a_p)\setminus\Delta(c,a_{p+1}), \quad p\ge\rho(c)
\end{align*}
where $\rho(c)=\inf\{\rho\in\ZZ^+: a_p<\delta_c\}$ is a
threshold defined for each $c\in\cS$.

\begin{figure}[htpb]
  \centering
\includegraphics[width=12cm]{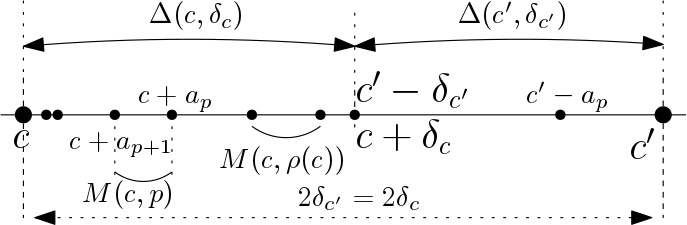}
\caption{Sketch of the partition in one-sided neighborhoods:
  right-side neighborhood of $c\in\cS$ and left-side
  neighborhood of $c'\in\cD$.}
  \label{fig:part1}
\end{figure}

In addition, from assumption~\eqref{eq:der0},
there is $B>0$ so that
\begin{align}
  \label{eq:nondegexp}
  |f'x|
  \ge
  \frac1B d(x,\cS)^{\alpha(c)-1}
  \ge
  \frac1B d(x,\cS)^{\beta_1-1},
  \quad x\in \Delta(c,\delta_c), \forall c\in\cS
\end{align}

\subsubsection{Near a discontinuity point}
\label{sec:near-discont-point}

According to the main assumption in the statement of
Theorem~\ref{mthm:expslowapprox}, we can map a (one-sided)
neighborhood $\Delta(c,\delta)$ of $c\in\cD\setminus\cS$
into a (one-sided) neighborhood $\Delta(\tilde c,\epsilon)$
of some $\tilde c\in\cS$ in finitely many $T=T(c)\le T_0$
iterates of $f$, for some pair $\epsilon,\delta>0$.

Hence, we can use this map to \emph{pull-back the partition
  elements defined in neighborhoods of $\cS$ to obtain
  partition elements in neighborhoods of
  $\cD\setminus\cS$}. More precisely, there exist
$\rho(c)\in\ZZ^+,\delta>0$ so that the following is well defined
\begin{align}\label{eq:1bddrho0}
  M(c,p)
  =
  (f^T\mid_{\Delta(c,\delta)})^{-1}(M(\tilde c,p)), \quad p\ge\rho(c).
\end{align}
Extra conditions on $\rho$ will be imposed
at~\eqref{eq:escnosing}~\eqref{eq:conexao} in
Subsection~\ref{sec:full-set-conditions}.

We denote by $\cP$ the family of all intervals
$\{M(c,p):p\ge\rho(c),c\in\cD\}$ defined up to this point.
Since each $\rho(c), c\in\cD\setminus\cS$ needs to be big
    enough, $\cP$ is not a partition of $I\setminus\cD$.

\subsubsection{Global initial partition}
\label{sec:global-initial-parti-1}

Let now $\cP_0$ be formed by the collection of all intervals
$M(c,p)$ for all $c\in\cD$ and $p\ge\rho(c)$ together with
the connected components of
\begin{align*}
  I\setminus
  \left(\bigcup_{c\in\cD}\bigcup_{p\ge\rho(c)}M(c,p)\right),
\end{align*}
which will be the \emph{escape intervals}.  We denote these
components by $M(c,\rho(c)-1)$, the \emph{escape interval}
of $c\in\cD\setminus\cS$, whenever they intersect
$\Delta(c,\delta_c)$; see Figure~\ref{fig:part2}.

\begin{figure}[htpb]
  \centering
  \includegraphics[width=15cm]{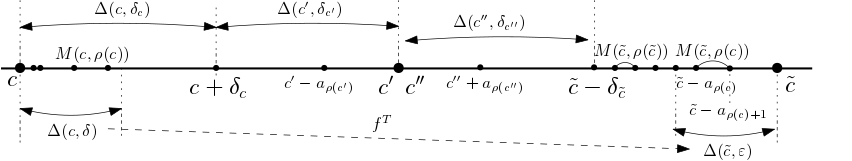}
  \caption{Sketch of the  partition on one-sided
    neighborhoods of consecutive points
    $c, c', c''\in \cD\setminus\cS$ and $\tilde c=f^Tc$ with
    $\tilde c\in\cS$, illustrating the pull-back from
    $\tilde c$ to $c$ and the one-sided neighborhoods and
    middle points between one-sided elements of $\cD$.}
  \label{fig:part2}
\end{figure}

\begin{remark}\label{rmk:escapecomponents}
  For any given escape interval there might be two points
  $c, c'\in\cD$ such that $M(c,\rho(c)-1)=M(c',\rho(c)-1)$,
  but the focus will be the length of these intervals, and
  the length of $M(c,p)$ does not depend on $c\in\cS$ for
  $p\neq\rho(c)-1$ by construction.
\end{remark}

\subsubsection{Estimates on atom lengths}
\label{sec:global-initial-parti}

For a subset $A$ of $I$ we denote by $|A|$ the Lebesgue
measure of $A$.  For each element $\eta=M(c,p)$ with
$p>\rho(c)$ of $\cP_0$ we denote by $\eta^+$ the interval
obtained by joining $\eta$ with its two neighbors in
$\cP_0$.

From the definitions in
Subsection~\ref{sec:near-singular-point} and
Lemma~\ref{le:an}, since $|M(c,p)|=a_p-a_{p+1}$
\begin{align}
  \label{eq:atomsizesing}
  c\in\cD, p\ge\rho(c)
  \implies
  \left\{
  \begin{array}{rcl}{}
    |M(c,p)^+|
    &\le&
          K_1|M(c,p)| \le K_1K_0 p^{-1-1/\epsilon_1\beta_1}
    \\
    \frac{p^{-1-1/\epsilon_1\beta_1}}{K_0}
    &\le&
          |M(c,p)|
  \end{array}
          \right..
\end{align}
For $c\in\cD\setminus\cS$ we recall that by assumption for
$1\le j<T=T(c)\le T_0$ there exists $c_j\in\cD\setminus\cS$
satisfying $d(f^j(c),\cD)=d(f^j(c),c_j)>0$ and the
collection $\cC=\{c_j: 1\le j<T(c), c\in\cD\setminus\cS\}$
is finite. Hence we get, from the Mean Value Theorem
\begin{align}\label{eq:connectsing}
  \sigma|M(c,p)|
  \le
  \sigma^T|M(c,p)|
  \le
  |M(\tilde c,p)|
  =
  |(f^T)'(\xi_{c,p})|\cdot|M(c,p)|
  \le
  K_1 |M(c,p)|
\end{align}
for some $\xi_{c,p}\in M(c,p)$, where $\tilde c\in\cS$ is
the singular point associated to $c$, $|(f^T)'(\xi_{k,p})|$
is uniformly bounded from below by $\sigma^T>\sigma$ and from above
by $K_1\ge\sigma^{T_0}>1$ for all $p\ge\rho(c)$ and
$c\in\cD\setminus\cS$, since $\cC$ is finite.
This implies in particular
\begin{align}
  \label{eq:M+}
  c\in\cD\setminus\cS, p\ge\rho(c)
  \implies
  |M(c,p)^+|
  \le
  \frac{|M(\tilde c,p)^+|}{\sigma^T}
  \le
  \frac{K_1}{\sigma}|M(\tilde c,p)|
  \le
  K_1^2  |M(c,p)|.
\end{align}
Using~\eqref{eq:atomsizesing} for $p\ge\rho(c)$ we arrive at 
\begin{align}\label{eq:atomsizediscont}
  \frac{p^{\gamma-1}}{K_0K_1}
  &\le
  \frac{|M(\tilde c,p)|}{K_1}
  \le
    |M(c,p)|
  \le   
  \frac{|M(\tilde c,p)|}{\sigma}
  \le
  \frac{K_0}{\sigma}
  p^{\gamma-1}.
\end{align}
From \eqref{eq:atomsizesing} and \eqref{eq:atomsizediscont}
we can relate distance to $\cD$ with the length of the atoms
of $\cP$. For $c\in\cS$ we can use Lemma~\ref{le:an} and the
definition of the partition to deduce\footnote{We write
  \( f\approx g \) if the ratio \( f/g \) is bounded above
  and below independently of $c\in\cS, p\ge\rho(c)$. }
\begin{align*}
  d\big(M(c,p),\cD)
  &=
    a_{p+1}
    =
    a_p\left(1+\frac1p\right)^\gamma
    \approx
    \big(a_p^{1+\epsilon_1\beta_1}\big)^{1/(1+\epsilon_1\beta_1)}
    \approx
    |M(c,p)|^{1/(1+\epsilon_1\beta_1)}\approx a_p.
\end{align*}
For $c\in\cD\setminus\cS$, using the above relation together
with \eqref{eq:connectsing} and~\eqref{eq:atomsizediscont}
we again get
\begin{align*}
  \frac{d\big(f^T(M(k,p)),\cD)}{K_1}
  \le
  d\big(M(c,p),\cD)
  &\le
    \frac{d\big(f^T(M(c,p)),\cD)}{\sigma^T}
    \approx
    |M(c,p)|^{1/(1+\epsilon_1\beta_1)}\approx a_p;
\end{align*}
hence, there exists a constant $K_2>0$ so that
\begin{align}
  \label{eq:distD_length}
  \frac1{K_2}
  \le
  \frac{d\big(M(c,p),\cD)}
  {|M(c,p)|^{1/(1+\epsilon_1\beta_1)}}
  \le K_2,
  \quad p\ge\rho(c), c\in\cD.
\end{align}

\subsection{Dynamical refinement of the partition and
  bounded distortion}
\label{sec:dynamic-refinem-part}

Following \cite[Section 6]{araujo2006a}, the partition
$\cP_0$ is dynamically refined so that any pair $x,y$ of
points in the same atom of the $n$th refinement $\cP_n$,
i.e.  $y\in\cP_n(x)$, belong to the same element $\eta^+$
during the first consecutive $n$ iterates: there are
$\eta_i\in\cP_0$ so that 
$f^i(x),f^i(y)\in\eta_i^+$ for
$i=0,\dots,n$. 
Moreover, $\cP_n$ is a collection of intervals for each
$n\ge1$ and $f^{n+1}\mid\omega: \omega\to f^{n+1}\omega$ is
a diffeomorphism for every interval $\omega\in\cP_n$.

The details of this refinement will be presented at
Subsection~\ref{sec:dynamic-refinem-algo}. Before this, we
first obtain a bounded distortion property for $\cP_n$ which
follows just from the above general properties of the
refinement and the choice of the initial partition.

\subsubsection{Bounded distortion}
\label{sec:bounded-distort}

Slightly more general than in \cite[Section
6.3]{araujo2006a} (where this was only stated for atoms of
$\cP_n$ in $n$ iterates while it is also valid for atoms of
$\cP_{n-1}$), uniform expansion and the domination of the
length of atoms of $\cP_0$ by a power of the distance to
$\cD$ imply bounded distortion on atoms of the partition
$\cP_{n-1}$.

\begin{lemma}
  \label{le:bddist}
  There exists $D>0$ depending only on $f$
  such that 
  \begin{align*}
    y\in\cP_{n-1}(x), n\ge1
    \implies
      \log\left|\frac{(f^n)'(x)}{(f^n)'(y)}\right|\le D.
  \end{align*}
\end{lemma}

\begin{proof}
  For $\omega\in\cP_{n-1}$ for some $n\ge1$ and
  $x,y\in\omega$, since
  $f^i\mid\omega:\omega\to\omega_i=f^i\omega$ is a
  diffeomorphism for $i=1,\dots,n$, then writing
  $x_i=f^ix, y_i=f^iy$
  \begin{align}\label{eq:bddist0}
    \left|\log\frac{|(f^n)'(x)|}{|(f^n)'(y)|}\right|
    &\le
      \sum_{i=0}^{n-1} \left|\log\frac{|f'x_i|}{|f'y_i|}\right|
      \le
      \sum_{i=0}^{n-1} \frac{\big||f'x_i|-|f'y_i|\big|}
      {\min\{|f'x_i|, |f'y_i|\}}.
  \end{align}
  If $\omega_i\in\Delta(c,\delta_c)$ with
  $c\in\cD\setminus\cS$, then by uniform expansion
  and~\eqref{eq:derD-S} we can bound the $i$th summand from
  above by
  \begin{align}\label{eq:bddist1}
    H\frac{|x_i-y_i|^\alpha}{|f'\xi_i|}
    \le
    \frac{H}{\sigma}\sigma^{\alpha(i-n+1)}|x_n-y_n|^\alpha.
  \end{align}
  Otherwise, $\omega_i\in\Delta(c,\delta_c)$ with
  $c\in\cS$. Then we set $\theta\in(0,1)$ such that
  $\theta/(1-\theta)=\epsilon_1$ and $1-\beta_1>\epsilon_1\beta_1$,
  use~\eqref{eq:der0} together with the size of the
  partition elements near $\cS$ to bound the middle $i$th
  summand of \eqref{eq:bddist0} as follows
  \footnote{Here it is important to have $\cP_0$ defined
    in $\Delta(c,\delta_c)$ for $c\in\cS$ with $\rho(c)$ the
    minimum possible value such that $a_p<\delta_c$, to
    have a tight control of $|x_i-y_i|$.}
  \begin{align}\label{eq:bddist1a}
    \left|\log\frac{|f'x_i|}{|f'y_i|}\right|
    \le
    C\left|\log\frac{|x_i-c|}{|y_i-c|}\right|
    \le
    C\frac{|x_i-y_i|}{|x_i-c|}
    &=
      C\frac{|x_{i+1}-y_{i+1}|^\theta}{|f'\xi_i|^\theta}
      \cdot
      \frac{|x_i-y_i|^{1-\theta}}{|x_i-c|},
  \end{align}
  where $C$ depends only on $f$ and $\xi_i$ is between $x_i$
  and $y_i$. Since
  $|f'\xi_i|\approx|\xi_i-c|^{\theta(\alpha(c)-1)}$
  from~\eqref{eq:der0} again and 
  \begin{align*}
    \left|\frac{x_i-c}{\xi_i-c}\right|
    &=
      \left|\frac{x_i-c}{x_i-c+\xi_i-x_i}\right|
      \le
      \frac{|x_i-c|}{|x_i-c|-|y_i-x_i|}
    \le
      \frac1{1-(a_p-a_{p-1})/a_p}
  \end{align*}
  for some $p\ge\rho(c)$ so that $\omega_i\in M(c,p)^+$ for
  some $c\in\cS$, then Lemma~\ref{le:an} ensures that the
  last expression is bounded and we can find another
  constant to bound \eqref{eq:bddist1a} by
  \begin{align*}
    \hat C |x_{i+1}-y_{i+1}|^\theta
    \left(\frac{|x_i-y_i|^{1-\theta}}
    {|\xi_i-c|^{\theta(\alpha(c)-1)+1}}
    \right).
  \end{align*}
  The last expression in parenthesis is bounded from above
  by a constant since
  $\frac{\theta(\alpha(c)-1)+1}{1-\theta}=
  {1+\alpha(c)\theta/(1-\theta)}=1+\alpha(c)\epsilon_1\le
  1+\epsilon_1\beta_1$ and
\begin{align*}
  |x_i-y_i|
  \le
  K_2|x_i-c|^{1+\epsilon_1\beta_1}
  \le
  K_2|x_i-c|^{1+\alpha(c)\epsilon_1}
\end{align*}
is precisely provided by~\eqref{eq:distD_length} from the
construction of the initial partition.
Hence we can find a constant $\tilde D>0$ depending only on
$f$ so that
\begin{align}\label{eq:bddist3}
  \left|\log\frac{|f'x_i|}{|f'y_i|}\right|
  \le
  C\frac{|x_i-y_i|}{|x_i-c|}
  \le
  \tilde D|x_{i+1}-y_{i+1}|^\theta
  \le
  \tilde D\sigma^{\theta(i-n+1)}|x_n-y_n|^\theta.
\end{align}
Finally, putting~\eqref{eq:bddist1} and~\eqref{eq:bddist3}
together, since $|x_n-y_n|\le1$ we get
\begin{align*}
  \left|\log\frac{|(f^n)'(x)|}{|(f^n)'(y)|}\right|
  &\le
    \max\left\{\frac{H}{\sigma},\tilde D\right\}
    \sum_{i=0}^{n-1}\sigma^{(i-n+1)\min\{\alpha,\theta\}}\le D
\end{align*}
where $D>0$ depends on $\sigma, H, \tilde D, \alpha$ and
$\theta$ which, in turn, depend on $f$.
\end{proof}

\subsubsection{Conditions on $\rho_0$}
\label{sec:full-set-conditions}

We now specify a threshold $\rho_0\in\ZZ^+$ for the distance
to $\cD\setminus\cS$ given by the family of intervals to
consider in $\cP_0$ near discontinuity points and its
complementary escape intervals. We recall that for
$c\in\cD\setminus\cS$ and $1\le j<T=T(c)\le T_0$ there exists
$c_j\in\cD\setminus\cS$ satisfying
$d(f^j(c),\cD)=d(f^j(c),c_j)>0$ and so 
$\cC=\{c_j: 1\le j<T(c), c\in\cD\setminus\cS\}$ is finite.
Moreover, since $\cS$ is finite, we assume that
\begin{align}\label{eq:escnosing}
  \rho_0>\rho(c), \quad\forall c\in\cS.
\end{align}
Hence we can choose $\rho_0\in\ZZ^+$ big enough and find
$\bar\delta>0$ so that, setting
$\rho(c)=\rho_0, c\in\cD\setminus\cS$, then we have the
following besides~\eqref{eq:1bddrho0}
\begin{align}
  \label{eq:conexao}
  d(f^j(c),c_j)
  &>
    (\rho_0-1)^\gamma\ge \bar\delta,
    \quad
  0< j<T(c), c\in\cD\setminus\cS.
\end{align}
This ensures that for $p\ge\rho_0$ we get
$f^jM(c,p)\subset M(c_j,\rho_0-1), 0<j<T$, that is, the
orbit of $M(c,p)$ is contained in the escape set near
$\cD\setminus\cS$ until it reaches a one-sided neighborhood
of some point of $\cS$.  We define
\begin{align*}
    \kappa_0(\rho_0)
    &=
    \sup_{c,c'\in\cD\setminus\cS}\frac{|M(c,\rho_0-1)^+|}
  {|M(c',\rho_0-1)|}
  \qand
  \un{\beta}=\un{\beta}_{\ell}
  =
  1+
  \frac1{\ell}\sup_{c\in\cD\setminus\cS}\frac{\log\sigma}{\log(2\delta_c)};
\end{align*}
and then set
$\beta_2=\frac{(1+\epsilon_1)\beta_1}{1+\epsilon_1\beta_1}$
and choose $\ell\in\ZZ^+$ big enough so that
$\un{\beta}\le\beta_2$.

We note that\footnote{Because $2\delta_c$ is the size of one
  monotonous branch of $f$ having $c$ in its boundary, so
  $\sigma\cdot2\delta_c<1$.}  $0<\un{\beta}<1$ and
$|M(c,\rho_0-1)|\xrightarrow[\rho_0\nearrow\infty]{}2\delta_c$
for each $c\in\cD\setminus\cS$ and also
$ \kappa_0(\rho_0)\xrightarrow[\rho_0\nearrow\infty]{}
\kappa_0=\sup\{2^{1-\beta_2}\delta_c\delta_{c'}^{-\beta_2}:
c,c'\in\cD\setminus\cS\}$.

Then we define $\beta_3$ to be such that
$\beta_2+(1+\epsilon_1\beta_1)/(\epsilon_1\beta_1)<\beta_3<1$. We
also set $L=\kappa_0BD^2 (K_1K_0)^{\beta_2-1}$ and choose
$\rho_0\in\ZZ^+$ so that, in addition
to~\eqref{eq:1bddrho0}, \eqref{eq:escnosing}
and~\eqref{eq:conexao}, the following is also true
\begin{align}
  \label{eq:escapeint0}
  \frac12\le\frac{\kappa_0(\rho_0)}{\kappa_0}\le\frac32;
  \quad
  2\delta_c
  \le
  \frac{|M(c,\rho_0-1)|}{\sigma^{1/\ell(1-\un{\beta})}}
  \qand
  p>\rho_0
  \implies
  \frac{|M(c,p)^+|}{|M(c,p)|^{\beta_3}}\le L^{-3};
\end{align}
ensuring, in particular, that
$|f(M(c,\rho_0-1))|\ge\sigma|M(c,\rho_0-1))| >
|M(c,\rho_0-1))|^{\un{\beta}}$, for all $c\in\cD\setminus\cS$.
Moreover we also need that
\begin{align}
  \label{eq:escapeint1}
  p\ge\rho_0, c\in\cD, c'\in\cD\setminus\cS
  \implies
  \begin{cases}
    |M(c,p)|^{\beta_3}/|M(c',\rho_0-1)|^{\beta_2} \le
    |M(c,p)|^{\beta_3-\beta_2}
    \\
    |M(c,p)|^{\beta_3-\beta_2}\cdot
    BK_1^{1-\beta_2}K_2^{1-\beta_1}\le1
    \\
    |M(c,p)|^{1/(1+\epsilon_1\beta_1)}\ge K_2|M(c,p)|
  \end{cases}.
\end{align}
These conditions can be simultaneously achieved by a large
enough $\rho_0$ and are crucial in the proof of
Lemma~\ref{le:sizedepth} and in
subsections~\ref{sec:distance-singul-set}
and~\ref{sec:expect-value-splitt}.

\subsubsection{The dynamical refinement algorithm}
\label{sec:dynamic-refinem-algo}

The refinement algorithm is defined inductively, assuming
that $\cP_n$ is already defined and, for each
$\omega\in\cP_n$, there are sets
$R_n(\omega)=\{r_1<\dots<r_s\}$ (with $r_1\ge0$ and
$r_s\le n$) of $u_n(\omega)=\#R_n(\omega)$ \emph{return
  times} and $D_{n}(\omega)=\{(c_1,p_1), \dots,(c_s,p_s)\}$
whose pairs give the corresponding \emph{return depths}, to
be defined below.

First for $\omega=M(c,p)\in\cP_0$ for some $c\in\cD$ we set
$R_0(\omega)=\{0\}$ and $D_0(\omega)=\{(c,p)\}$.  Then, for
each $n\ge1$ we assume that $\cP_n$ is defined and for each
$\omega\in\cP_n$ that $R_n(\omega), D_n(\omega)$ are also
defined. Then we analyze the interval $f^{n+1}\omega$:
\begin{itemize}
\item if 
  $f^{n+1}(\omega)$ intersects more than three atoms of
  $\cP_0$, then we set $\omega\in\cP_{n+1}$,
  $R_{n+1}(\omega)=R_n(\omega)$ and
  $D_{n+1}(\omega)=D_n(\omega)$. For the points in $\omega$,
  the iterate $n+1$ is called a 
  \emph{free time}. 
\item otherwise the iterate $n+1$ is a \emph{return time}
  for the points in $\omega$, and we consider the
  subsets\footnote{The requirement that the splitting is
    done only if the interval $f^{n+1}\omega$ intersects
    more than three elements of $\cP_0$ is crucial to ensure
    that a certain proportion of this interval goes to each
    subinterval; see Figure~\ref{fig:returns},
    Remark~\ref{rmk:zeta0}(1) and Remark~\ref{rmk:3split}.}
  $\bar\eta_{c,p}=\big(f^{n+1}\mid \omega\big)^{-1}(M(c,p))$
  of the interval $\omega$ for all elements $M(c,p)$ of
  $\cP_0$ which intersect $f^{n+1}(\omega)$; see
  Figure~\ref{fig:repart}.

\begin{figure}[htpb]
  \centering
  \includegraphics[width=12cm]{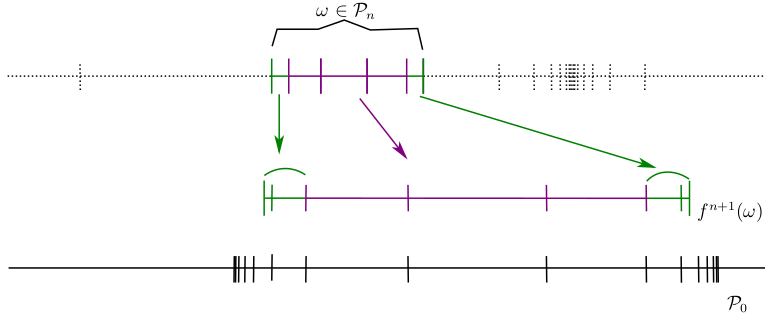}
  \caption{Illustration of the refinement of $\omega\in\cP_n$.}
  \label{fig:repart}
\end{figure}

This family $\bar\cQ=\{\bar\eta_{c,p}\}$ is a partition of
$\omega$ Lebesgue modulo zero such that
$f^{n+1}(\bar\eta_{c,p})$ is either equal to $M(c,p)$; or
strictly contained in $M(c,p)$.

In the latter case, $f^{n+1}(\bar\eta_{c,p})$ is necessarily at
an extreme of the interval $f^{n+1}(\omega)$ and we join
$\bar\eta_{c,p}$ with its neighbor in $\cQ$ obtaining an
interval $\eta_{c,p}$.

In this way we construct a new partition $\cQ=\{\eta_{c,p}\}$ of
$\omega$ which satisfies
  \begin{align}\label{eq:nesting}
  M(c,p)\subseteq f^{n+1}(\eta_{c,p}) \subseteq M(c,p)^+;
  \end{align}
  for all $(c,p)$ with either $c\in\cS$ and $p\ge\rho(c)$,
  or $c\in\cD\setminus\cS$ and $p\ge\rho_0-1$, so that
  $\eta_{c,p}\in\cQ$; and we set
  \begin{align*}
  \eta_{c,p}\in\cP_{n+1}, \quad
R_{n+1}(\eta_{c,p})&=R_n(\eta_{c,p})\cup\{n+1\} \qand
\\D_{n+1}(\eta_{c,p})&=D_n(\eta_{c,p})\cup\{(c,p)\}.
  \end{align*}
  The cases $\eta_{c,\rho_0-1}\in\cQ$ with
  $c\in\cD\setminus\cS$ are a special kind of return time,
  identified in what follows by naming $n+1$ an
  \emph{escape time} for the points of $\eta_{c,\rho_0-1}$.
\end{itemize}

\begin{remark}\label{rmk:escape}
  If $\eta_{c,\rho_0}\in\cQ$, then we might have
  $M(c,\rho_0-1) \cap f^{n+1}\eta_{c,\rho_0}\neq\emptyset$
  but even in this case $M(c,\rho_0-1)$ is not necessarily
  covered; see Figure~\ref{fig:returns}.
\end{remark}

To finish the refining algorithm, we repeat the procedure
for each $\omega\in\cP_n$ completing the construction of
$\cP_{n+1}$ from $\cP_n$ for $n\ge1$.

Clearly, since the atoms of the initial partition $\cP_0$
are intervals, then this construction shows that all the
atoms of $\cP_{n}$ are intervals, for all $n\ge1$.

\begin{figure}[htpb]
  \centering
  \includegraphics[width=15cm]{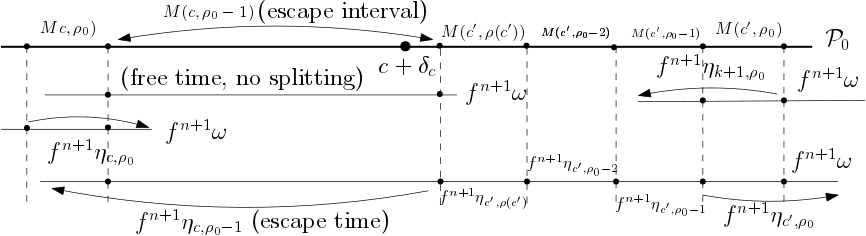}
  \caption{Illustration of different relative positions of
    $f^{n+1}\omega$ and atoms of $\cP_0$ between consecutive
    elements of $c\in\cD\setminus\cS$ and $c'\in\cS$. Note
    how enforcing splitting of $f^{n+1}\omega$ only if it
    intersects more than $3$ elements of $\cP_0$ ensures
    that a certain proportion of $\omega$ is removed to
    obtain any $\eta_{c,p}$.}
  \label{fig:returns}
\end{figure}

\subsection{Measure of atoms of $\cP_n$ as a function of
  return depths}
\label{sec:dist-depth}

Here we estimate the measure of $\omega\in\cP_n$ using
$R_n(\omega)$ and $D_n(\omega)$ in a similar way to
\cite[Section 6.4]{araujo2006a} but with the new
length/distance relations from~\eqref{eq:distD_length}.

We start by fixing $n\in\ZZ^+$, $u\in\{1,\dots,n\}$ and
taking $\omega^0\in\cP_0$.  Let $\omega\in\cP_n$ be such
that $\omega\subset\omega^0$ and
$u_n(\omega) =u=\#\R_n(\omega)$ is the number of return
times $R_n(\omega)=\{t_0,t_1,\dots,t_u\}$ of $\omega$, where
$0=t_0<t_1<\dots<t_u\le n$; and
$D_n(\omega)=\{(c_1,p_1),\dots,(c_u,p_u)\}$ be the return
depths determined by $\omega$ through the refinement
algorithm.

For $m=1,\dots,u$ we write
$\omega^m=\omega((c_1,p_1),\dots,(c_m,p_m))\in\cP_{t_m}$ the
subset of $\omega^0$ satisfying
\begin{align}\label{eq:nested}
  f^{t_i}(\omega^m)\subset M(c_i,p_i)^+,
  1\le i<m \,\mbox{ and }\,
  M(c_m,p_m)\subset f^{t_m}(\omega^m)\subset
  M(c_m,p_m)^+,
\end{align}
by the definition of the sequence of partitions $\cP_n$.
We get a nested sequence of sets
\[
  \omega^0 \supsetneq \omega^1
  \supsetneq \dots \supsetneq \omega^u=\omega.
\]
Fixing $u\le n$ and $p_1, \dots, p_u$, we define
\begin{align*}
\cT=\{ \omega\in\cP_n : \omega\subset\omega^0, \,
u_n(\omega) = u \qand
D_n(\omega)=\{(c_1,p_1),\dots,(c_u,p_u)\}, c_i\in\cD \},
\end{align*}
this is, the set of points contained in atoms of $\cP_n$
which lie inside $\omega^0$ and whose first $n$ iterates
have a given number $u$ of return times and a given sequence
$p_1,\dots, p_u$ of return depths.

Now we define by induction a sequence of refinements of $\cT$
which will enable us to determine the estimates we need. 
Start by putting $\cV_0=\omega^0$. We define
for $1\le i\le u$
\begin{align*}
  \cV_{i}
  =\bigcup
  \{ \omega^{i}: \omega\in\cT \,\&\,
  \eqref{eq:nested} \text{  holds with  }
  m=i\}
\end{align*}
and note that $\cV_i\subset\cV_{i-1}$.  Next we compare
$\big|\cV_{j}\big|$ with
$\big| \cV_{j-1}\big|, j=1,\dots,u$, using the conditions on
$\rho_0$ and definitions of $\beta_2,\beta_3, \un{\beta}$
given in Subsection~\ref{sec:full-set-conditions}.

\begin{lemma}
  \label{le:sizedepth}
  There exists $\zeta_0=\zeta_0(\rho_0)>0$ such that
  \begin{align}\label{eq:comparecVesc}
    \big|\cV_{m}\big|
    &\le
      e^{-\zeta_0}\big| \cV_{m-1} \big|, \quad m=1,\dots, u.
  \end{align}
  Moreover, if $p_m>\rho_0$ for any
  $m\in\{1,\dots,u\}$, 
  then
  \begin{align}\label{eq:compareV}
  \big|\cV_{m}\big|
  &\le
    \frac{1}{L^2\kappa_0\sigma^{t_m-t_{m-1}-1}}
    \cdot
    \frac{|M(c_m,p_m)|^{\beta_3}}
    { |M(c_{m-1},p_{m-1})|^{\beta_2}}
    |\cV_{m-1}|. 
  \end{align}
These estimates provide the general bound
\begin{align}\label{eq:Vgeneral}
  |\cV_u|
  \le
  |\omega^0|^{1-\beta_2}
  e^{-\zeta_0(n-R)}
  \sigma^{-R+r}
  \prod_{p_i>\rho_0}
  |M(c_i,p_i)|^{\beta_3-\beta_2},
\end{align}
where $R$ is the number of iterates between the $r$ return
times having depth larger than $\rho_0$.
\end{lemma}


Assuming the lemma, we fix $n\ge1$, the escape/return times
$0=t_0<t_1<\dots<t_u\le n$ and $(c_1,p_1),\dots,(c_u,p_u)$
and consider the subset
\begin{align*}
  A^u
  &=
    A^{u}_{(c_1,p_1),\dots,(c_u,p_u)}(n)
  =
  \big\{
  x\in M: u_n(x)=u \text{ and }
  f^{t_i}(x)\in M(c_i,p_i),\, i=1,\dots,u
    \big\}
  \\
  &=
    \big\{
    x\in\omega\in\cP_n: D_n(\omega)=\{(c_0,p_0), (c_1,p_1),\dots,(c_u,p_u)\}
    \big\}
\end{align*}
where $c_i\in\cD, p_i\ge\rho_0-1, i=0,1,\dots,u$ and $u_n(x)$ for
$x\in\omega\in\cP_n$ denotes the number of escape/returns of
$\omega$ until the $n$th iterate. Then from
\eqref{eq:Vgeneral} we get
\begin{align}
  |  A^{u}_{(c_1,p_1),\dots,(c_u,p_u)}(n)|
  &\le
    \sum_{\omega\in\cP_n\cap A^{u}\atop \omega\subset\omega^0\in\cP_0}
      e^{-\zeta_0(n-R)}
  \sigma^{-R+r}
  \prod_{p_i>\rho_0}
  |M(c_i,p_i)|^{\beta_3-\beta_2}
    \cdot|M(c_0,p_0)|^{1-\beta_2} \nonumber
  \\
  &\le
      e^{-\zeta_0(n-R)}
  \sigma^{-R+r}
  \prod_{p_i>\rho_0}
  |M(c_i,p_i)|^{\beta_3-\beta_2}
    \sum_{\omega_0\in\cP_0}|\omega^0|^{1-\beta_2}\nonumber
  \\
  &\le  \label{eq:probreturns}
    C_0   e^{-\zeta_0(n-R)}
  \sigma^{-R+r}
  \prod_{p_i>\rho_0}
  |M(c_i,p_i)|^{\beta_3-\beta_2},
\end{align}
where $C_0= \sum_{\omega_0\in\cP_0}|\omega^0|^{1-\beta_2}>1$
depends only on $\beta_1$ and $\epsilon_1$. This is
well-defined since $1-\beta_1>\epsilon_1\beta_1$, by the
choice of $\epsilon_1$, and denoting $\#\cD$ the number of
elements of $\cD$:
\begin{align*}
  C_0
  \le
  \#\cD\cdot
  \sum_{p\ge\rho_0}
  \big(p^{-1-1/(\epsilon_1\beta_1)}\big)^{(1-\beta_1)(1+\epsilon_1\beta_1)}
  =
  \#\cD
  \sum_{p\ge\rho_0}
  p^{-\frac{1-\beta_1}{\epsilon_1\beta_1}}<\infty.
\end{align*}

The proof of Lemma~\ref{le:sizedepth} is contained in the
remaining of this subsection.

\subsubsection{Between return times not considering relative
depths}
\label{sec:between-escapes}

Now we start the proof of Lemma~\ref{le:sizedepth}.  Let us
assume that $m\in\{1,\dots,u\}$ and $t_m$ is an
\emph{escape time}: $p_m=\rho_0-1$ and
$c_m\in\cD\setminus\cS$. Then by the partition
algorithm,  bounded distortion and~\eqref{eq:atomsizesing}
(recall Figure \ref{fig:returns})
\begin{align*}
  \frac{|\omega^{m-1}\setminus \omega^m|}{|\omega^{m-1}|}
  &\ge
    \frac1{D^2}\cdot
    \frac{|f^{t_m}(\omega^{m-1}\setminus
    \omega^m)|}{|f^{t_m}\omega^{m-1}|}
    \ge
    \frac1{D^2}\cdot
    \frac{|M(c,p)|}{|M(c_m,\rho_0-1)^+|}
\end{align*}
where $M(c,p)$ is one of the neighbors of $M(c_m,\rho_0-1)$
in $\cP_0$: either $M(c_m,\rho_0)$; or
$M(c,\rho_0)$ for some $c\in\cD\setminus\cS, c\neq c_m$;
or else $c\in\cS$ and $p=\rho(c)\le\rho_0-1$.

If $t_m$ is such that $p_m=\rho_0$ and
$c_m\in\cD\setminus\cS$, then analogously we obtain
\begin{align*}
  \frac{|\omega^{m-1}\setminus \omega^m|}{|\omega^{m-1}|}
  &\ge
    \frac1{D^2}\cdot
    \frac{|M(c_m,\rho_0+1)|}{|M(c_m,\rho_0)^+|}.
\end{align*}
Otherwise, either $p_m<\rho_0$ and then $c_m\in\cS$; or
$p_m>\rho_0$ and we similarly obtain
\begin{align*}
  \frac{|\omega^{m-1}\setminus \omega^m|}{|\omega^{m-1}|}
  &\ge
    \frac1{D^2}\cdot
    \frac{|M(c_m,p)|}{|M(c_m,p_m)^+|}
    >\frac1{K_1D^2},
\end{align*}
where $M(c_m,p)$ is one of the neighbors of $M(c_m,p_m)$ in
$\cP_0$ in the one-sided neighborhood of $c_m$.
So in all cases
$  \frac{|\omega^{m-1}\setminus \omega^m|}{|\omega^{m-1}|}
    \ge
    \frac{\xi(\rho_0)}{D^2},
$
where
$ \xi(\rho_0) = \frac1{K_0K_1}\inf_{c\in\cD\setminus\cS}
\frac{(\rho_0+1)^{-1-\epsilon_1\beta_1}}
{2\delta_c-(\rho_0+2)^{-1/(\epsilon_1\beta_1)}}$.

This proves \eqref{eq:comparecVesc} since
\begin{align*}
    |\cV_m|
  &=
  \sum_{\omega^m\in\cV_m}|\omega^m|
  =
  \sum_{\omega^m\in\cV_m\atop \omega^{m}\subset\omega^{m-1}\in\cV_{m-1}}
  \frac{\big|\omega^{m-1}\setminus(\omega^{m-1}\setminus\omega^{m})\big|}
  {|\omega^{m-1}|}|\omega^{m-1}|
  \\
  &\le
  \left(1-\frac{\xi}{D^2}\right)
    \sum_{\omega^m\subset\omega^{m-1}\in\cV_{m-1}}|\omega^{m-1}|
    \le
    e^{-\zeta_0}|\cV_{m-1}|, \text{   for  }
    0<\zeta_0<\zeta=-\log(1-\xi/D^2).
\end{align*}

\begin{remark}
  \label{rmk:zeta0}
  \begin{enumerate}
  \item Note that
    $\zeta=\zeta(\rho_0)\xrightarrow[\rho_0\nearrow\infty]{}0$. Moreover,
    \emph{the estimates in this subsection depend on the
      requirement in the refining algorithm of only
      splitting the intervals when they intersect more than
      three elements of $\cP_0$}.
  \item Since
    $1\ge|f^{t_{m}}\omega^{m-1}|
    \ge\sigma^{t_m-t_{m-1}}|f^{t_{m-1}}\omega^{m-1}| \ge
    \sigma^{t_m-t_{m-1}}|M(c_{m-1},p_{m-1})|$ we deduce
    $t_m-t_{m-1}\le-\log|M(c_{m-1},p_{m-1})|/\log\sigma$. Hence
    \emph{the number of iterates between consecutive return times is
      bounded by the depth of the first return.}
    
  \item In particular, if $t_{m-1}$ is an escape time,
    i.e. $p_{m-1}=\rho_0-1$ and $c_{m-1}\in\cD\setminus\cS$,
    then $t_m-t_{m-1}\le L_0$
    for a uniform constant $L_0$ depending only on the
    minimum length of $M(c,\rho_0-1)$ for $c\in\cD\setminus\cS$. Because
    $|M(c,\rho_0-1)|$ grows to $\delta_c$ as
    $\rho_0\nearrow\infty$, then \emph{this shows that the
      number of iterates between consecutive escape times is
      uniformly bounded by a constant which does not depend
      on $\rho_0$.}
  \end{enumerate}
\end{remark}

\subsubsection{Between returns which are not escape times}
\label{sec:after-deep-returns}

We fix $m\in\{1,\dots,u\}$.  On the one hand, note that if
$p_{m-1}\ge\rho_0$, then by the partition algorithm, either
$c=c_{m-1}\in\cS$ and we get the bound
\begin{align}
  \frac{\big| \omega^{m}\big|}{\big|\omega^{m-1}\big|}
  &\le
    \frac{\big|
    \omega^{m}\big|}{\big|\wh{\omega}^{m-1}\big|}
    \le
    D^2\cdot \frac{\Big|f^{t_m}\omega^{m}\Big|}
    {\Big| f^{t_m}\wh{\omega}^{m-1}  \Big|},
    \text{  where  }
    \wh{\omega}^{m-1}=\omega^{m-1}\cap f^{-t_{m}}(\cU_0 )\nonumber
  \\
  &\le\label{eq:elementquotient}
    D^2\cdot \frac{\big| M(c_m,p_m)^+ \big|}
    {\sigma^{t_{m}-t_{m-1}-1}|f'(\xi)|\cdot \big|
    f^{t_{m-1}}\wh{\omega}^{m-1}\big|}
    \text{  for some  }\xi\in f^{t_{m-1}}\wh{\omega}^{m-1};
\end{align}
or $c=c_{m-1}\in\cD\setminus\cS$ and we use~\eqref{eq:nondegexp}
\begin{align*}
  \Big| f^{t_m}\wh{\omega}^{m-1}  \Big|
  \ge
  \sigma^{t_{m}-t_{m-1}-T(c)-1}|f'(\zeta)|\cdot \big|
  f^{t_{m-1}+T(c)}\wh{\omega}^{m-1}\big|
  \ge
  \sigma^{t_{m}-t_{m-1}-1}|f'(\zeta)|\cdot \big|
  f^{t_{m-1}}\wh{\omega}^{m-1}\big|
\end{align*}
for some $\zeta\in f^{t_{m-1}+T(c)}\wh{\omega}^{m-1}$; where
$\wh{\omega}^{m-1}$ restricts focus in the denominator to
points which lie in the \emph{return region}
$\cU_0=\cup_{c\in\cD}\cup_{p\ge\rho_0-1}I(c,p)$ whose
trajectories are controlled by the refinement algorithm.

If $c\in\cS$, by Lemma~\ref{le:an}
and the choice of $\wh{\omega}^{m-1}$
\begin{align}\label{eq:expoente}
  |f'(\xi)|
  \ge
  \frac{\left(p_{m-1}^\gamma\right)^{\beta_1-1}}B
  =
  \frac{p_{m-1}^{(1-\beta_1)/\epsilon_1\beta_1}}B
  >
  \frac{K_0^{(\beta_1-1)/(1+\epsilon_1\beta_1)}}B
  |M(c,p_{m-1})|^{(\beta_1-1)/(1+\epsilon_1\beta_1)}
\end{align}
where
$M(c,p_{m-1})\subset f^{t_{m-1}}\wh{\omega}^{m-1}\subset
f^{t_{m-1}}\omega^{m-1}$ by definition of return
time. Otherwise, $c\in\cD\setminus\cS$ and by the previous
estimates together with \eqref{eq:M+}, noting that
$(\beta_1-1)/(1+\epsilon_1\beta_1)=\beta_2-1$
\begin{align}
  |f'(\zeta)|
  &\ge
  \frac{\left(p_{m-1}^\gamma\right)^{\beta_1-1}}B
    \ge\nonumber
    \frac{K_0^{\beta_2-1}}B
    |M(\tilde c,p_{m-1})|^{\beta_2-1}
  \\
  &=\label{eq:bndconection}
    \frac{K_0^{\beta_2-1}}B
    |f^TM(c,p_{m-1})|^{\beta_2-1}
  \ge
    \frac{(K_0K_1)^{\beta_2-1}}B
    |M(c,p_{m-1})|^{\beta_2-1}
\end{align}
where $\tilde c=f^Tc$ is the singularity to which $c$ is
connected.

On the other hand, if $p_{m-1}=\rho_0-1$ and
$c\in\cD\setminus\cS$, then we use the conditions from
Subsection~\ref{sec:full-set-conditions} to deduce
\begin{align}\label{eq:expansesc}
  |f^{t_{m-1}+1}\omega^{m-1}|
  \ge
  \sigma|M(c,\rho_0-1)|
  \ge
  |M(c,\rho_0-1))|^{\un{\beta}}.
\end{align}
Since $K_1>K_0>1$, in any of the cases above we arrive at
\begin{align}\label{eq:bound-m-1}
  \frac{\big| \omega^{m}\big|}{\big|\omega^{m-1}\big|}
  &\le
    \frac{BD^2(K_1K_0)^{\beta_2-1}}
    {\sigma^{t_m-t_{m-1}-1}}\cdot
    \frac{\big| M(c_m,p_m)^+ \big|}
    { |M(c_{m-1},p_{m-1})|^{\max\{\beta_2,\un{\beta}\}}}.
\end{align}
Now if $t_{m}$ is such that $p_m>\rho_0$, then
from~\eqref{eq:escapeint0}
we can bound~\eqref{eq:bound-m-1} as
\begin{align*}
  \frac{L|M(c_m,p_m)^+|}{\kappa_0\sigma^{\Delta t_m-1}|M(c_m,p_m)|^{\beta_3}}
  \cdot
  \frac{|M(c_m,p_m)|^{\beta_3}}
  {|M(c_{m-1},p_{m-1})|^{\beta_2}}
  \le
  \frac1{L^2\kappa_0\sigma^{\Delta t_m-1}}\cdot
  \frac{|M(c_m,p_m)|^{\beta_3}}
  {|M(c_{m-1},p_{m-1})|^{\beta_2}},
\end{align*}
where $\Delta t_m=t_m-t_{m-1}$.
Now we can obtain \eqref{eq:compareV}
\begin{align*}
  |\cV_m|
  &=
  \sum_{\omega^m\in\cV_m}|\omega^m|
  =
  \sum_{\omega^m\in\cV_m\atop \omega^{m}\subset\omega^{m-1}\in\cV_{m-1}}
  \frac{|\omega^{m}|}{|\omega^{m-1}|}
  |\omega^{m-1}|
    \le
    \frac1{L^2\kappa_0\sigma^{\Delta t_m-1}}
    \cdot
    \frac{|M(c_m,p_m)|^{\beta_3}}
    { |M(c_{m-1},p_{m-1})|^{\beta_2}}
  |\cV_{m-1}|.
\end{align*}

This completes the proof of all but the last estimate in
Lemma~\ref{le:sizedepth}.

\subsubsection{Probability of a given sequence of returns}
\label{sec:probab-sequence-retu}

Now we apply the estimates~\eqref{eq:comparecVesc}
and~\eqref{eq:compareV} 
to prove~\eqref{eq:Vgeneral}.  Consider
$s_0=0<1\le r_1<s_1< r_2<s_2< \dots< r_h<s_h\le
u<r_{h+1}=u+1$ the indexes marking the beginning of
sequences of return times with return depths $p\le\rho_0$
from $t_{s_i}$ to $t_{r_{i+1}-1}$ within
$t_1<t_2<\dots<t_u$; and sequences of consecutive return
times with depth $p>\rho$ from $t_{s_i}$ to $t_{r_{i+1}-1}$.

More precisely, for each $j=0,\dots,h$ we have for
$r_j\le i <s_j$ that $p_i>\rho_0$; and for 
$s_j\le i<r_{j+1}$ we have $p_i\le\rho_0$.  Using this
grouping of the consecutive return times we obtain that
\begin{itemize}
\item every pair of return times with depth $p\le\rho_0$
  introduces a factor $e^{-\zeta}$ as in
  \eqref{eq:comparecVesc};
\item the remaining pairs of return times introduce
  quotients as in~\eqref{eq:compareV}.
\end{itemize}
Now the details: assuming first that $r_1=1$ and $s_h=u$,
then
$|\cV_u|=|\cV_0|\prod_{i=1}^u\frac{|\cV_i|}{|\cV_{i-1}|}$
and we can group the iterates from $s_0=0=r_1-1$ to $s_1-1$, the
transition to $s_1$ together with the iterates from $s_1$ to
$r_2-1$ 
\begin{align*}
  |\cV_{r_2-1}|
  &=
  |\cV_0|
  \prod_{i=r_1}^{r_2-1}\frac{|\cV^i|}{|\cV^{i-1}|}
  \le
    e^{-\zeta(r_2-s_1)}|\omega^0|
  |M(c_{s_1-1},p_{s_1-1})|^{\beta_3}
  \prod_{j=r_1}^{s_1-2}
  \frac{|M(c_j,p_j)|^{\beta_3}}
  {L^2\kappa_0\sigma^{\Delta t_j-1}|M(c_{j-1},p_{j-1})|^{\beta_2}}
  \\
  &\le
    e^{-\zeta(r_2-s_1)}|\omega_0|^{1-\beta_2}
    \frac{|M(c_{s_1-1},p_{s_1-1})|^{\beta_3}}{(L^2\kappa_0)^{s_1-r_1-1}}
    \prod_{j=1}^{s_1-2}
    \frac{|M(c_j,p_j)|^{\beta_3-\beta_2}}
  {\sigma^{\Delta t_j-1}}.
\end{align*}
Repeating for the groups of iterates from $s_j$ to
$s_{j+1}-1$ and including the transition from $s_j-1$ to
$s_j$, for $j=1,\dots, h$, we obtain, writing
$s=\sum_{j=1}^h (r_j-s_{j-1})$ for the number returns in a
chain of return times with $p_i\le\rho_0$, and
using~\eqref{eq:escapeint1}
\begin{align*}
  |\cV_{u-1}|
  &\le
    \frac{e^{-\zeta s}}{|\omega^0|^{\beta_2-1}}
    \prod_{j=1}^h
    \frac{|M(c_{s_h-1},p_{s_h-1})|^{\beta_3}}
    {(L^2\kappa_0)^{s_j-r_j-1}
    |M(c_{r_{j+1}-1},\rho_0-1)|^{\beta_2}}
    \prod_{r_j\le i<s_j-1}\hspace{-0.3cm}
    \frac{|M(c_i,p_i)|^{\beta_3-\beta_2}}
    {\sigma^{\Delta t_i-1}}
  \\
  &\le
    e^{-\zeta s}|\omega^0|^{1-\beta_2}
    \prod_{j=1}^h
    \frac1{(L^2\kappa_0)^{s_j-r_j-1}}
    \prod_{r_j\le i<s_j}
    \frac{|M(c_i,p_i)|^{\beta_3-\beta_2}}
    {\sigma^{\Delta t_i-1}}.
\end{align*}
Now since $s_h=u$, by~\eqref{eq:bound-m-1} the definition of
$\kappa_0$ and the choice of $\rho_0$, we get
\begin{align}
  |\cV_u|
  &=\nonumber
    \frac{|\cV_{u}|}{|\cV_{u-1}|}|\cV_{u-1}|
    \le
        \frac{L}
    {\kappa_0\sigma^{\Delta t_u}}\cdot
    \frac{|M(c_u,\rho_0-1)^+|}
    { |M(c_{u-1},p_{u-1})|^{\beta_2}}\cdot|\cV_{u-1}|
  \\
  &\le\label{eq:V1}
    |M(c_u,\rho_0-1)|^{\beta_2}
    \frac{e^{-\zeta s}|\omega^0|^{1-\beta_2}}{\sigma^{\Delta t_u}}
    \prod_{j=1}^h
    \prod_{r_j\le i<s_j-1}
    \frac{|M(c_i,p_i)|^{\beta_3-\beta_2}}
    {\sigma^{\Delta t_i-1}}.
\end{align}
Moreover, because $s_h=u$, then
$f^{t_u}\omega^u\supset M(c_u,\rho_0-1)$
and $f^{t_u+j}\omega^u$ has no returns for
$j=1,\dots,n-t_u$. Thus by \eqref{eq:expansesc}
\begin{align*}
      1&\ge|f^n\omega^u|=|f^{n-t_u-1}(f^{t_u+1}\omega^u)|
     \ge\sigma^{n-t_u-1}|f^{t_u}\omega^u|^{\un{\beta}}
\end{align*}
and then
$ |M(c_u,\rho_0-1)|^{\beta_2} \le
|f^{t_u}\omega^u|^{\beta_2} \le
|f^{t_u}\omega^u|^{\un{\beta}} \le \sigma^{t_u-n+1}$.
Writing
\begin{itemize}
\item $R=\sum_{j=1}^{h}(t_{s_j+1}-t_{r_j})$ for the quantity
  of iterations between return times with depth $p>\rho_0$
  from the $r_j$-th to $(s_j+1)$-th return, and
\item $r=\sum_{j=1}^h(s_j-r_j)$ for the quantity of return
  times involved in this chain of returns;
\end{itemize}
we arrive at
\begin{align}
  \label{eq:V2}
  \big|\cV_u\big|
  &\le
    |\omega^0|^{1-\beta_2} e^{-\zeta s}\sigma^{t_u-n+1-R+r}
    \prod_{p_i>\rho_0}    |M(c_i,p_i)|^{\beta_3-\beta_2}.
\end{align}
Now we use the last expression to obtain the other cases: we
might have $1<r_1$ or $s_h<u$, and so there is an initial
and/or final sequence of consecutive escape times. Since
\begin{align*}
  |\cV_u|&=|\cV_0|
           \prod_{i=1}^{r_1-1}\frac{|\cV_i|}{|\cV_{i-1}|}
           \prod_{i=r_1}^{s_h}\frac{|\cV_i|}{|\cV_{i-1}|}
           \prod_{i=s_h+1}^{u}\frac{|\cV_i|}{|\cV_{i-1}|}
\end{align*}
we can bound the middle product by the
expression~\eqref{eq:V1} with $M(c_{r_1-1},\rho_0-1)$ in the
place of $\omega^0$ and $s_h$ in the place of $u$:
\begin{align*}
  \prod_{i=r_1}^{s_h}\frac{|\cV_i|}{|\cV_{i-1}|}
  &\le
    \frac{e^{-\zeta s}L|M(c_{s_h},\rho_0-1)^+|}
    {\kappa_0\sigma^{\Delta t_{s_h}}|M(c_{r_1-1},\rho_0-1)|^{\beta_2}}
    \prod_{j=1}^h
    \prod_{r_j\le i<s_j}
    \frac{|M(c_i,p_i)|^{\beta_3-\beta_2}}
    {\sigma^{\Delta t_i-1}}
  \\
  &\le
    e^{-\zeta s}  \sigma^{-R+r}
    \prod_{p_i>\rho_0}  |M(c_i,p_i)|^{\beta_3-\beta_2},
\end{align*}
where $s=\sum_{j=2}^h (r_j-s_{j-1})$.  Finally, the first
factor is bounded according to
Subsection~\ref{sec:between-escapes}
\begin{align*}
  |\cV_0|\prod_{i=1}^{r_1-1}\frac{|\cV_i|}{|\cV_{i-1}|}
  \le
  e^{-\zeta(r_1-2)}|\cV_1|
  \le
  e^{-\zeta(r_1-2)}|\omega^0|
  =
  e^{-\zeta(r_1-s_0-2)}|\omega^0|
\end{align*}
and likewise the last factor is bounded by
$e^{-\zeta(u-s_h-1)}=e^{-\zeta(r_{h+1}-s_h-2)}$.  Hence
writing $\bar s=s+r_{h+1}-s_h-1 + r_1-s_0-1$ to incorporate
all the return times inside all chains of returns with depth
at most $\rho_0$, we arrive at
\begin{align*}
  |\cV_u|\le|\omega^0|^{1-\beta_2}
  e^{-\zeta\bar s}\sigma^{-R+r}
  \prod_{p_i>\rho_0}
  |M(c_i,p_i)|^{\beta_3-\beta_2}.
\end{align*}
In addition, by Remark~\ref{rmk:zeta0}(3) the total number
of iterates between consecutive escape times is related to
the number of escapes times by
$
  \bar s L_0 \ge \sum_{j=0}^{h}(t_{r_{j+1}}-t_{s_{j}+1}) =
  n-R,
$
so $e^{-\zeta \bar s}\le e^{-\frac{\zeta}{L_0}(n-R)}$. 
Hence
$
  |\cV_u|
  \le
  |\omega^0|^{1-\beta_2}
  e^{-\frac{\zeta}{L_0}(n-R)}
  \sigma^{-R+r}
  \prod_{p_i>\rho_0}
  |M(c_i,p_i)|^{\beta_3-\beta_2}.
$
This provides the general bound \eqref{eq:Vgeneral} after
setting $\zeta_0=\zeta/L_0$ and completes the proof of
Lemma~\ref{le:sizedepth}.


\subsection{Distance to the singular set between returns}
\label{sec:distance-singul-set}

We show that the distance to the singular set for iterates
$t_i<j<t_{i+1}$ between return times of a given
$\omega\in\cP_n$ is controlled by the depth of the last
return time, as follows.

In the same setting since the beginning of this section, by
the refinement algorithm we can write
$M(c_i,p_i)\subset f^{t_i}\omega^i\subset M(c_i,p_i)^+$,
where either $c_i\in\cS$ and $p_i\ge\rho(c)$, or
$c_i\in\cD\setminus\cS$ and $p_i\ge\rho_0-1$, since $t_i$
is a return.

We also have that there are $(c^j,p^j)$ so that
$f^{j}\omega^i\subset f^jM(c_i,p_i)^+\subset M(c^j,p^j)^+$,
where $p^j\ge\rho(c^j)$ for $c^j\in\cS$ or $p^j\ge\rho_0-1$
for $c^j\in\cD\setminus\cS$; and
$f^j\mid\omega^i:\omega^i\to f^j\omega^i$ is a
diffeomorphism, for each $j=t_i+1,\dots,t_{i+1}-1$. These
are the \emph{host intervals} at iterate $t_i<j<t_{i+1}$.
\begin{remark}\label{rmk:3split}
  These intervals are well-defined since the refinement
  algorithm only splits $f^j\omega^i$ when this interval
  intersects more that $3$ intervals of $\cP_0$.
\end{remark}
Moreover, we can compare the distance to the singular set
with the length of the host interval
following~\eqref{eq:distD_length} and noting that
$d(M(c,\rho_0-1)^+,\cD)=d(M(c,\rho(c)),\cD)$. In addition,
from~\eqref{eq:conexao}, we assume that
\begin{align}\label{eq:trunca}
2\delta<\inf\{\bar\delta,\dist(M(c,\rho_0-1)^+,\cD):
c\in\cD\setminus\cS\}
\end{align}
in what follows.  Next we divide the iterates
$j=t_i,\dots,t_{i+1}-1$ into sequences of consecutive visits
near a singularity and connection iterates between a
discontinuity and a singularity.

If $c_i\in\cD\setminus\cS$ with $p_i>\rho_0$, then
$\dist_\delta(f^j\omega^i,\cD)=1$ for
$t_i<j<t_i+ T, T=T(c_i)$ by the choice of $\delta$.  By
construction of the partition
$f^{t_i+T}(\omega^i)\supset M(\tilde c_i,p)$ for some
$\tilde c_i\in\cS$ and so
$|f^{t_i+T}(\omega^i)|\ge\sigma^T
|f^{t_i}\omega^i|>|M(c_i,p_i)|$ which implies
from~\eqref{eq:distD_length} that
$ \sum_{j=t_i}^{t_i+T-1} \log \dist_\delta(f^j\omega^i,\cD)
$ equals
\begin{align}\label{eq:double}
  \log \dist_\delta(f^{t_i}\omega^i,\cD)
  \ge
  \log\frac{|M(c_i,p_i)|^{\frac1{1+\epsilon_1\beta_1}}}{K_2}
  \ge
  \log|M(c_i,p_i)|.
\end{align}
\begin{remark}\label{rmk:trunca}
  We can assume without loss of generality that $p_i>\rho_0$
  for $c_i\in\cD\setminus\cS$. Otherwise we have
  $M(c_i,p_i)\subset f^{t_i}\omega^i$ and consequently
  $f^jM(c_i,p_i)\subset f^{t_i+j}\omega^i\subset
  M(c^j,p^j)^+$ with $p_i=\rho_0$ or $p_i=\rho_0-1$, which
  implies that $|M(c^j,p^j)^+|\ge\sigma^j|M(c_i,p_i)|$
  and so $p^j\le\rho_0$\footnote{All
    $M(c,p),p>\rho(c),c\in\cD$ are smaller than any
    $M(c',\rho-1), c'\in\cD\setminus\cS$.}  for all
  $0<j<t_{i+1}$. Hence $f^j\omega^i$ is $2\delta$-away from
  $\cD$ by the choice of $\delta$, so all the iterates
  $t_i\le j<t_{i+1}$ are truncated:
  $\sum_{j=t_i}^{t_{i+1}-1}-\log
  \dist_\delta(f^j\omega^i,\cD)=0$.
\end{remark}
By the Mean Value Theorem and \eqref{eq:nondegexp} we get
\begin{align*}
  \frac{|f^{j+1}\omega^i|}{|f^j\omega^i|}=|f'(\xi_j)|
  \ge
  \frac{d(\xi_j,c^j)^{\beta_1-1}}B, \quad\text{for some}\quad
  \xi_j\in f^j\omega^i \quad\text{if $c^j\in\cS$}.
\end{align*}
If $(c^j,p^j)$ has $c^j\in\cS$, then
from~\eqref{eq:atomsizesing}, ~\eqref{eq:distD_length}
and~\eqref{eq:escapeint1} the quotient $
\frac{|f^{j+1}\omega^i|}{|f^j\omega^i|}$ is bigger than
\begin{align}\label{eq:cases_p}
  \frac{d(M(c^j,p^j)^+,c^j)^{\beta_1-1}}{B}
    \ge
  \frac{|M(c^j,p^j+1)|^{\frac{\beta_1-1}{1+\epsilon_1\beta_1}}}{BK_2^{1-\beta_1}}
  \ge
    \begin{cases}
      \frac{|M(c^j,p^j)|^{\beta_3-1}}{BK_2^{1-\beta_1}}
      \text{  if  } \rho(c)\le p\le\rho_0
      \\
      |M(c^j,p^j)|^{\beta_3-1} \text{  if  }p>\rho_0
    \end{cases}
.
\end{align}
In particular, for $j=t_i+T(a_{k_i})$ we necessarily have
$c^j=\tilde c_i$ and $p^j=p_i$ by the construction of the
refinement, then
\begin{align*}
  |f^{t_i+T+1}\omega^i|\ge
  \frac{|M(\tilde  c_i,p_i+1)|^{\beta_2-1}}
  {BK_2^{1-\beta_1}} |f^{t_i+T}\omega^i|
  \ge
  \frac{|M(c_i,p_i+1)|^{\beta_2-1}}
  {BK_2^{1-\beta_1}K_1^{1-\beta_2}} |M(\tilde c_i,p_i)|
  >
  |M(c_i,p_i)|^{\beta_3}.
\end{align*}
Finally, if $c^j\in\cD\setminus\cS$ and $p^j>\rho_0$ for
$j>t_i+T$, then from~\eqref{eq:conexao} there are no returns
during times $j+1,\dots, j+T$ where $T=T(c^j)$, hence
$  |f^{j+T+1}\omega^i|
  \ge
  \sigma^T|f^j\omega^i|
$.
Therefore, denoting by
$t_i<\ell_1<\dots<\ell_s<t_{i+1}$ the free iterates that
have an host interval near a discontinuity and
$T_m=T(c^{\ell_m}), m=1,\dots, s$, we obtain
the relation
\begin{align*}
1\ge
  &|f^{t_{i+1}}\omega_i|
  =
    |f^{t_i+T}\omega_i|
    \prod_{j=t_i+T}^{\ell_1-1}
    \frac{|f^{j+1}\omega_i|}{|f^j\omega_i|}
   \left( \prod_{m=1}^s\prod_{j=\ell_m}^{\ell_{m+1}-1}
    \frac{|f^{j+1}\omega_i|}{|f^j\omega_i|}\right)
    \prod_{j=\ell_s+T_s}^{t_{i+1}-1}
    \frac{|f^{j+1}\omega_i|}{|f^j\omega_i|}
  \\
  &\ge
    |M(c_i,p_i)|^{\beta_3}
    \prod\{|M(c^j,p^j)|^{\beta_3-1}:t_i+T\le j <t_{i+1},
    c^j\in\cS
    \text{ or } j=\ell_m, 1\le m\le s\}.
\end{align*}
Now using~\eqref{eq:distD_length},~\eqref{eq:escapeint1}
and~\eqref{eq:double} we get that
$\sum_{j=t_i}^{t_{i+1}-1} \log\dist_\delta(f^j\omega_i,\cD)$
is bounded from below by
\begin{align*}
  \log&\left(|M(c_i,p_i)| \prod\{|M(c^j,p^j)|
    : t_i+T\le j <t_{i+1}, c^j\in\cS
        \text{ or } j=\ell_m, 1\le m\le s\}\right)
  \\
  &\ge
    \log|M(c_i,p_i)|+\frac{\beta_3}{1-\beta_3}\log|M(c_i,p_i)|
    =\frac1{1-\beta_3}\log|M(c_i,p_i)|,
\end{align*}
that is
\begin{align}
  \label{eq:slowapprox0}
  \sum_{j=t_i}^{t_{i+1}-1} -\log\dist_\delta(f^j\omega_i,\cD)
  \le
  \frac1{\beta_3-1}\log|M(c_i,p_i)|.
\end{align}
Finally, if $c_i\in\cS$, then we need not use the initial
sequence of $T$ iterates, so we obtain,
using~\eqref{eq:cases_p} perhaps with $p\le\rho_0$, a
similar inequality to~\eqref{eq:slowapprox0} with a
different constant.

\subsubsection{Upper bound for the recurrence frequency}
\label{sec:upper-bound-recurr}

Consequently, we obtain the following upper bound for the
recurrence frequency to $\cD$ of the orbits of points of
$\omega\in\cP_n$ with given $n\ge1$, return times
$0=t_0<t_1<\dots<t_u\le n$ and respective depths
$(c_1,p_1),\dots,(c_u,p_u)$
\begin{align}
  \label{eq:slowapproxdepth}
  S_n\Delta_\delta(\omega)
  =
  \sum_{j=0}^{n-1}\Delta_\delta(f^j\omega,\cD)
  &\le
    -C_1\sum_{j=0} ^{n-1}
    \sum_{p_i>\rho_0}\log|M(c_i,p_i)|
\end{align}
for all $\delta$ satisfying \eqref{eq:trunca}, where $C_1>0$
is a constant not depending on $\omega$ or $n$ or on the
return depths, not even on $\rho_0$ after this threshold is
fixed as in Subsection~\ref{sec:full-set-conditions}.

\subsection{Expected value of return depths and
  exponentially slow recurrence}
\label{sec:expect-value-splitt}

Here we complete the proof of
Theorem~\ref{mthm:expslowapprox}. We estimate the expected
value of deep returns up to $n$ iterates of the dynamics,
proving the following statement, where we denote
\begin{align*}
  \cD_n^\delta(x)=
  -\sum_{j=0} ^{n}
    \sum_{i=1,\dots,u\atop  p_i>\Theta}\log|M(c_i,p_i)|
\end{align*}
for $\Theta\in\ZZ^+$ big enough (to be defined in the
statement of Lemma~\ref{le:expected} below);
$\delta=\delta(\Theta)>0$ is such  that
\begin{align*}
  p>\Theta \iff d(M(c,p),\cD)<\delta, c\in\cD 
  \quad\text{hence}\quad
  \delta=\Theta^{-1/\epsilon_1\beta_1}=a_\Theta,
\end{align*}
and $x\in\omega\in\cP_n$ with
$u=u_n(\omega)\le n$ and the return depths
$(c_1,p_1),\dots,(c_u,p_u)$.

To state the result precisely, let
$S(\rho,\xi)=(\#\cD)\cdot \sum_{p\ge\rho}|M(c_0,p)|^\xi$
for any fixed\footnote{Recall
  Remark~\ref{rmk:escape}: the length of $M(c,p)$ does not
  depend on $c$ for $p\ge\rho_0$.} $c_0\in\cD$  and
$\rho\ge\rho_0$. Note that\footnote{See the definition of
  $C_0$ after Lemma~\ref{le:sizedepth}: the restriction on
  $\xi$ ensures $S(\rho,\xi)\approx\sum_{p>\rho}p^{-\theta}$
  with $\theta>1$.} this is well-defined for all
$\xi>\beta_2+\frac{1+\epsilon_1\beta_1}{\epsilon_1\beta_1}$.

\begin{lemma}
  \label{le:expected}
  Let $\epsilon_0>0$ be given and $z>0$ be such that
  $z<(1+\epsilon_1\beta_1)^{-1}$ and
  $\beta_3-2\beta_2-z>\frac{1+\epsilon_1\beta_1}{\epsilon_1\beta_1}$.
  Then for a big enough $\rho_0\in\ZZ^+$ satisfying all the
  conditions in Subsection~\ref{sec:full-set-conditions}
  together with $S(\rho_0,\beta_3-\beta_2)<\epsilon_0/2$,
  and for $\Theta>\rho_0$ such that
  $S(\Theta,\beta_3-\beta_2-z)<\epsilon_0/2$, there exists a
  constant $K>0$ such that
  $\int_{[\cD_n^\delta(x)\ge n\epsilon/C_1]} e^{z
    \cD_n^\delta(x)} \, dx \le K e^{\epsilon_0 n}$ for all
  $n>\epsilon_0^{-1}|\log|M(c_0,\rho_0)||$.
\end{lemma}

This result provides the exponential slow recurrence bound
\eqref{eq:exptail} as follows.

\subsubsection{Measure of bad recurrence}
\label{sec:measure-points-with}

This is now a direct consequence of Chebyshev's inequality
together with \eqref{eq:slowapproxdepth} and
Lemma~\ref{le:expected}: given $\epsilon>0$ we have
\begin{align*}
  \{x\in M: S_n\Delta_\delta(x) \ge n\epsilon\}
  \subseteq
\{x:  \cD_n^\delta(x)\ge n\epsilon/C_1\}
  =
  [\cD_n^\delta\ge n\epsilon/C_1]
  =
  \{x:e^{z\cD_n^\delta(x)}\ge e^{n\epsilon/C_1}\}.
\end{align*}
Hence, we choose $0<\epsilon_0<\epsilon/C_1$ and find
$\Theta>\rho_0$ big enough so that 
\begin{align*}
 \lambda \{x\in M: S_n\Delta_\delta(x)
  \ge n\epsilon\}
  \le
  e^{-n\epsilon/C}\int_{[\cD_n^\delta\ge n\epsilon/C_1]} e^{z\cD_n^\delta}\,dx
  =
  K e^{-n\big(\epsilon/C_1-\epsilon_0\big)}
\end{align*}
which is exponentially small with rate
$\xi=\epsilon/C_1-\epsilon_0$ for all
$n>C_1\epsilon^{-1}|\log|M(c_0,\rho_0)||$.

This concludes the proof of
Theorem~\ref{mthm:expslowapprox}, except for the proof of
Lemma~\ref{le:expected}, which comprises the rest of this
section.

\subsubsection{Expected value of deep return depths}
\label{sec:probab-deep-returns}

Now we start the proof of Lemma~\ref{le:expected} and along
the way also prove Corollary~\ref{mcor:expmapsholes}.

We fix $\epsilon_0,z>0$ and $\Theta>\rho_0$ and also
$n$ as in the statement of the lemma.

Then we have two possibilities for $\omega\in\cP_n$ such
that $\cD_n^\delta(\omega)>n\epsilon_0$: either $\omega$ has
no deep return in the iterates from $1$ to $n-1$, or
$\omega$ admits some deep return time $t\in\{1,\dots,n-1\}$.

\subsubsection*{No deep returns and
  Corollary~\ref{mcor:expmapsholes}}

If there are no deep return times from $1$ to $n-1$, we have
$\cD_n^\delta(\omega)=-\log|M(c_0,p_0)|>n\epsilon_0$ where
$\omega^0=M(c_0,p_0)\supset\omega$ and
$|\omega^0|<e^{-n\epsilon_0}<1$. Hence the subset
\begin{align*}
  G_n(\delta)=\{x\in M: \cD_n^\delta(x)>n\epsilon_0
  \qand
  v(\omega)=0\}
\end{align*}
of all the points belonging to such intervals
satisfies
\begin{align*}
  G_n(\delta)\supseteq H_n(\delta)=\{x\in M: d(f^jx,\cD)>\delta,
  \forall 0<j<n\}.
\end{align*}

\begin{lemma}\label{le:nodeep}
  Given $0<\delta<a_{\rho_0+1}$ there exists $\tilde K>0$
  such that
  $|G_n(\delta)|<\tilde K e^{-\tilde \epsilon_0 n}$ for each
  $n>(1-\gamma)\log\rho_0+\log(\sigma K_0^{-1})$,
  where
  $\tilde\epsilon_0=\frac{1+2\epsilon_1\beta_1}{1+\epsilon_1\beta_1}$.
\end{lemma}

\begin{proof}
  We can choose $\Theta>\rho_0$ so that $\delta<a_\Theta$,
  set $z=\frac12(1+\epsilon_1\beta_1)^{-1}$ and
  $\epsilon_0=1$, and then calculate
\begin{align*}
  \int_{G_n(\delta)}e^{z\cD_n(x)}\,dx
  &=
    \sum_{\omega\in\cP_n\atop\omega\in
    G_n(\delta)}|\omega^0|^{-z}|\omega|
    <
    \sum_{|\omega_0|<e^{-n}}|\omega_0|^{1-z}
    <\hspace{-1cm}
    \sum_{p^{-(1+\epsilon_1\beta_1)/\epsilon_1\beta_1}<K_0K_1e^{-n}}
    \hspace{-1cm}\frac{K_0}{\sigma}
    p^{-\frac{1+\epsilon_1\beta_1}{\epsilon_1\beta_1}(1-z)},
\end{align*}
where we used $p>\rho_0$ by the choice of $n$
and~\eqref{eq:atomsizediscont}. It is not difficult to see
that the above series is bounded by
$\tilde
K\exp\big(-\frac12 n/(1+\epsilon_1\beta_1)\big)$
for some constant $\tilde K>0$ not depending on $n$. Thus
$ |G_n(\delta)| \le e^{-n}\int_{G_n(\delta)}e^{z\cD_n^\delta(x)}\,dx
\le
\tilde K
\exp\left(n\left(\frac{1+2\epsilon_1\beta_1}{1+\epsilon_1\beta_1}\right)\right)
=\tilde K e^{-\tilde\epsilon_0n} $ by Chebyshev's Inequality.
\end{proof}
This shows in particular that
$|H_n(\delta)|<\tilde K e^{-\tilde\epsilon_0n}$ for all
sufficiently big $n>1$, which proves
Corollary~\ref{mcor:expmapsholes}.

\subsubsection*{Deep returns and Lemma~\ref{le:expected}}

From now on we assume without loss of generality that
$0<\epsilon_0<1$, and also that $\omega$ admits some deep
return time $t\in\{1,\dots,n-1\}$.

Thus we consider $\omega\in\cP_n$ and $u\in\{1,\dots,n\}$, but
also $v\in\{1,\dots,u\}$, that is, there exists a sequence
of deep returns
for $\omega$ at indexes $1\le r_1 < \dots < r_v\le u$
between the return times $0=t_0<t_1<\dots<t_u\le n$.

More precisely, among the return depths
$\rho_1,\dots,\rho_u$ at times $t_1<\dots<t_u$, there are
exactly $v$ return times such that for $i=1,\dots,v$
\[
 t_{r_i} \mbox{  satisfies  }
 f^{t_{r_i}}(x) 
 \in M(c_{r_i},\rho_{r_i})
\mbox{ for all }
x\in\omega \mbox{ and  } \rho_{r_i} > \Theta, c_{r_i}\in\cD
\]
for $\Theta\in\ZZ^+$, $\Theta>\rho_0$; while
$\rho_j\le\Theta$ for
$j\in\{1,\dots,u\}\setminus\{r_1,\dots, r_v\}$.

We let $u_n(x), v_n(x)$ for $x\in\omega\in\cP_n$ be the
number of returns and deep returns of $\omega$ until the
$n$th iterate. Then we define
\begin{align*}
  A^{u,v}
  &=
    A^{u,v}_{\rho_1,\dots,\rho_v}(n)
  =
    \big\{
    x\in M: u_n(x)=u, \, v_n(x)=v\,\text{ and }
    p_{r_i}=\rho_i,\, i=1,\dots,v
    \big\}
\end{align*}
the set of points $x\in\omega\in\cP_n$ which in $n$ iterates
have $u$ returns and $v$ deep returns among these, with the
specified deep return depths $\rho_1,\dots,\rho_v$.

From \eqref{eq:probreturns} we know how to estimate the
probability of a certain sequence of returns and have seen
that the estimate does not depend on the particular order of
the returns nor on their return times, but only on their
depths.  Using \eqref{eq:probreturns} we estimate the
measure of $A^{u,v}$ considering all possible combinations
of the events $\cV_u$ which are included in $A^{u,v}$. Note
that for any given $v\le u$ there are $\binom{u}{v}$ ways of
having $v$ deep returns among $u$ return situations and
$\binom{n-1}{u}$ choices of the times of the return
situations. Hence the Lebesgue measure of
$A^{u,v}$
is bounded from above by 
\begin{align*}
 |A^{u,v}| 
  &\le
    \binom{n-1}{u}
    \sum_{\omega\in\cP_n\cap A^{u,v}\atop \omega\subset\omega^0\in\cP_0}
    \binom{u}{v}
  |\omega^0|^{1-\beta_2}
  e^{-\zeta_0(n-R)}
  \sigma^{-R+r}
  \prod_{p_i>\rho_0}
    |M(c_i,\rho_i)|^{\beta_3-\beta_2},
\end{align*}
where $c_i=c_i(\omega)$ is the element of
$\cD$ associated to the
$i$th return time, $r=r(\omega)=\#\{0<i\le u:
p_i(\omega)>\rho_0\}$ and
$R=R(\omega)$ is the total number of iterates between each
of the
$r$ return times counted above and the next return
time. Counting all the possible depths of the non-deep
return times involved in the last product we obtain
\begin{align}\label{eq:Auv}
  |A^{u,v}|
  &\le\hspace{-0.2cm}
    \sum_{\omega_0\in\cP_0} \hspace{-0.2cm}
    \binom{n}{u}\hspace{-0.6cm}
    \sum_{\substack{\omega_0\supset\omega\in\cP_n\\v_n(\omega)=v,
  u_n(\omega)=u}} \hspace{-0.4cm}
  \binom{u}{v}\hspace{-0.4cm}
  \sum_{\substack{\rho_0<p_j(\omega)\le\Theta\\c_j\in\cD, 1\le j\le u-v}}
  \prod_{j=1}^{u-v} |M(c_j,p_j)|^{\beta_3-\beta_2} \hspace{-0.2cm}
  \sum_{c_i\in\cD\atop1\le i\le v}
  \prod_{i=1}^v
  |M(c_i,\rho_i)|^{\beta_3-\beta_2}
  |\omega^0|^{1-\beta_2}\nonumber
  \\
  &\le
    C_0\binom{n}{u}
    \left(
    \#\cD\hspace{-.3cm}\sum_{\rho_0<p_j\le\Theta}|M(c_0,p)|^{\beta_3-\beta_2}
    \right)^{u-v}
    \binom{u}{v}
    (\#\cD)^v
    \prod_{i=1}^v
    |M(c_i,\rho_i)|^{\beta_3-\beta_2}\nonumber
  \\
  &\le
    C_0
    \binom{n}{u} \binom{u}{v} S(\rho_0,\beta_3-\beta_2)^{u-v}  (\#\cD)^v
    \prod_{i=1}^v
    |M(c_0,\rho_i)|^{\beta_3-\beta_2},
\end{align}
where $C_0$ was defined after Lemma~\ref{le:sizedepth}.
We can now complete the proof of Lemma~\ref{le:expected}.
By definition and using Lemma~\ref{le:nodeep}
\begin{align*}
  \int_{[\cD_n^\delta\ge n\epsilon/C_1]}\hspace{-1cm} e^{z \cD_n^\delta(x)} \, dx
  &<
    e^{-\tilde\epsilon_0n}+
    \sum_{\omega\in\cP_n \atop 0< v_n(\omega)\le
    u_n(\omega)<n}
    e^{z\cD_n^\delta(\omega)}\cdot|\omega|
\end{align*}
and from \eqref{eq:Auv} and the definition of $\cD_n^\delta$
we bound the last summation from above by
\begin{align*}
  C_0\sum_{u=1}^{n}
  &\binom{n}{u}
    \sum_{v=1}^u
    \binom{u}{v}
    S(\rho_0,\beta_3-\beta_2)^{u-v}
    \sum_{\rho_i>\Theta\atop1\le i\le v}
    (\#\cD)^v
    \prod_{i=1}^v
    |M(c_0,\rho_i)|^{\beta_3-\beta_2-z}
  \\
  &\le
    C_0\sum_{u=1}^{n}
    \binom{n}{u}
    \sum_{v=1}^u
    \binom{u}{v}
    S(\rho_0,\beta_3-\beta_2)^{u-v}
    \left(
    \#\cD
    \sum_{\rho>\Theta}
    |M(c_0,\rho)|^{\beta_3-\beta_2-z}
    \right)^v
  \\
  &=
    C_0\sum_{u=1}^{n}
    \binom{n}{u}
    \big(
    S(\rho_0,\beta_3-\beta_2)
    +
    S(\Theta,\beta_3-\beta_2-z)
    \big)^u
  \\
  &=
    C_0\big(
    1+    S(\rho_0,\beta_3-\beta_2)
    +
    S(\Theta,\beta_3-\beta_2-z)
    \big)^n
    =
    C_0e^{\upsilon n},
\end{align*}
where
$\upsilon=\log \big( 1+ S(\rho_0,\beta_3-\beta_2) +
S(\Theta,\beta_3-\beta_2-z) \big)$.

This value
$\upsilon>0$ can be made smaller than any given
$\epsilon_0>0$ by choosing
$\rho_0\in\ZZ^+$ big enough to satisfy all conditions in
subsection~\ref{sec:full-set-conditions} together with
$S(\rho_0,\beta_3-\beta_2)<\frac{\epsilon_0}{2}$, and then
choosing another integer
$\Theta>\rho_0$ so that
$S(\Theta,\beta_3-\beta_2-z)<\frac{\epsilon_0}{2}$, since
both $\beta_3-\beta_2$ and
$\beta_3-\beta_2-z$ satisfy the conditions for convergence
of the series.

Finally, since
$ e^{-\tilde\epsilon_0n}+C_0e^{\upsilon n} = C_0e^{\upsilon
  n}(1+C_0^{-1}e^{-(\tilde\epsilon_0+\upsilon)n}) \le K
e^{\upsilon n} $ for a constant $K=K(\epsilon_0)>0$ not
depending on $n$, we have completed the proof of
Lemma~\ref{le:expected} and, with it, the proof of
Theorem~\ref{mthm:expslowapprox}.


\def\cprime{$'$}

\bibliographystyle{abbrv} 

\bibliography{../bibliobase/bibliography}

\end{document}